\documentclass[letterpaper]{amsart}

\usepackage[letterpaper]{geometry}
\usepackage{amsmath, amsthm, amsfonts, amsbsy, thmtools, amssymb}

\usepackage{thm-restate}
\usepackage{hyperref} 
\usepackage[mathscr]{euscript}
\usepackage[T1]{fontenc}
\usepackage{tikz}
\usetikzlibrary{arrows}
\usepackage{cleveref}

\makeatletter
\def\namedlabel#1#2{\begingroup
	#2%
	\def\@currentlabel{#2}%
	\phantomsection\label{#1}\endgroup
}
\makeatother

\numberwithin{equation}{section}

\newtheorem{thm}{Theorem}[section]
\newtheorem{prop}[thm]{Proposition}
\newtheorem{lem}[thm]{Lemma}
\newtheorem{cor}[thm]{Corollary}

\theoremstyle{remark}
\newtheorem{rem}[thm]{Remark}
\theoremstyle{definition}
\newtheorem{definition}[thm]{Definition}
\newtheorem{example}[thm]{Example}

\title{Limit Shapes and Local Statistics for the Stochastic Six-Vertex Model}

\author{Amol Aggarwal}  

\begin{document}

\begin{abstract} 
	
	In this paper we consider the stochastic six-vertex model on a cylinder with arbitrary initial data. First, we show that it exhibits a limit shape in the thermodynamic limit, whose density profile is given by the entropy solution to an explicit, non-linear conservation law that was predicted by Gwa-Spohn in 1992 and by Reshetikhin-Sridhar in 2018. Then, we show that the local statistics of this model around any continuity point of its limit shape are given by an infinite-volume, translation-invariant Gibbs measure of the appropriate slope. 
\end{abstract} 

\maketitle 

\tableofcontents

\section{Introduction} 

\label{Model}

\subsection{Preface}

Over the past several decades, a substantial amount of effort has been directed towards the understanding of limit shapes and local statistics for statistical mechanical models. In this paper we analyze these two phenomena for the six-vertex model with stochastic weights.

In the context of dimer models, one of the earlier results along this direction concerned the \emph{limit shape phenomenon} exhibited by domino tilings, stating that the height function of a uniformly random tiling of a large domain is likely to concentrate (after suitable normalization) around a global limit. This was first established in the case when the domain is an Aztec diamond by Cohn-Elkies-Propp \cite{LSRT}, who provided an exact form for the limit shape. Later, this result was substantially generalized by Cohn-Kenyon-Propp \cite{VPT} to domino tilings of essentially arbitrary domains. The latter work expressed the limit shape through a variational principle, namely, as the maximizer of a certain (explicit) concave functional, which was later rewritten by Kenyon-Okounkov \cite{LSCE} as the solution to a complex wave equation that in many cases can be solved through the method of characteristics. 

A salient feature of these limit shapes is that they can be inhomogeneous for certain choices of the boundary, that is, the density of tiles can (asymptotically non-negligibly) differ in different regions of the rescaled domain $D$. Thus a question of interest is to understand how the boundary data affect the \emph{local statistics}, the joint law of nearly neighboring tiles, of a random tiling.

The physical prediction (see, for example, Conjecture 13.5 of \cite{VPT} and Section 1.5 of \cite{D}) in this context is that these local statistics should be determined by (the gradient of) the global limiting height function $H	: D \rightarrow \mathbb{R}$ for the tiling model. More specifically, around some point $(u, v) \in D$, they should be given by the unique \cite{RS} infinite-volume, translation-invariant, ergodic Gibbs measure with slope equal to $\nabla H (u, v)$; these infinite-volume Gibbs measures are explicit, as they can be expressed as a determinantal point process with a known kernel \cite{LSLD,DA}. For many choices of boundary data, this prediction has been shown to hold true \cite{DOA,ASD,URB,NPOE,CP,CI,LLSBHP,CPLGR,AR,DPD,DT}. 

The proofs of these results on limit shapes and local statistics were largely based on the determinantal (free-fermionic) structure underlying the tiling models \cite{SDL}. So, they were also applicable to some classes of more general dimer models \cite{ST,TPMVOP,DA}, where such structure persists. 

Yet, the solvability of the six-vertex model is of a substantially different nature and can be attributed to a one-parameter family of mutually commuting transfer operators, which can be diagonalized through the quantum inverse scattering method (algebraic Bethe ansatz); see the books of Baxter \cite{ESMSM} and Korepin-Bogoliubov-Izergin \cite{QISMCF}. This framework expresses the eigenvalues of these operators in terms of solutions to an intricate system of non-linear equations called the Bethe equations. These equations are known to simplify considerably in the five-vertex degeneration of the six-vertex model, in which one of the six weights of the model is set to $0$. Using this fact, de Gier-Kenyon-Watson \cite{LSAFM} recently established the variational principle for the limit shape of the five-vertex model with general boundary conditions. 

Based on earlier free energy predictions due to Lieb \cite{RESI} and Sutherland-Yang-Yang \cite{ESMTFAEEF} (which were in turn based on a heuristic analysis of the Bethe equations), variational principles have also been conjectured for the general six-vertex model under arbitrary boundary conditions \cite{BCSVM,SVMFBC,ILSSVM}. However, for most values of the six-vertex weights, it is unknown how to analyze the Bethe equations. Partly for this reason, the limit shape phenomenon had until now remained without understanding for any other type of six-vertex model (except under certain families of boundary data). Moreover, we are not aware of any previous results on the local statistics for a non-free-fermionic six-vertex model on a finite domain.

In this paper we consider the \emph{stochastic six-vertex model}, which was introduced by Gwa-Spohn \cite{SVMRSASH} in 1992 as an instance of the six-vertex model whose weights are stochastic and therefore give rise to a Markovian transfer matrix. Its weights are given on the left side of \Cref{vertexdomain}; we will describe the model in more detail in \Cref{StochasticModel} below. 

Over the past several years, this model has received considerable interest \cite{CSSVMP,SSVMQHL,SHSVMM,SSVM,STESVM,TFAEPSSVM,SLSVM,ELHSP} in the probability and mathematical physics communities, due to the fact that it is one of the simplest models that lies in the intersection of two-dimensional statistical mechanics and stochastic integrability; as such, it exhibits a number of interesting physical phenomena. For instance, under suitable limit degenerations, it converges to the asymmetric simple exclusion process (ASEP) \cite{CSSVMP,SSVM}; the Kardar-Parisi-Zhang (KPZ) equation \cite{SLSVM,ELHSP}; and a stochastic variant of the telegraph equation \cite{STESVM}. Furthermore, it is a member of the KPZ universality class \cite{SVMRSASH}, exhibiting fluctuations of order $N^{1 / 3}$ \cite{CFSAEPSSVMCL,PTAEPSSVM,SSVMQHL,SHSVMM,SSVM} and correlations on spatial scales of order $N^{2 / 3}$ \cite{TFAEPSSVM} on a domain of size $N$. 

The limit shape phenomenon for the stochastic six-vertex model, run under arbitrary initial (boundary) data, was predicted in the original physics work \cite{SVMRSASH}. In particular, it was predicted there that the limiting height function $H$ of this model should satisfy a Hamilton-Jacobi equation of the form $\partial_y H = \varphi (\partial_x H)$, for a suitable function $\varphi$ (given by \eqref{kappafunction} below). An alternative heuristic derivation of the same predicted limit shape, assuming the conjectural variational principle for the general six-vertex model, was provided more recently by Reshetikhin-Sridhar in \cite{LSSSVM}.  

Under certain classes of initial data, this prediction has been proven. Indeed, in \cite{SSVM}, Borodin-Corwin-Gorin address the limit shape (and fluctuations) for this model when run under domain-wall boundary data; this was later extended in \cite{PTAEPSSVM} to the case of (generalized) Bernoulli initial data. Although both of these results imply precise asymptotic statements about the stochastic six-vertex model in their respective settings, the algebraic analysis implemented in their proofs requires that the boundary data be of a special type, and so does not apply to the case of general initial data. 

As mentioned above, there seem to have been no previous results on local statistics for the stochastic six-vertex model when run under any type of (deterministic) initial data. Still, assuming the equation for the limit shape proposed in \cite{SVMRSASH,LSSSVM}, one can formulate a prediction for how they should behave in analogy with what was explained above in the case of tiling models. Specifically, around some point $(u, v)$ in our rescaled domain, the local statistics around $(u, v)$ should be given by an infinite-volume, translation-invariant, ergodic Gibbs measure for the stochastic six-vertex model whose slope is equal to $\nabla H (u, v)$. Denoting $\rho = G (u, v) = \partial_x H (u, v)$ and using the predicted identity $\partial_y H = \varphi (G)$, the slope of this Gibbs measure can be alternatively expressed as $\big( \rho, \varphi (\rho) \big)$. In particular, any such slope must lie on the curve $\big\{ (s, t)	: t = \varphi (s) \big\}$. In the physics literature, this curve is sometimes referred to as the \emph{conical singularity} (also known as the \emph{tricritical point} \cite{TECS,SVMFBC} or \emph{KPZ point} \cite{CSFODDGC}) of the ferroelectric six-vertex model, as it is where three phases of the six-vertex model are predicted to coexist and the free energy to be singular \cite{TCPFSVM,TECS,CPSVMCSFM}. 

In Appendix A.2 of \cite{CFSAEPSSVMCL}, an infinite-volume, translation-invariant Gibbs measure for the stochastic six-vertex model with slope $\big( \rho, \varphi (\rho) \big)$ was introduced for each slope on this conical singularity. We will denote this measure by $\mu (\rho)$ and will recall its precise definition in \Cref{Translation} below. The analog of the local statistics prediction for the stochastic six-vertex model then states that, around $(u, v)$, the local statistics should be given by this measure $\mu (\rho)$, where $\rho = \partial_x H (u, v)$. 

Although the measure $\mu (\rho)$ is not known to be given by a determinantal point process, \cite{CFSAEPSSVMCL} establishes both qualitative and quantitative properties about it, which are considerably different from those for infinite-volume, translation-invariant Gibbs measures for tiling models. For instance, it was shown by Kenyon that the local statistics for random tilings are conformally invariant \cite{CI} with Gaussian free field fluctuations \cite{F,DA}. In contrast, the six-vertex measures $\mu (\rho)$ are quite anisotropic, exhibiting Baik-Rains fluctuations of exponent $\frac{1}{3}$ along a single direction and Gaussian fluctuations of exponent $\frac{1}{2}$ elsewhere \cite{CFSAEPSSVMCL}. The latter point can again be viewed as a manifestation of the fact that the stochastic six-vertex model is in the KPZ universality class. 

In this paper we establish both the limit shape (see \Cref{sixvertexcylinderglobal} below) and local statistics (see \Cref{sixvertexlocalcylinder} below) predictions for the stochastic six-vertex model under arbitrary initial data. In particular, the latter provides the first local statistics result for a non-free-fermionic six-vertex model. 

Since the stochastic six-vertex model is not known to be determinantal, and since we remain unable to analyze the general Bethe equations, our methods will be different from the ones used previously to analyze dimer models and the five-vertex model. In particular, we first (as in \cite{SVMRSASH} or Section 2.2 of \cite{SSVM}) reinterpret the stochastic six-vertex model as a discrete-time interacting particle system. Then, we analyze this system by suitably adapting hydrodynamical limit methods developed for continuous-time, attractive interacting particle systems to the discrete-time setting. 

To that end, we first show that the stochastic six-vertex model is \emph{monotone} and \emph{attractive}. The first refers to the fact that the stochastic six-vertex model preserves ordering with respect to particle locations, and the second refers to the existence of a coupling between stochastic six-vertex models such that coupled particles will almost surely always remain coupled. The proof of attractivity will be based on an explicit coupling that arises from the \emph{multi-class} (or \emph{higher rank}) stochastic six-vertex model. The weights for that model were introduced in the independent works of Bazhanov \cite{TSEA} and Jimbo \cite{QGM}, although they were studied from the perspective of stochastic dynamics more recently \cite{CSVMST,ADFSVM,SM}; a more general framework for producing integrable systems with stochastic weights was provided in \cite{SSE}. In our particular setting, it will be useful to adopt a mild modification of this model by allowing the classes of particles to decrease when they couple. 

Next we establish the limit shape, or \emph{hydrodynamical limit}, of the stochastic six-vertex model in the special case when run under double-sided Bernoulli initial data. If the resulting limiting profile exhibits a rarefaction fan, then this result can be accessed using the algebraic framework of \cite{CFSAEPSSVMCL,PTAEPSSVM}. However, in the alternative setting when the limiting profile exhibits a shock, we are not aware of exact identities for the stochastic six-vertex model that are amenable to asymptotic analysis. 

So, we instead modify the framework that was introduced by Andjel-Vares \cite{HEAS} to study attractive interacting particle systems with double-sided Bernoulli initial data to apply in our setting. In addition to requiring the attractivity of the stochastic six-vertex model, these methods also require a classification of the translation-invariant stationary measures for the model. Therefore, we provide a classification of such measures through a suitable adaptation of the work of Liggett \cite{CEP}, who established the analogous result for the ASEP. 

Then, we apply the framework of Bahadoran-Guiol-Ravishankar-Saada \cite{HAPAE} (see also the survey \cite{HOAPS} and references therein) that establishes hydrodynamical limit for monotone, attractive interacting particle systems under arbitrary initial data, assuming that one understands these limits in the case of double-sided Bernoulli initial data. This will lead to the proof of the limit shape \Cref{sixvertexcylinderglobal}. 

Before proceeding, let us mention that there also exist other developed methods \cite{HLAPS,REHM,EHTP} for establishing the hydrodynamical limit for continuous-time, asymmetric particle systems. To our understanding, in order to access the limit with an arbitrary initial profile, they all require attractivity of the underlying particle system (with the exception of the relative entropy method of Yau \cite{REHM}, which instead requires differentiability of the limit shape, an assumption that does not hold in our setting). The first was due to Rezakhanlou \cite{HLAPS} and was based on a microscopic entropy estimate. Later, an alternative proof (that applied to a different class of models) was provided by Sepp\"{a}l\"{a}inen, obtained by interpreting the particle system as a random growth model. 

Unfortunately, we were unable to implement any of these methods for the stochastic six-vertex model, due to the fact that its dynamics are given by discrete-time, sequential updates. For instance, the microscopic entropy estimate from \cite{HLAPS} is proven using the generator for the interacting particle system; however, for the stochastic six-vertex model, this generator is a (sort of intricate) matrix that we were not able to analyze. Furthermore, the stochastic six-vertex model can be viewed as a random growth model, but with elaborate sequential update dynamics that we were also unable to analyze. Since the methods of \cite{HAPAE} were instead largely based on the monotonicity and attractivity of the model, and on properties of the limit shape, they were more amenable to establish the hydrodynamical limit in our specific setting. 

To establish the local statistics result \Cref{sixvertexlocalcylinder}, we follow the framework introduced by Bahadoran-Mountford \cite{CLEP}, who established the corresponding result for the ASEP through an analysis of the multi-species ASEP. The analog here will be to analyze the dynamics of the multi-species stochastic six-vertex model. As in \cite{CLEP}, this will require an ``approximate coupling'' statement (given by \Cref{rhotksum2} below) that essentially states that a stochastic six-vertex model with an approximately constant density profile can be coupled with a stationary stochastic six-vertex model such that the two models nearly couple after a sufficiently long time. 

This statement can be viewed as a generalization of the two-block estimate (see, for instance, Lemma 3.2 in Chapter 5 of the book \cite{SLIS} by Kipnis) for interacting particle systems. Indeed, in the case of the ASEP, \cite{CLEP} essentially reduced it to a two-block estimate that had earlier been proven by Kosygina \cite{BEHSL}. In our setting of the stochastic six-vertex model, we will establish the approximate coupling statement directly, through a suitable adaptation of the method implemented in \cite{BEHSL} of equating two different approximations for the total current of the model. Although these expressions for the total current from \cite{BEHSL} were obtained using the generator of the ASEP (which we recall is a bit intricate in our six-vertex situation), it will still be possible to establish analogs of them in our discrete-time setting, using the monotonicity and attractivity of the stochastic six-vertex model, as well as properties of its stationary measures.

The remainder of this section is organized as follows. In \Cref{StochasticModel} we define the model of interest to us, which is the stochastic six-vertex model on the discrete cylinder. In \Cref{Height} we state our results, which constitute a limit shape phenomenon and local statistics for this model with generic boundary data; the latter result will require a certain class of infinite-volume translation-invariant Gibbs measures, which we recall in \Cref{Translation}.

Throughout this article, we fix real numbers $0 < b_1 < b_2 < 1$ and $\lambda > 0$. Furthermore, for any probability measure $\mu$, we let $\mathbb{P}_{\mu}$ and $\mathbb{E}_{\mu}$ denote the probability measure and expectation with respect to $\mu$, respectively. Additionally, we denote the complement of any event $E$ by $E^c$.

\subsection{The Model} 

\label{StochasticModel}

Let $N \ge 1$ be a positive integer and set $L = \lfloor \lambda N \rfloor$. Define the discrete torus $\mathfrak{T} = \mathfrak{T}_N = \mathbb{Z} / N \mathbb{Z}$ and the discrete cylinder $\mathfrak{C} = \mathfrak{C}_{N; L} = \mathfrak{T} \times \{ 1, 2, \ldots , L \}$. We view the cylinder $\mathfrak{C}$ as a graph by connecting vertices $(x_1, y_1), (x_2, y_2) \in \mathfrak{C}$ if $(x_1 - x_2, y_1 - y_2) \in \big\{ (-1, 0), (1, 0), (0, -1), (0, 1) \big\}$; the torus $\mathfrak{T}$ can be viewed as a graph in a similar way. 

Now let us define six-vertex ensembles on $\mathfrak{C}$. To that end, an \emph{arrow configuration} is a quadruple $(i_1, j_1; i_2, j_2)$ such that $i_1, j_1, i_2, j_2 \in \{ 0, 1 \}$ and $i_1 + j_1 = i_2 + j_2$. We view such a quadruple as an assignment of arrows to a vertex $v \in \mathfrak{C}$. Specifically, $i_1$ and $j_1$ denote the numbers of vertical and horizontal arrows entering $v$, respectively; similarly, $i_2$ and $j_2$ denote the numbers of vertical and horizontal arrows exiting $v$, respectively. The fact that $i_1 + j_1 = i_2 + j_2$ means that the numbers of incoming and outgoing arrows at $v$ are equal; this is sometimes referred to as \emph{arrow conservation}. There are six possible arrow configurations, which are depicted on the left side of \Cref{vertexdomain}. 

A \emph{six-vertex ensemble} on $\mathfrak{C}$ is defined to be an assignment of an arrow configuration to each vertex of $\mathfrak{C}$ in such a way that neighboring arrow configurations are \emph{consistent}; this means that, if $v_1, v_2 \in \mathfrak{C}$ are two adjacent vertices, then there is an arrow to $v_2$ in the configuration at $v_1$ if and only if there is an arrow from $v_1$ in the configuration at $v_2$. Observe in particular that the arrows in a six-vertex ensemble form up-right directed paths connecting vertices of $\mathfrak{C}$, which emanate vertically from a vertex in $\mathfrak{T} \times \{ 0 \}$ and exit vertically through a vertex in $\mathfrak{T} \times \{ L + 1 \}$; see the right side of \Cref{vertexdomain} for a depiction.  

\emph{Boundary data} for a six-vertex ensemble on $\mathfrak{C}$ is prescribed by dictating which elements of $\mathfrak{T} \times \{ 0 \}$ are entrance sites for a path; this can be defined through a sequence $\psi = \big( \psi (x) \big)_{x \in \mathfrak{T}} \in \{ 0, 1 \}^{\mathfrak{T}}$, where $\psi (x)$ denotes the indicator for the event that a path vertically enters $\mathfrak{C}$ through the vertex $(i, 1)$. For example, on the right side of \Cref{vertexdomain}, we have $\psi = (1, 1, 0, 0, 1)$. One could also prescribe which sites on $\mathfrak{T} \times \{ L + 1 \}$ are exit sites for a path (which would be defined through a sequence $\chi = \big( \chi (x) \big) \in \{ 0, 1 \}^{\mathfrak{T}}$, similar to the one above) but, following \cite{LSSSVM}, we will not do this here and will instead sum over all possible such exiting locations; this corresponds to a \emph{free boundary condition} on the upper end of the cylinder $\mathfrak{C}$. 

Now, fix some $N$-tuple $\psi = \big( \psi (x) \big)$ and, following \cite{SVMRSASH}, let us define a probability measure $\mathcal{P} = \mathcal{P} (b_1, b_2) = \mathcal{P}_N (b_1, b_2; \lambda; \psi)$ on the set $\mathfrak{E} = \mathfrak{E}_{N; L; \psi}$ of six-vertex ensembles on $\mathfrak{C}$ with boundary condition given by $\psi$. To that end, we first define a \emph{vertex weight} $w (i_1, j_1; i_2, j_2)$ to each arrow configuration $(i_1, j_1; i_2, j_2)$, given explicitly by 
\begin{flalign} 
\label{wi1j1i2j2}
\begin{aligned} 
w (1, 0; 1, 0) = b_1; \qquad w (0, 1; & 0, 1) = b_2; \qquad w (1, 0; 0, 1) = 1 - b_1; \qquad w(0, 1; 1, 0) = 1 - b_2; \\
& w (0, 0; 0, 0) = 1 = w (1, 1; 1, 1),
\end{aligned}
\end{flalign}

\noindent and $w (i_1, j_1; i_2, j_2) = 0$ for any $(i_1, j_1; i_2, j_2)$ not of the above form; see the left side of \Cref{vertexdomain} for a depiction. These weights are \emph{stochastic}, in that $\sum_{i_2, j_2} w (i_1, j_1; i_2, j_2) = 1$ for any fixed $(i_1, j_1) \in \{ 0, 1 \} \times \{ 0, 1 \}$. 

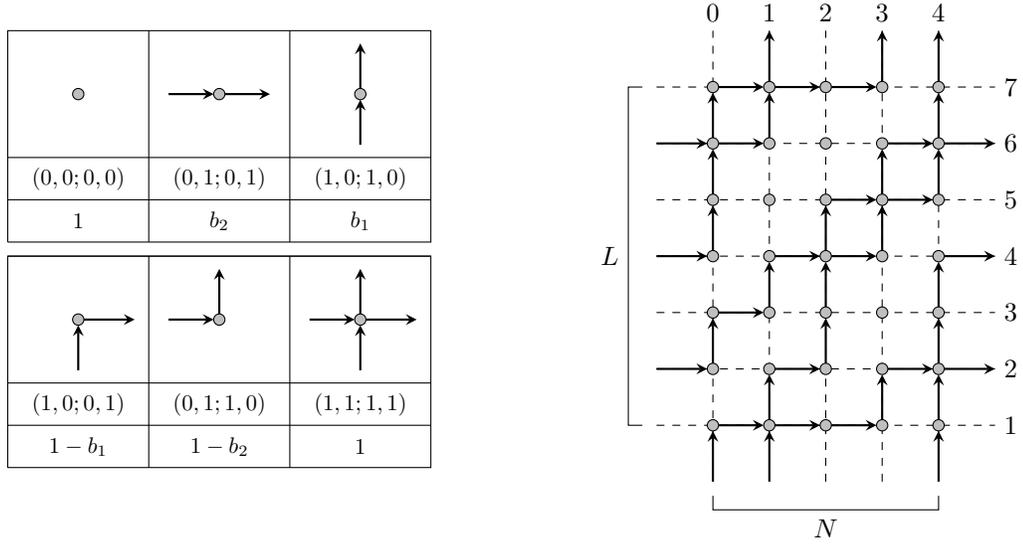
\begin{figure}[t]
	
	\begin{center}
		
		\begin{tikzpicture}[
		>=stealth,
		scale = .75 	
		]
		
		\draw[] (-11.5, 3.25) -- (-4, 3.25) -- (-4, 7) -- (-11.5, 7) -- (-11.5, 3.25);
		
		\draw[] (-11.5, -.75) -- (-4, -.75) -- (-4, 3) -- (-11.5, 3) -- (-11.5, -.75);

		\draw[] (-11.5, 4) -- (-4, 4);
		\draw[] (-9, 3.25) -- (-9, 7);
		\draw[] (-6.5, 3.25) -- (-6.5, 7);
		\draw[] (-11.5, 4.75) -- (-4, 4.75);
		
		\draw[] (-11.5, 3) -- (-4, 3);
		\draw[] (-9, -.75) -- (-9, 3);
		\draw[] (-6.5, -.75) -- (-6.5, 3);
		\draw[] (-11.5, .75) -- (-4, .75);
		\draw[] (-11.5, 0) -- (-4, 0);
		
		\draw[] (-10.25, 4.375) circle [radius = 0] node[scale = .85]{$(0, 0; 0, 0)$};
		\draw[] (-7.75, 4.375) circle [radius = 0] node[scale = .85]{$(0, 1; 0, 1)$};
		\draw[] (-5.25, 4.375) circle [radius = 0] node[scale = .85]{$(1, 0; 1, 0)$};
		
		\draw[] (-10.25, .375) circle [radius = 0] node[scale = .85]{$(1, 0; 0, 1)$};
		\draw[] (-7.75, .375) circle [radius = 0] node[scale = .85]{$(0, 1; 1, 0)$};
		\draw[] (-5.25, .375) circle [radius = 0] node[scale = .85]{$(1, 1; 1, 1)$};

		\draw[] (-10.25, 3.625) circle [radius = 0] node[scale = .85]{$1$};
		\draw[] (-7.75, 3.625) circle [radius = 0] node[scale = .85]{$b_2$};
		\draw[] (-5.25, 3.625) circle [radius = 0] node[scale = .85]{$b_1$};
		
		\draw[] (-10.25, -.375) circle [radius = 0] node[scale = .85]{$1 - b_1$};
		\draw[] (-7.75, -.375) circle [radius = 0] node[scale = .85]{$1 - b_2$};
		\draw[] (-5.25, -.375) circle [radius = 0] node[scale = .85]{$1$};

		\draw[->, black,  thick] (-8.65, 5.875) -- (-7.85, 5.875);
		\draw[->, black,  thick] (-7.65, 5.875) -- (-6.85, 5.875);
		
		\draw[->, black,  thick] (-5.25, 4.975) -- (-5.25, 5.775);
		\draw[->, black,  thick] (-5.25, 5.975) -- (-5.25, 6.775);

		\draw[->, black,  thick] (-10.15, 1.875) -- (-9.25, 1.875);
		\draw[->, black,  thick] (-10.25, .975) -- (-10.25, 1.775);

		\draw[->, black,  thick] (-8.65, 1.875) -- (-7.85, 1.875);
		\draw[->, black,  thick] (-7.75, 1.975) -- (-7.75, 2.775);
		
		\draw[->, black,  thick] (-6.15, 1.875) -- (-5.35, 1.875);
		\draw[->, black,  thick] (-5.15, 1.875) -- (-4.25, 1.875);
		\draw[->, black,  thick] (-5.25, .975) -- (-5.25, 1.775);
		\draw[->, black,  thick] (-5.25, 1.975) -- (-5.25, 2.775);
		
		\filldraw[fill=gray!50!white, draw=black] (-10.25, 1.875) circle [radius=.1];
		\filldraw[fill=gray!50!white, draw=black] (-10.25, 5.875) circle [radius=.1];
		\filldraw[fill=gray!50!white, draw=black] (-7.75, 1.875) circle [radius=.1];
		\filldraw[fill=gray!50!white, draw=black] (-7.75, 5.875) circle [radius=.1];
		\filldraw[fill=gray!50!white, draw=black] (-5.25, 1.875) circle [radius=.1];
		\filldraw[fill=gray!50!white, draw=black] (-5.25, 5.875) circle [radius=.1];

		\draw[-, dashed] (0, 0) -- (6, 0) node[right = 0]{$1$};
		\draw[-, dashed] (0, 1) -- (6, 1) node[right = 0]{$2$};
		\draw[-, dashed] (0, 2) -- (6, 2) node[right = 0]{$3$};
		\draw[-, dashed] (0, 3) -- (6, 3) node[right = 0]{$4$};
		\draw[-, dashed] (0, 4) -- (6, 4) node[right = 0]{$5$};
		\draw[-, dashed] (0, 5) -- (6, 5) node[right = 0]{$6$};
		\draw[-, dashed] (0, 6) -- (6, 6) node[right = 0]{$7$};
		
		\draw[-, dashed] (1, -1) -- (1, 7) node[above = 0]{$0$}; 
		\draw[-, dashed] (2, -1) -- (2, 7) node[above = 0]{$1$};  
		\draw[-, dashed] (3, -1) -- (3, 7) node[above = 0]{$2$}; 
		\draw[-, dashed] (4, -1) -- (4, 7) node[above = 0]{$3$}; 
		\draw[-, dashed] (5, -1) -- (5, 7) node[above = 0]{$4$}; 
		
		\draw[-, black] (-.25, 0) -- (-.5, 0) -- (-.5, 6) -- (-.25, 6); 
		\draw[] (-.5, 3) circle[radius = 0] node[left = 0]{$L$}; 
		
		\draw[-, black] (1, -1.25) -- (1, -1.5) -- (5, -1.5) -- (5, -1.25); 
		\draw[] (3, -1.5) circle[radius = 0] node[below = 0]{$N$}; 
		
		\draw[->, black, thick] (1, -1) -- (1, -.1); 
		\draw[->, black, thick] (2, -1) -- (2, -.1);
		\draw[->, black, thick] (5, -1) -- (5, -.1);
		
		\draw[->, black, thick] (1.1, 0) -- (1.9, 0);
		\draw[->, black, thick] (2.1, 0) -- (2.9, 0);
		\draw[->, black, thick] (3.1, 0) -- (3.9, 0);
		
		\draw[->, black, thick] (2, .1) -- (2, .9);
		\draw[->, black, thick] (4, .1) -- (4, .9);
		\draw[->, black, thick] (5, .1) -- (5, .9);
		
		\draw[->, black, thick] (0, 1) -- (.9, 1);
		\draw[->, black, thick] (2.1, 1) -- (2.9, 1);
		\draw[->, black, thick] (4.1, 1) -- (4.9, 1);
		\draw[->, black, thick] (5.1, 1) -- (6, 1);
		
		\draw[->, black, thick] (1, 1.1) -- (1, 1.9); 
		\draw[->, black, thick] (3, 1.1) -- (3, 1.9);
		\draw[->, black, thick] (5, 1.1) -- (5, 1.9);

		\draw[->, black, thick] (1.1, 2) -- (1.9, 2);
		
		\draw[->, black, thick] (2, 2.1) -- (2, 2.9); 
		\draw[->, black, thick] (3, 2.1) -- (3, 2.9);
		\draw[->, black, thick] (5, 2.1) -- (5, 2.9);

		\draw[->, black, thick] (0, 3) -- (.9, 3);
		\draw[->, black, thick] (2.1, 3) -- (2.9, 3);
		\draw[->, black, thick] (3.1, 3) -- (3.9, 3);
		\draw[->, black, thick] (5.1, 3) -- (6, 3);
		
		\draw[->, black, thick] (1, 3.1) -- (1, 3.9); 
		\draw[->, black, thick] (3, 3.1) -- (3, 3.9);
		\draw[->, black, thick] (4, 3.1) -- (4, 3.9);
		
		\draw[->, black, thick] (3.1, 4) -- (3.9, 4);
		\draw[->, black, thick] (4.1, 4) -- (4.9, 4);
		
		\draw[->, black, thick] (1, 4.1) -- (1, 4.9); 
		\draw[->, black, thick] (4, 4.1) -- (4, 4.9);
		\draw[->, black, thick] (5, 4.1) -- (5, 4.9);

		\draw[->, black, thick] (0, 5) -- (.9, 5);
		\draw[->, black, thick] (1.1, 5) -- (1.9, 5);
		\draw[->, black, thick] (4.1, 5) -- (4.9, 5);
		\draw[->, black, thick] (5.1, 5) -- (6, 5);

		\draw[->, black, thick] (1, 5.1) -- (1, 5.9); 
		\draw[->, black, thick] (2, 5.1) -- (2, 5.9);
		\draw[->, black, thick] (5, 5.1) -- (5, 5.9);

		\draw[->, black, thick] (1.1, 6) -- (1.9, 6);
		\draw[->, black, thick] (2.1, 6) -- (2.9, 6);
		\draw[->, black, thick] (3.1, 6) -- (3.9, 6);
		
		\draw[->, black, thick] (2, 6.1) -- (2, 7); 
		\draw[->, black, thick] (4, 6.1) -- (4, 7);
		\draw[->, black, thick] (5, 6.1) -- (5, 7);

		\filldraw[fill=gray!50!white, draw=black] (1, 0) circle [radius=.1];
		\filldraw[fill=gray!50!white, draw=black] (1, 1) circle [radius=.1];
		\filldraw[fill=gray!50!white, draw=black] (1, 2) circle [radius=.1];
		\filldraw[fill=gray!50!white, draw=black] (1, 3) circle [radius=.1];
		\filldraw[fill=gray!50!white, draw=black] (1, 4) circle [radius=.1];
		\filldraw[fill=gray!50!white, draw=black] (1, 5) circle [radius=.1];
		\filldraw[fill=gray!50!white, draw=black] (1, 6) circle [radius=.1];
		
		\filldraw[fill=gray!50!white, draw=black] (2, 0) circle [radius=.1];
		\filldraw[fill=gray!50!white, draw=black] (2, 1) circle [radius=.1];
		\filldraw[fill=gray!50!white, draw=black] (2, 2) circle [radius=.1];
		\filldraw[fill=gray!50!white, draw=black] (2, 3) circle [radius=.1];
		\filldraw[fill=gray!50!white, draw=black] (2, 4) circle [radius=.1];
		\filldraw[fill=gray!50!white, draw=black] (2, 5) circle [radius=.1];
		\filldraw[fill=gray!50!white, draw=black] (2, 6) circle [radius=.1];
		
		\filldraw[fill=gray!50!white, draw=black] (3, 0) circle [radius=.1];
		\filldraw[fill=gray!50!white, draw=black] (3, 1) circle [radius=.1];
		\filldraw[fill=gray!50!white, draw=black] (3, 2) circle [radius=.1];
		\filldraw[fill=gray!50!white, draw=black] (3, 3) circle [radius=.1];
		\filldraw[fill=gray!50!white, draw=black] (3, 4) circle [radius=.1];
		\filldraw[fill=gray!50!white, draw=black] (3, 5) circle [radius=.1];
		\filldraw[fill=gray!50!white, draw=black] (3, 6) circle [radius=.1];
		
		\filldraw[fill=gray!50!white, draw=black] (4, 0) circle [radius=.1];
		\filldraw[fill=gray!50!white, draw=black] (4, 1) circle [radius=.1];
		\filldraw[fill=gray!50!white, draw=black] (4, 2) circle [radius=.1];
		\filldraw[fill=gray!50!white, draw=black] (4, 3) circle [radius=.1];
		\filldraw[fill=gray!50!white, draw=black] (4, 4) circle [radius=.1];
		\filldraw[fill=gray!50!white, draw=black] (4, 5) circle [radius=.1];
		\filldraw[fill=gray!50!white, draw=black] (4, 6) circle [radius=.1];
		
		\filldraw[fill=gray!50!white, draw=black] (5, 0) circle [radius=.1];
		\filldraw[fill=gray!50!white, draw=black] (5, 1) circle [radius=.1];
		\filldraw[fill=gray!50!white, draw=black] (5, 2) circle [radius=.1];
		\filldraw[fill=gray!50!white, draw=black] (5, 3) circle [radius=.1];
		\filldraw[fill=gray!50!white, draw=black] (5, 4) circle [radius=.1];
		\filldraw[fill=gray!50!white, draw=black] (5, 5) circle [radius=.1];
		\filldraw[fill=gray!50!white, draw=black] (5, 6) circle [radius=.1];

		\end{tikzpicture}
		
	\end{center}	
	
	\caption{\label{vertexdomain} The chart to the left shows all six possible arrow configurations, along with the associated vertex weights. An example of a six-vertex ensemble on $\mathfrak{C}_{5; 7}$ is shown to the right.}
\end{figure}

We view $w (i_1, j_1; i_2, j_2)$ as the weight of a vertex in a six-vertex ensemble whose arrow configuration is $(i_1, j_1; i_2, j_2)$. Then, we define the \emph{weight} $w (\mathcal{E})$ of a six-vertex ensemble $\mathcal{E} \in \mathfrak{E}$ as equal to the product of the weights of all vertices in the ensemble. By induction on $L$, the stochasticity of the weights \eqref{wi1j1i2j2} can be used to deduce that $\sum_{\mathcal{E} \in \mathfrak{E}} w (\mathcal{E}) = 1$. 

Now let $\mathcal{P}$ denote the probability measure on $\mathfrak{E}$ that assigns probability $w (\mathcal{E})$ to any ensemble $\mathcal{E} \in \mathfrak{E}$. This probability measure is called the \emph{stochastic six-vertex model} on $\mathfrak{C}$ and was introduced by Gwa-Spohn \cite{SVMRSASH} in 1992.

\subsection{Infinite-Volume Translation-Invariant Gibbs Measures} 

\label{Translation}

Although our results in this paper will primarily concern the stochastic six-vertex model on the (finite) cylinder $\mathfrak{C}$, certain six-vertex measures in infinite-volume will arise as limit points of local statistics for the model in our setting. Therefore, in this section we recall a class of infinite-volume, translation-invariant Gibbs measures for the stochastic six-vertex model that was introduced in Section A.2 of \cite{CFSAEPSSVMCL}. 

These models will be defined on all of $\mathbb{Z}^2$, but let us first define the stochastic six-vertex model on the nonnegative quadrant. To that end, we fix (possibly random) boundary conditions on the nonnegative quadrant, that is, vertices on the positive $x$-axis and positive $y$-axis that are entrance sites for a directed path. For $\rho_1, \rho_2 \in [0, 1]$, a boundary condition that will be of particular interest to us is \emph{double-sided $(\rho_1, \rho_2)$-Bernoulli boundary data}, in which sites on the $y$-axis are independently entrance sites with probability $\rho_1$, and sites on the $x$-axis are independently entrance sites with probability $\rho_2$. 

Following \cite{SSVM,HSVMSRF,SHSVML}, the stochastic six-vertex model on the quadrant is defined to be a probability measure $\mathcal{P} = \mathcal{P} (b_1, b_2)$ that is the limit of a family of probability measures $\mathcal{P}_n = \mathcal{P}_n (b_1, b_2)$ defined on the set of six-vertex ensembles whose vertices are all contained in triangles of the form $\mathcal{T}_n = \{ (x, y) \in \mathbb{Z}_{\ge 0}^2: x + y \le n \}$. The first such probability measure $\mathcal{P}_1$ is supported on the unique six-vertex ensemble on $\mathcal{T}_1$ with no arrows. 

For each positive integer $n$, we define $\mathcal{P}_{n + 1}$ from $\mathcal{P}_n$ through the following Markovian update rules. Use $\mathcal{P}_n$ to sample a six-vertex ensemble $\mathcal{E}_n$ on $\mathcal{T}_n$. This gives arrow configurations (of the type shown on the left side of \Cref{vertexdomain}) to all vertices in the positive quadrant strictly below the diagonal $\mathcal{D}_n = \{ (x, y) \in \mathbb{Z}_{> 0}^2: x + y = n \}$. Each vertex on $\mathcal{D}_n$ is also given ``half'' of an arrow configuration, in the sense that it is given the directions of all entering paths but no direction of any exiting path. 

To extend $\mathcal{E}_n$ to an ensemble on $\mathbb{T}_{n + 1}$, we must ``complete'' the configurations (specify the exiting paths) at all vertices $(x, y) \in \mathcal{D}_n$. Any half-configuration can be completed in at most two ways; selecting between these completions is done randomly, according to the probabilities \eqref{wi1j1i2j2}. All choices are mutually independent. 

In this way, we obtain a random ensemble $\mathcal{E}_{n + 1}$ on $\mathcal{T}_{n + 1}$; the resulting probability measure on path ensembles with vertices in $\mathcal{T}_{n + 1}$ is denoted $\mathcal{P}_{n + 1}$. Define the limit $\mathcal{P} = \lim_{n \rightarrow \infty} \mathcal{P}_n$. 

Let us next explain how to use this six-vertex model on the quadrant to define a certain class of translation-invariant six-vertex models on all of $\mathbb{Z}^2$. To that end, we first define $\kappa > 1$ and $\varphi: [0, 1] \rightarrow [0, 1]$ by 
\begin{flalign}
\label{kappafunction}
\kappa = \displaystyle\frac{1 - b_1}{1 - b_2} > 1; \qquad \varphi (z) = \displaystyle\frac{\kappa z}{(\kappa - 1) z + 1}, \quad \text{for any $z \in [0, 1]$.}
\end{flalign} 

Now, fix $\rho \in [0, 1]$, and consider the stochastic six-vertex model on the nonnegative quadrant with double-sided $\big( \varphi (\rho), \rho \big)$-Bernoulli initial data; denote the associated measure on the set of six-vertex ensembles on $\mathbb{Z}_{> 0}^2$ by $\mu_0 = \mu_0 (\rho)$. It was shown as Lemma A.2 of \cite{CFSAEPSSVMCL} that this measure is \emph{translation-invariant} in the following sense. 

Sample a six-vertex ensemble $\mathcal{E}$ with respect to $\mu_0$. For any $(x, y) \in \mathbb{Z}_{\ge 0}^2$, let $\chi^{(v)} (x, y)$ denote the indicator for the event that an arrow in $\mathcal{E}$ vertically exits from $(x, y)$; that is, letting $\big( i_1 (x, y), j_1 (x, y); i_2 (x, y), j_2 (x, y) \big)$ denote the arrow configuration at $(x, y)$, we set $\chi^{(v)} (x, y) = i_1 (x, y + 1) = i_2 (x, y)$. Similarly, $\chi^{(h)} (x, y) = j_1 (x + 1, y) = j_2 (x, y)$ denotes the indicator for the event that an arrow in $\mathcal{E}$ horizontally exits through $(x, y)$. Then, for any $(x, y) \in \mathbb{Z}_{\ge 0}^2$, the random variables $\big\{ \chi^{(h)} (x, y + 1), \chi^{(h)} (x, y + 2), \ldots \big\} \cup \big\{ \chi^{(v)} (x + 1, y), \chi^{(v)} (x + 2, y), \ldots \big\}$ are mutually independent. Furthermore, each $\chi^{(h)} (x, y)$ and $\chi^{(v)} (x, y)$ is a $0-1$ Bernoulli random variable with mean $\varphi (\rho)$ and $\rho$, respectively. 

Observe in particular that, if $(x, y) = (0, 0)$, then this is the definition of double-sided $\big( \varphi (\rho), \rho \big)$-Bernoulli initial data for $\mu_0$. The fact that it is also true for any $(x, y) \in \mathbb{Z}_{\ge 0}^2$ allows us to define a family of measures $\mu_N = \mu_N (\rho)$ as follows. For each integer $N \ge 1$, let $\mu_N = \mu_N (\rho)$ denote the measure on $\mathbb{Z}_{\ge -N}^2$ formed by translating $\mu_0$ by $(-N, -N)$ (that is, $N$ spaces down and to the left). Due to the translation-invariance of $\mu_0$ mentioned above, these measures are compatible in the sense that $\mu_M$ is the restriction of $\mu_N$ to $\mathbb{Z}_{> -M}^2$, for any integers $N \ge M \ge 0$. 

Therefore, we can define the limit $\mu = \mu (\rho) = \lim_{N \rightarrow \infty} \mu_N (\rho)$ on all of $\mathbb{Z}^2$. By the translation-invariance of $\mu_0$, this limit is quickly seen to be invariant with respect to any vertical or horizontal shift. Thus, $\mu (\rho)$ is an \emph{infinite-volume, translation-invariant Gibbs measure} for the stochastic six-vertex model. 

Such measures are typically classified by their \emph{slope}, which under the above notation is defined to be the pair $\big( \mathbb{E} \big[ \chi^{(v)} (0, 0) \big], \mathbb{E} \big[ \chi^{(h)} (0, 0) \big] \big)$. In particular, the slope of $\mu (\rho)$ is equal to $\big( \rho, \varphi (\rho) \big)$. Thus, the above procedure produces a one-parameter family of infinite-volume, translation-invariant Gibbs measures for the stochastic six-vertex model.

\subsection{Results} 

\label{Height} 

In this paper we will be interested in the stochastic six-vertex model in the \emph{thermodynamic limit}, that is, in properties of the measure $\mathcal{P}$ as $N$ tends to $\infty$ (while $0 < b_1 < b_2 < 1$ and $\lambda > 0$ remain fixed). We will in particular analyze two properties, namely, a limit shape phenomenon and the local statistics, of this model. We will see that the former determines the latter, so we begin by explaining the limit shape phenomenon in more detail. 

To that end, we require some additional notation. With any six-vertex ensemble $\mathcal{E}$ on $\mathfrak{C} = \mathfrak{C}_{N; L}$ under boundary condition $\psi = \big( \psi (x) \big) \in \{ 0, 1 \}^N$, we associate a \emph{particle configuration} $\eta = \big( \eta_y (x) \big)$, where $x$ and $y$ range over $\mathfrak{T}_N$ and $\{ 0, 1, \ldots , L \}$, respectively. Specifically, we set $\eta_y (x)$ to be the indicator for the event that there exists a (vertical) arrow in $\mathcal{E}$ directed from $(x, y)$ to $(x, y + 1)$; this coincides with the quantity $\chi^{(v)} (x, y)$ from \Cref{Translation}.\footnote{Our results to be described below can also be formulated and established in terms of the horizontal edge indicators $\big\{ \chi^{(h)} (x, y) \big\}$ but for brevity we will not pursue this here.} For instance, $\eta_0 (x) = \psi (x)$ for each $x \in \mathfrak{T}_N$. The measure $\mathcal{P}$ from \Cref{StochasticModel} induces a probability measure on the set of particle configurations $\eta$ with $\eta_0 = \psi$; we also denote this measure on particle configurations by $\mathcal{P}$. 

We would first like to understand the global limit of a randomly chosen particle configuration $\eta$ (or path ensemble $\mathcal{E}$) sampled from the measure $\mathcal{P}$ (under some given boundary conditions), as $N$ tends to $\infty$. More specifically, we will show that the normalized random function $\eta_{\lfloor y N \rfloor} \big( \lfloor x N \rfloor \big)$ weakly converges in probability to a deterministic function $G_y (x)$ on $(x, y) \in \mathbb{T} \times [0, \lambda]$, where $\mathbb{T} = \mathbb{R} / \mathbb{Z}$ denotes a torus; this is known as a \emph{limit shape phenomenon}. 

Such a result is given by the following theorem, which was originally predicted by Gwa-Spohn as equation (5) of \cite{SVMRSASH} and then later by Reshetikhin-Sridhar as Proposition 3 of \cite{LSSSVM}. 

\begin{thm}

\label{sixvertexcylinderglobal} 

Fix real numbers $0 < b_1 < b_2 < 1$ and $\lambda, \varepsilon > 0$, and a measurable function $\Psi: \mathbb{T} \rightarrow [0, 1]$. Recalling $\varphi$ from \eqref{kappafunction}, let $G_y (x) = G (x, y)$ denote the entropy solution to the conservation law  
\begin{flalign}
\label{gxydefinition} 
\displaystyle\frac{\partial}{\partial y} G (x, y) + \displaystyle\frac{\partial}{\partial x} \Big( \varphi \big( G(x, y) \big) \Big) = 0,
\end{flalign}

\noindent on $\mathbb{T} \times [0, \lambda]$, with initial data given by $G_0 (x) = G(x, 0) = \Psi (x)$ for each $x \in \mathbb{T}$. 

For each integer $N \ge 1$, let $\psi = \psi^{(N)} = \big( \psi (x) \big) = \big( \psi^{(N)} (x) \big)_{x \in \mathfrak{T}_N} \in \{ 0, 1 \}^N$ denote a boundary condition, and assume that
\begin{flalign}
\label{limitpsi}
\displaystyle\lim_{N \rightarrow \infty} \displaystyle\sup_{0 \le x_1 \le x_2 \le 1} \left| \displaystyle\frac{1}{N} \displaystyle\sum_{j = \lfloor x_1 N \rfloor }^{\lfloor x_2 N \rfloor} \psi^{(N)} (j) - \displaystyle\int_{x_1}^{x_2} \Psi (x) dx \right| = 0.
\end{flalign}

Furthermore, for each $N \ge 1$, let $\eta = \eta^{(N)} = \big( \eta_y (x) \big) = \big( \eta_y^{(N)} (x) \big)$ denote a particle configuration sampled with respect to the measure $\mathcal{P}_N \big( b_1, b_2; \lambda; \psi^{(N)} \big)$ from \Cref{StochasticModel}. Then, 
\begin{flalign}
\label{etasumestimaten}
\displaystyle\lim_{N \rightarrow \infty} \mathbb{P} \left[ \displaystyle\max_{\substack{0 \le X_1 \le X_2 < N \\ 0 \le Y_1 \le Y_2 \le L}} \bigg| \displaystyle\frac{1}{N^2} \displaystyle\sum_{y = Y_1}^{Y_2} \displaystyle\sum_{x = X_1}^{X_2} \eta_y (x) - \displaystyle\int_{Y_1 / N}^{Y_2 / N} \displaystyle\int_{X_1 / N}^{X_2 / N} G_y (x) dx dy \bigg| > \varepsilon \right] = 0. 
\end{flalign} 

\end{thm} 

\begin{rem}
	
The equation \eqref{gxydefinition} cannot be interpreted in the strong sense, since its solutions can develop singularities (called \emph{shocks}) in finite time. However, weak solutions to \eqref{gxydefinition} are not unique, and so one requires a way of specifying one ``physically relevant'' solution to the equation. This is known as the \emph{entropy solution} and can be defined as the unique weak solution to \eqref{gxydefinition} satisfying the \emph{entropy inequality}, which in this case states that 
\begin{flalign*}
\displaystyle\int_0^{\lambda} \displaystyle\int_{\mathbb{T}} \bigg( \partial_y f (s, t) \big| G (s, t) - c \big| + \partial_x f (s, t) \Big| \varphi \big( & G (s, t) \big) - \varphi (c) \Big|  \bigg) ds dt + \displaystyle\int_{\mathbb{T}} \big| G (s, 0) - c \big| f(s, 0) ds \ge 0,
\end{flalign*}  

\noindent for any nonnegative, smooth function $f: \mathbb{T} \times [0, \lambda] \rightarrow \mathbb{R}$ and constant $c \in \mathbb{R}$; see equation (2.6) of the book \cite{SCL}. For bounded, measurable initial data, the existence and uniqueness of such a solution (measurable and bounded on $\mathbb{T} \times [0, \lambda]$, and also continuous in $t$) is due to Kru\v{z}kov \cite{QSIV}; see also Theorem 2.3.5 of \cite{SCL}.

\end{rem}

Next, let us describe our results on \emph{local statistics} for the stochastic six-vertex model; these concern the large $N$ limit of the joint distribution (in a random six-vertex ensemble sampled from $\mathcal{P}$) of the arrow configurations assigned to vertices that are in a finite neighborhood of some $(X, Y) \in \mathfrak{C}_{N; L}$. The following theorem states that these local statistics are given by the infinite-volume translation-invariant Gibbs measures $\mu (\rho)$ (with slope $\big( \rho, \varphi (\rho) \big)$) described in \Cref{Translation}, where $\rho$ is determined by the global limit shape \Cref{sixvertexcylinderglobal}. In what follows, for any six-vertex ensemble $\mathcal{E}$ on a domain $\mathcal{D}$, we denote by $\mathcal{E} |_{\Lambda}$ the restriction of $\mathcal{E}$ to some subset $\Lambda \subseteq \mathcal{D}$.

\begin{thm}
	
\label{sixvertexlocalcylinder}

Adopt the notation and assumptions of \Cref{sixvertexcylinderglobal}. Fix an integer $k > 0$ and some pair $(u, v) \in \mathbb{T} \times (0, \lambda)$ such that that $(u, v)$ is a continuity point of $G$; denote $G (u, v) = \rho$. For each integer $N \ge 1$, let $U_N \in \mathfrak{T}_N$ and $V_N \in [k + 1, L]$ be integers such that $\lim_{N \rightarrow \infty} \frac{U_N}{N} = u$ and $\lim_{N \rightarrow \infty} \frac{V_N}{N} = v$. 

Sample six-vertex ensembles $\mathcal{E} = \mathcal{E}^{(N)}$ on $\mathfrak{C}$ with respect to the measure $\mathcal{P}_N \big( b_1, b_2; \lambda; \psi^{(N)} \big)$ from \Cref{StochasticModel}, and $\mathcal{F} = \mathcal{F}_{\rho}$ on $\mathbb{Z}^2$ with respect to the measure $\mu (\rho)$ from \Cref{Translation}. Then, the law of $\mathcal{E} |_{[U_N - k, U_N + k] \times [V_N - k, V_N + k]}$ converges to that of $\mathcal{F} |_{[-k, k] \times [-k, k]}$, as $N$ tends to $\infty$.

\end{thm}

\begin{rem} 
	
	Both \Cref{sixvertexcylinderglobal} and \Cref{sixvertexlocalcylinder} admit analogs when the cylinder $\mathfrak{C}_{N; L}$ is replaced by an $N \times L$ rectangle, after fixing boundary data along its left and bottom boundaries and allowing free boundary conditions along its top and right boundaries. In this case, the global limit is again given by the entropy solution to the equation \eqref{gxydefinition} (now on $[0, 1] \times [0, \lambda]$) with specified boundary conditions along the $x$-axis and $y$-axis. Moreover, the local statistics are given by the infinite-volume translation-invariant Gibbs measures from \Cref{Translation} whose slopes are again prescribed by the global law. It is plausible that these results can be established through similar methods as used to show \Cref{sixvertexcylinderglobal} and \Cref{sixvertexlocalcylinder}, but we will not pursue this here for the sake of brevity.

\end{rem}

The remainder of this paper is organized as follows. In \Cref{Coupling} we first explain how to reinterpret the stochastic six-vertex model as an interacting particle system, and then provide monotone and attractive couplings between several such systems. 

We next establish the limit shape result given by \Cref{sixvertexcylinderglobal} for the stochastic six-vertex model. This will require understanding the limit shape of the model in the case of double-sided Bernoulli initial data, which will in turn require a classification of the extremal, translation-invariant, stationary measures for the stochastic six-vertex model. Thus, in \Cref{StationaryTranslation} we will classify all such measures. Then, we will establish the hydrodynamical limit for the stochastic six-vertex model with double-sided Bernoulli initial data in \Cref{HeightLimit}, which will be used in \Cref{LimitGeneral} to prove the limit shape \Cref{sixvertexcylinderglobal} for the stochastic six-vertex model with arbitrary initial data. 

After this, we will establish the local statistics result \Cref{sixvertexlocalcylinder}. In particular, we will prove this statement in \Cref{StationaryConverge}, conditional on a certain ``approximate coupling'' result, which will then be shown in \Cref{ProofApproximate}.

\subsection*{Acknowledgments}

The author heartily thanks Alexei Borodin, Ivan Corwin, and Jeffrey Kuan for enlightening discussions. The author is also grateful to the anonymous referee for helpful suggestions on an earlier draft of this manuscript. This work was partially supported by the NSF Graduate Research Fellowship under grant number DGE1144152 and NSF grant DMS-1664619.

\section{Monotonicity and Couplings} 

\label{Coupling}

The proofs of \Cref{sixvertexcylinderglobal} and \Cref{sixvertexlocalcylinder} will largely use certain monotonicity results and couplings that we provide in this section. However, before describing them, we first in \Cref{StochasticVertexLine} explain a way of sampling the stochastic six-vertex model on the discrete upper half-plane through an interacting particle system. We then establish a monotonicity result for the stochastic six-vertex model in \Cref{Monotonicity}. Next, in \Cref{HigherRank}, we recall the definition of the multi-class stochastic six-vertex model from \cite{SM}, which will be used to introduce the coupling of interest to us in \Cref{CouplingHigherRank}.

\subsection{An Associated Particle System} 

\label{StochasticVertexLine}

In this section we provide an alternative (equivalent) definition of the stochastic six-vertex model as an interacting particle system on $\mathbb{Z}$, partially following Section 2.2 of \cite{SSVM}. 

Here, we will consider the stochastic six-vertex model on the discrete upper half-plane $\mathfrak{H} = \mathbb{Z} \times \mathbb{Z}_{\ge 0}$ with some boundary condition $\psi = \big( \psi (x) \big)_{x \in \mathbb{Z}} \in \{ 0, 1 \}^{\mathbb{Z}}$; let the set of six-vertex ensembles on $\mathfrak{H}$ with boundary data $\psi$ be denoted by $\mathfrak{E}_{\psi}$. As explained in \Cref{Height}, any six-vertex ensemble $\mathcal{E} \in \mathfrak{E}_{\psi}$ can be expressed through a particle configuration $\eta = \big\{ \eta_t (x) \big\}$, where $(x, t)$ ranges over $\mathfrak{H}$; under this notation, $\eta_t (x)$ is the indicator for the existence of a vertical arrow in $\mathcal{E}$ directed from $(x, t)$ to $(x, t + 1)$. As in \cite{SSVM}, we will view these vertical arrows as particles; the parameters $x \in \mathbb{Z}$ and $t \in \mathbb{Z}_{\ge 0}$ will index space and time, respectively. Thus, $\eta_t (x)$ denotes the indicator for the existence of a particle at location $x$ at time $t$; arrow conservation then implies that the total number of particles is conserved over time. 

\emph{Initial data} at time $t = 0$ for this particle system is then prescribed by the boundary data; set $\eta_0 (x) = \psi (x)$ for each $x \in \mathbb{Z}$. We assume for the moment that the system has finitely many particles (this assumption will later be removed), which is equivalent to the condition that $\sum_{x = -\infty}^{\infty} \psi (x) < \infty$. These particles will then jump according to certain stochastic dynamics, defined below. 

In what follows, it will be useful to \emph{tag} the particles of the model, meaning that we track their evolution over time by indexing them based on initial position. Specifically, let the initial positions of the particles be denoted by $\textbf{p}_0 = \big( p_0 (-M), p_0 (1 - M), \ldots , p_0 (N) \big)$, so that there are $M + N + 1$ particles in the system; then, $\eta_0 (x) = \psi (x) = \textbf{1}_{x \in \textbf{p}_0}$ for each $x \in \mathbb{Z}$. We order the particles of $\textbf{p}_0$ such that $p_0 (-M) < p_0 (1 - M) < \cdots < p_0 (N)$. The particle initially at site $p_0 (k)$ will be referred to as \emph{particle $k$} for each $k \in [-M, N]$. For each $k \in \mathbb{Z}$ and $t > 0$, $p_t (k)$ will denote the position of particle $k$ at time $t$. For any $t \in \mathbb{Z}_{\ge 0}$, we set $p_t (i) = -\infty$ for each $i < -M$ and $p_t (i) = \infty$ for each $i > N$.  

\begin{definition} 
	
\label{sixvertexparticles1}

Given the locations $\textbf{p}_{t - 1} = \big( p_{t - 1} (-M), p_{t - 1} (1 - M), \ldots , p_{t - 1} (N) \big)$ of all particles at some time $t - 1 \ge 0$, their locations $\textbf{p}_t = \big( p_t (-M), p_t (1 - M), \ldots , p_t (N) \big)$ at time $t$ are defined according to the following stochastic procedure.

\begin{enumerate} 
	
\item Let $\big\{ \chi_t (x) \big\}_{x \in \mathbb{Z}}$ and $\big\{ j_t (x) \big\}_{x \in \mathbb{Z}}$ denote a sequence of mutually independent random variables, with each $\chi_t (x)$ a $b_1$-Bernoulli $0-1$ random variable and each $j_t (x)$ chosen according to the $b_2$-geometric distribution. Specifically, for each $x \in \mathbb{Z}$, we set $\mathbb{P} \big[ \chi_t (x) = 1] = b_1 = 1 - \mathbb{P} \big[ \chi_t (x) = 0 \big]$ and $\mathbb{P} \big[ j_t (x) = r \big] = (1 - b_2) b_2^{r - 1}$ for any integer $r \ge 1$. 

\item Let $k \in [-M, N]$ be an integer, and assume that $p_t (k - 1)$ has been set; denote $x = p_{t - 1} (k)$. 

\begin{enumerate}

\item If $p_t (k - 1) < x$, then set $p_t (k) = \min \big\{ x + j_t (x), p_{t - 1} (k + 1) \big\}$ if $\chi_t (x) = 0$ and $p_t (k) = x$ if $\chi_t (x) = 1$.

\item If $p_t (k - 1) = x$, then set $p_t (k) = \min \big\{ x + j_t (x), p_{t - 1} (k + 1) \big\}$. 

\end{enumerate}
	
\end{enumerate}

\end{definition}

Stated alternatively, for each $k \in \{ -M, 1 - M, \ldots , N \}$ (in that order), particle $k$ jumps some nonnegative number of spaces to the right, as follows. It will first decide whether to move or stay. If particle $k - 1$ has not jumped onto particle $k$'s original position $p_{t - 1} (k) = x$, then particle $k$ will choose to stay or move with probabilities $b_1$ and $1 - b_1$, respectively; this is defined by the random variable $\chi_t (x)$. If instead particle $k - 1$ jumped onto particle $k$'s original position, then particle $k$ cannot stay and must move. If particle $k$ decides to move, then it will jump to the right according to a $b_2$-geometric distribution (given by $j_t (x)$), ending at the time $t - 1$ location of particle $k + 1$ if it attempts to jump to the right of it.

This provides a way of sampling a random \emph{particle position sequence} $\textbf{p} = \textbf{p}_t = (\textbf{p}_t) = \big( \textbf{p}_t (k) \big)$ under some initial data $\psi = \big( \psi (x) \big)$. By setting $\eta_t (x) = \textbf{1}_{x \in \textbf{p}_t}$ for any $(x, t) \in \mathfrak{H}$, this determines a random particle configuration $\eta = \eta_t = (\eta_t) = \big( \eta_t (k) \big)$, which in turn determines a random six-vertex ensemble $\mathcal{E} \in \mathfrak{E}_{\psi}$. Denote the probability of selecting such an $\mathcal{E}$ under this procedure by $p(\mathcal{E})$. 

It is quickly observed that $-\infty < p_t (k) < \infty$ almost surely for each $k \in [-M, N]$ and $t \in \mathbb{Z}_{\ge 0}$, meaning that the associated vertex ensemble $\mathcal{E}$ only contains finitely vertices that are not assigned arrow configuration $(0, 0; 0, 0)$. Thus, the weight $w (\mathcal{E})$ (recall \Cref{StochasticModel}) of any such ensemble is almost surely well-defined. The following proposition, which was originally observed in Section 2.2 of \cite{SSVM} and can quickly be verified from the above definitions, indicates that above induced measure on $\mathfrak{E}_{\psi}$ coincides with the six-vertex measure introduced in \Cref{StochasticModel}.

\begin{lem}[{\cite[Section 2.2]{SSVM}}]
	
\label{pw} 

Under the above notation, we have that $p (\mathcal{E}) = w (\mathcal{E})$ for any six-vertex ensemble $\mathcal{E} \in \mathfrak{E}_{\psi}$. 	

\end{lem}

Based on the sampling procedure given by \Cref{sixvertexparticles1}, we can define the stochastic six-vertex model with infinitely many particles through a method similar to the one implemented by Harris \cite{ASMPGM} in the context of the ASEP on $\mathbb{Z}$. More specifically, let $\textbf{p}_0 = \big( p_0 (k) \big)_{k \in \mathbb{Z}}$ denote a possibly infinite particle position sequence on $\mathbb{Z}$. For a fixed $N \in \mathbb{Z}_{\ge 1}$, the stochastic six-vertex model $\textbf{p}_t = \big( p_t (k) \big)$ for $t \in [0, N]$ is defined as follows. 

First sample mutually independent random variables $\big\{ \chi_t (x) \big\}$ and $\big\{ j_t (x) \big\}$ for $(x, t) \in \mathbb{Z} \times \{ 1, 2, \ldots , N \}$ as in \Cref{sixvertexparticles1}. We call an integer $k$ \emph{separating} if 
\begin{flalign}
\label{separatingk} 
\chi_t (k + t) = 0 \quad \text{and} \quad  \displaystyle\max_{m < k + t} \big( j_t (m) + m \big) \le k + t, \quad \text{for each $t \in [1, N]$.}
\end{flalign} 

We then have the following lemma. 

\begin{lem}
	
\label{riseparating} 

There almost surely exists a (random) doubly-infinite sequence of integers $ \cdots < R_{-1} < R_0 < R_1 < \cdots $ such that $R_i$ is separating for each $i \in \mathbb{Z}$. 
\end{lem} 

\begin{proof}
	
For each $k \in \mathbb{Z}$, let $F(k)$ denote the event that $k$ is not separating. Furthermore, for any $u \in \mathbb{Z}$ and $v \in \mathbb{Z}_{\ge 1}$, define the event $E(u; v) = \bigcap_{k = u + 1}^{u + v^2} F(k)$ that no separating integers exist in the interval $[u + 1, u + v^2]$. We claim that
\begin{flalign} 
\label{euvestimate} 
\mathbb{P} \big[ E (u; v) \big] \le c^{-1} (1 - c)^v 
\end{flalign} 

\noindent for some sufficiently small constant $c = c(b_2, N) > 0$. 

To that end, first define the events 
\begin{flalign*}
& F_1 (k; t) = \big\{ \chi_t (k + t) = 1 \big\} \cup \left\{ \displaystyle\max_{k + t - v \le m < k + t} \big( j_t (m) + m \big) > k + t \right\}; \\
& F_2 (k; t) = \left\{ \displaystyle\max_{ m < k + t - v} \big( j_t (m) + m \big) > k + t \right\}, 
\end{flalign*}

\noindent for any integers $k \in [u + 1, u + v^2]$ and $t \in [1, N]$. Then, $F(k) = \bigcup_{t = 1}^N \big( F_1 (k; t) \cup F_2 (k; t) \big)$, so that 
\begin{flalign}
\label{euvf1f2} 
E (u; v) \subseteq \bigcap_{i = 1}^v \left( \bigcup_{t = 1}^N F_1 (u + iv; t) \right) \cup \bigcup_{i = 1}^v \bigcup_{t = 1}^N F_2 (u + iv; t).
\end{flalign}

Observe that the events $F_1 (u + iv; t)$ are mutually independent as $(i, v)$ ranges over $[1, v] \times [1, N]$, since $F_1 (u + iv; t)$ depends only on the random variables $\big\{ \chi_t (u + iv + t) \big\} \cup \{ j_t (u + iv - v + t), j_t (u + iv - v + t + 1), \ldots , j_t (u + iv + t - 1) \}$. Furthermore, for fixed $i \in [1, v]$, the events $\big\{ F_2 (u + iv; t) \big\}$ are mutually independent over $t \in [1, N]$.

Using these facts, let us bound the probability of the first event on the right side of \eqref{euvf1f2}. To do this, observe for any $(k, t) \in [u + 1, u + v^2] \times [1, N]$ that $\mathbb{P} \big[ \chi_t (k + t) = 1 \big] = b_1$ and that $\mathbb{P} \big[ j_t (k + t - m) \le m \big] = 1 - b_2^m$, for any $m > 0$. The independence of these events therefore implies that $\mathbb{P} \big[ F_1 (u + iv; t) \big] =  1 - (1 - b_1) \prod_{m = 1}^{\infty} (1 - b_2^m) < 1 - c_1$, for some constant $c_1 = c_1 (b_2) > 0$ (since $b_1 < b_2$). Hence, 
\begin{flalign}
\label{probabilityf1uivt}
\mathbb{P} \left[ \bigcap_{i = 1}^v \left( \bigcup_{t = 1}^N F_1 (u + iv; t) \right) \right] \le (1 - c_1^N)^v.
\end{flalign}

\noindent To bound the probability of the second event on the right side of \eqref{euvf1f2}, observe that 
\begin{flalign}
\label{probabilityf2uivt} 
\mathbb{P} \big[ F_2 (u + iv; t) \big] = 1 - \displaystyle\prod_{m = v + 1}^{\infty} \Big( 1 - \mathbb{P} \big[ j_t (u + iv + t - m) > m \big] \Big) = 1 - \displaystyle\prod_{m = v + 1}^{\infty} (1 - b_2^{m - 1}) \le c_2^v,
\end{flalign}

\noindent for some constant $c_2 = c_2 (b_2) < 1$ and any $i \in [1, v]$. Thus, combining \eqref{euvf1f2}, \eqref{probabilityf1uivt}, \eqref{probabilityf2uivt}, and a union bound yields 
\begin{flalign*}
\mathbb{P} \big[ E (u; v) \big] \le (1 - c_1^N)^v  + v \big( 1 - ( 1 - c_2^v)^N \big),
\end{flalign*}

\noindent from which \eqref{euvestimate} follows. Thus, $\sum_{i = 1}^{\infty} \mathbb{P} \big[ E (5^k; 2^k) \big] + \sum_{i = 1}^{\infty} \mathbb{P} \big[ E (-5^k; 2^k) \big] < \infty$, from which we deduce the lemma from the Borel-Cantelli lemma. 
\end{proof} 

Letting $\cdots < R_{-1} < R_0 < R_1 < \cdots $ be as in \Cref{riseparating}, we can sample the stochastic six-vertex model on each finite domain 
\begin{flalign*} 
\mathfrak{D}_i = \big\{ (x, t) \in \mathbb{Z} \times \{ 0, 1, \ldots , N \}: R_i + t < x \le R_{i + 1} + t \big\}
\end{flalign*} 

\noindent through \Cref{sixvertexparticles1}. The stochastic six-vertex models on these domains $\mathfrak{D}_i$ are mutually disjoint, since \eqref{separatingk} and the dynamics described in \Cref{sixvertexparticles1} together imply that any particle in $[R_i + t, R_{i + 1} + t - 1]$ at time $t - 1$ must be in the interval $[R_i + t + 1, R_{i + 1} + t]$ at time $t$. Since $\bigcup_{i \in \mathbb{Z}} \mathfrak{D}_i = \mathbb{Z} \times \{ 0, 1, \ldots , N \}$, the union of these finite stochastic six-vertex models on the $\mathfrak{D}_i$ yields an infinite stochastic six-vertex model on $\mathbb{Z} \times \{ 0, 1, \ldots , N \}$. The particle configuration $\eta = \eta_t$ is obtained by setting $\eta_t (x) = \textbf{1}_{x \in \textbf{p}_t}$ for each $x \in \mathbb{Z}$ and $t \in \{ 0, 1, \ldots , N \}$, from which one obtains the associated six-vertex ensemble $\mathcal{E}$.

\subsection{Monotonicity} 

\label{Monotonicity} 

In this section we establish \Cref{lambdaximonotone} below, which is a monotonicity result for the stochastic six-vertex model with respect to a certain ordering on particle position sequences. We begin by defining this ordering on both particle configurations and position sequences. 

\begin{definition}
	
	\label{ordering}
	
	Let $\eta = \big( \eta (i) \big) \in \{ 0, 1 \}^{\mathbb{Z}}$ and $\xi = \big( \xi (i) \big)_{i \in \mathbb{Z}} \in \{ 0, 1 \}^{\mathbb{Z}}$ denote particle configurations on $\mathbb{Z}$. We say that $\eta \ge \xi$ (or equivalently that $\xi \le \eta$) if $\eta (i) \ge \xi (i)$ for each $i \in \mathbb{Z}$. Furthermore, if $\textbf{p} = \big( p (k) \big)$ and $\textbf{q} = \big( q(k) \big)$ denote two sequences of particle positions on $\mathbb{Z}$, then we say that $\textbf{p} \ge \textbf{q}$ (or equivalently that $\textbf{q} \le \textbf{p}$) if $p (k) \ge q(k)$ for each $k$. 
	
\end{definition}

\begin{rem} 
	
	\label{orderingetap}
	
	If $\eta$ and $\xi$ are two particle configurations associated with two sequences $\textbf{p}$ and $\textbf{q}$ of particle positions, respectively, then $\eta \ge \xi$ has no direct implication on whether $\textbf{p} \ge \textbf{q}$. 
	
\end{rem} 

The following proposition now states that the stochastic six-vertex model preserves ordering with respect to particle position sequences, in a certain sense. We refer to this as \emph{monotonicity}.

\begin{prop}
	
	\label{lambdaximonotone} 
	
	Let $\textbf{\emph{p}}_0 = \big( p_0 (k) \big)$ and $\textbf{\emph{q}}_0 = \big( q_0 (k) \big)$ denote two finite particle position sequences such that $\textbf{\emph{p}}_0 \ge \textbf{\emph{q}}_0$. Also let $\textbf{\emph{p}}_t = (\textbf{\emph{p}}_t)_{t \ge 0}$ and $\textbf{\emph{q}}_t = (\textbf{\emph{q}}_t)_{t \ge 0}$ denote the stochastic six-vertex models run with initial data $\textbf{\emph{p}}_0$ and $\textbf{\emph{q}}_0$, respectively. Then it is possible to couple the laws of $\textbf{\emph{p}}_t$ and $\textbf{\emph{q}}_t$ so that $\textbf{\emph{p}}_t \ge \textbf{\emph{q}}_t$ almost surely for each integer $t \ge 0$. 
\end{prop} 

\begin{proof}
	
	It suffices to establish this proposition when $t = 1$, for then the result would follow from induction on $t$ and the Markov property of the stochastic six-vertex model. Thus, let us assume that $t = 0$ and first exhibit this coupling when $\textbf{p}_0$ and $\textbf{q}_0$ \emph{differ in one site}. Letting $\textbf{q}_0 = \big( q_0 (-M), q_0 (1 - M), \ldots , q_0 (N) \big)$ and $\textbf{p}_0 = \big( p_0 (-M), p_0 (1 - M), \ldots , p_0 (N) \big)$, this means that there exists some integer $k \in [-M, N]$ such that $p_0 (i) = q_0 (i)$ for all $i \in [-M, N] \setminus \{ k \}$ and $p_0 (k) = q_0 (k) + 1$.
	
	To that end, let $\big\{ \chi (k) \big\}_{k \in \mathbb{Z}}$ and $\big\{ j (k) \big\}_{k \in \mathbb{Z}}$ denote sequences of mutually independent random variables, with each $\chi (k)$ a $b_1$-Bernoulli $0-1$ random variable and each $j (k)$ chosen according to the $b_2$-geometric distribution (as in \Cref{sixvertexparticles1}). Using these random variables, we can sample $p_1 (i)$ and $q_1 (i)$ for each $i < k - 1$ through a procedure similar to the one described by \Cref{sixvertexparticles1}, except where we index the randomness by particle number instead of by site. 
	
	More specifically, for any $i < k - 1$, the following explains how to sample $p_1 (i)$ assuming that $p_1 (i - 1)$ has been set; here, we let $x = p_0 (i)$.

	\begin{enumerate}
		
		\item If $p_1 (i - 1) < x$, then set $p_1 (i) = \min \big\{ x + j (i), p_0 (i + 1) \big\}$ if $\chi (i) = 0$ and $p_1 (i) = x$ if $\chi (i) = 1$.
		
		\item If $p_0 (i - 1) = x$, then set $p_1 (i) = \min \big\{ x + j (x), p_0 (i + 1) \big\}$. 
		
	\end{enumerate}

	We use the same procedure to sample $q_1 (i)$ when $i < k - 1$. In particular, since the same randomness (given by the $\chi (i)$ and $j (i)$) is used to sample $p_1 (i)$ and $q_1 (i)$, and since $p_0 (m) = q_0 (m)$ for each $m \le k - 1$, we have that $p_1 (i) = q_1 (i)$ for each $i < k - 1$ under this coupling. 
	
	The locations $p_1 (k - 1)$, $q_1 (k - 1)$, $p_1 (k)$, and $q_1 (k)$ will be sampled slightly differently. Once again, the procedure below is described only for $\textbf{p}$, but we use the same one for $\textbf{q}$. In what follows, we let $x = p_0 (k - 1)$. 
	
	\begin{enumerate}

		\item If $\chi (k - 1) = 1$ and $p_0 (k - 2) < x$, then set $p_1 (k - 1) = x$. In this case, further set $p_1 (k) = p_0 (k)$ if $\chi (k) = 1$ and set $p_1 (k) = \max \big\{ p_0 (k) + j(k), p_0 (k + 1) \big\}$ if $\chi (k) = 0$.
			
		\item Assume instead that $\chi (k - 1) = 0$ or $p_0 (k - 2) = x$. 
		
		\begin{enumerate}
			
			\item If $x + j (k - 1) < p_0 (k)$, then set $p_1 (k - 1) = x + j (k - 1)$. Further set $p_1 (k) = p_0 (k)$ if $\chi (k) = 1$, and set $p_1 (k) = \min \big\{ p_0 (k) + j(k), p_0 (k + 1)\big\}$ if $\chi (k) = 0$. 
				
			\item  If $p_0 (k) \le x + j (k - 1) < p_0 (k + 1)$, then set $p_1 (k - 1) = p_0 (k)$ and $p_1 (k) = x + j (k - 1) + 1$.
		
			\item  If $x + j (k - 1) \ge p_0 (k + 1)$, then set $p_1 (k - 1) = p_0 (k)$ and $p_1 (k) = p_0 (k + 1)$.
			
		\end{enumerate}

	\end{enumerate}

	Thus, if particle $k - 1$ does not attempt to jump to or past $p_0 (k)$, then particles $k - 1$ and $k$ are sampled individually, as before. Otherwise, the new locations of the particles are sampled simultaneously, according to the randomness given by $j (k - 1)$. 
	
	Next, for $i > k$, we sample the locations $p_1 (i)$ and $q_1 (i)$ as described above in the case $i < k - 1$. Although this procedure is different from the one given by \Cref{sixvertexparticles1}, one can quickly verify that it samples a particle position sequence according to the stochastic six-vertex measure from \Cref{StochasticModel} (as in \Cref{pw}). 
	
	Thus, this provides a coupling between $\textbf{p}_1$ and $\textbf{q}_1$. We claim under this coupling that $\textbf{p}_1 \ge \textbf{q}_1$ almost surely, to which end we must show that $p_1 (i) \ge q_1 (i)$ for each $i \in \mathbb{Z}$. It was observed above that $p_1 (i) = q_1 (i)$ when $i < k - 1$, that is, $\textbf{p}_1$ and $\textbf{q}_1$ coincide to the left of their $(k - 1)$-th particles.
	
	Let us now verify the claim when $i = k - 1$. To that end, first assume that $\chi (k - 1) = 1$ and $q_1 (k - 2) =  p_1 (k - 2) < p_0 (k - 1) = q_0 (k - 1)$. Then, the $(k - 1)$th particles of $\textbf{p}_0$ and $\textbf{q}_0$ do not move, and so $q_1 (k - 1) = q_0 (k - 1) = p_0 (k - 1) = p_1 (k - 1)$. If instead either $\chi (k - 1) = 0$ or $q_1 (k - 2) =  p_1 (k - 2) = p_0 (k - 1) = q_0 (k - 1)$, then the $(k - 1)$-th particles of $\textbf{p}_0$ and $\textbf{q}_0$ jump, and we have that $q_1 (k - 1) = \max \big\{ q_0 (k - 1) + j (k - 1), q_0 (k) \big\} \le \max \big\{ p_0 (k - 1) + j (k - 1), p_0 (k) \big\} = p_1 (k - 1)$. Thus, in either case, $p_1 (k - 1) \ge q_1 (k - 1)$. 
	
	Next we analyze when $i = k$. First, suppose that $q_0 (k - 1) + j (k - 1) < q_0 (k)$, so that the $(k - 1)$-th particle of $\textbf{q}_0$ does not attempt to jump at or to the right of $q_0 (k)$. Then either $\chi (k) = 1$, in which case the $k$-th particles of $\textbf{p}_0$ and $\textbf{q}_0$ do not jump and $q_1 (k) = q_0 (k) < p_0 (k) = p_1 (k)$, or $\chi (k) = 0$, in which case the $k$-th particles of $\textbf{q}_0$ and $\textbf{p}_0$ both jump and part 2(a) of the above coupling yields $q_1 (k) = \max \big\{ q_0 (k) + j (k), q_0 (k + 1) \big\} \le \max \big\{ p_0 (k) + j (k), p_0 (k + 1) \big\} = p_1 (k)$. 
	
	Next suppose that $q_0 (k - 1) + j (k - 1) \ge q_0 (k + 1)$, that is, the $(k - 1)$-th particle of $\textbf{q}_0$ attempts to jump at or to the right of $q_0 (k + 1)$. Then, the same holds for $\textbf{p}_0$ (namely, $p_0 (k - 1) + j(k - 1) \ge p_0 (k + 1)$), and so part 2(c) of the above coupling applied to both $\textbf{p}_0$ and $\textbf{q}_0$ implies $q_1 (k) = q_0 (k + 1) = p_0 (k + 1) = p_1 (k)$. 
	
	Now, assume that $q_0 (k) \le q_0 (k - 1) + j (k - 1) < q_0 (k + 1)$. Then, there are two cases to consider. The first is if $q_0 (k - 1) + j (k - 1) > q_0 (k)$, that is, the $(k - 1)$-th particle of $\textbf{q}_0$ attempts to jump to the right of $q_0 (k)$. Then, the same holds for $\textbf{p}_0$ (namely, $p_0 (k - 1) + j (k - 1) \ge p_0 (k)$), and so part 2(b) of the above coupling applied to both $\textbf{p}_0$ and $\textbf{q}_0$ yields $q_1 (k) = q_0 (k - 1) + j (k - 1) + 1 = p_0 (k - 1) +j (k - 1) + 1 = p_1 (k)$. The second is if $q_0 (k - 1) + j(k - 1) = q_0 (k)$, in which case part 2(b) of the above coupling (applied to $\textbf{q}_0$) implies $q_1 (k) = q_0 (k - 1) + j (k - 1) + 1 = q_0 (k) + 1 = p_0 (k)$. Since $p_0 (k) \le p_1 (k)$, this implies that $q_1 (k) \le p_1 (k)$, thereby confirming the claim when $i = k$.  
	
	Next suppose that $i > k$. If $\chi (i) = 1$ and $q_1 (i - 1) < q_0 (i)$, then the $i$-th particle of $\textbf{q}_0$ does not jump, and so $q_1 (i) = q_0 (i) = p_0 (i) \le p_1 (i)$. If either $\chi (i) = 0$ or $q_1 (i - 1) = q_0 (i)$, then either $\chi (i) = 0$ or $p_1 (i - 1) = p_0 (i)$ (since $p_0 (i) = q_0 (i)$ and $p_1 (i - 1) \ge q_1 (i - 1)$). Thus, the $i$-th particles of $\textbf{p}_0$ and $\textbf{q}_0$ both jump, and so $q_1 (i) = \max \big\{ q_0 (i) + j (i), q_0 (i + 1) \big\} = \max \big\{ p_0 (i) + j (i), p_0 (i + 1) \big\} = p_1 (i)$. This confirms the existence of the claimed coupling when $\textbf{p}_0$ and $\textbf{q}_0$ differ in one site. 
	
	If $\textbf{p}_0$ and $\textbf{q}_0$ do not differ in one site, then let $\textbf{p}_0 = \big( p_0 (-M), p_0 (1 - M), \ldots , p_0 (N) \big)$ and $\textbf{q}_0 = \big( q_0 (-U), q_1 (1 - U), \ldots , q_0 (V) \big)$. Further set $p_0 (i) = -\infty$ if $i < -M$; $p_0 (i) = \infty$ if $i > N$; $q_0 (i) = -\infty$ if $i < -U$; and $q_0 (i) = \infty$ if $i < V$. There then exists a doubly infinite sequence sequence $\cdots \le \textbf{r}_0^{(-1)} \le \textbf{r}_0^{(0)} \le \textbf{r}_0^{(1)}, \ldots $ of particle position sequences on $\mathbb{Z} \cup \{ -\infty, \infty \}$ such that the following two statements hold. 
	
	\begin{enumerate} 
	
	\item For each $i \in \mathbb{Z}$, $\textbf{r}_0^{(i)}$ and $\textbf{r}_0^{(i + 1)}$ are either equal or differ in one site. 
	
	\item We have that $\lim_{k \rightarrow \infty} \textbf{r}_0^{(-k)} = \textbf{q}_0$ and $\lim_{k \rightarrow \infty} \textbf{r}_0^{(k)} = \textbf{p}_0$. 
		
	\end{enumerate}

	For each $i \in \mathbb{Z}$, let $\textbf{r}_1^{(i)}$ denote the stochastic six-vertex model with initial data $\textbf{r}_0^{(i)}$, run for one time step. The above yields a coupling between $\big( \textbf{r}_1^{(i)}, \textbf{r}_1^{(i + 1)} \big)$ such that $\textbf{r}_1^{(i)} \le \textbf{r}^{(i + 1)}$ almost surely for each $i$; these induce couplings between $\big( \textbf{r}_1^{(-k)}, \textbf{r}_1^{(k)} \big)$ for each $k > 0$ such that $\textbf{r}_1^{(-k)} \le \textbf{r}_1^{(k)}$ almost surely. By letting $k$ tend to $\infty$, the compactness of the set of particle position sequences on $\mathbb{Z} \cup \{ -\infty, \infty \}$ yields a coupling between $\textbf{p}_1$ and $\textbf{q}_1$ such that $\textbf{p}_1 \ge \textbf{q}_1$ almost surely. 
\end{proof}

\subsection{Multi-Class Stochastic Six-Vertex Model} 

\label{HigherRank}

In this section we recall the definition of a multi-class variant of the stochastic six-vertex model, as it will be useful later to introduce couplings in \Cref{CouplingHigherRank}. This model arises as the spin $\frac{1}{2}$ case of the stochastic $U_q \big( \widehat{\mathfrak{sl}}_{n + 1} \big)$ vertex models, which were originally introduced in \cite{TSEA,QGM} and also studied in a number of recent works \cite{SM,CSVMST,ADFSVM,SSE}. Throughout this section, we fix integers $t \ge 0$ and $n \ge 1$. 

For simplicity, we will only define this model on a single row $\mathfrak{R} = \mathfrak{R}_t = \mathbb{Z} \times \{ t \} \subset \mathfrak{H}$, but its definition on other domains is entirely analogous. We first require the notion of a multi-class arrow configuration. Analogous to arrow configurations from \Cref{StochasticModel}, a \emph{multi-class arrow configuration} is a quadruple $(i_1, j_1; i_2, j_2)$ with $i_1, j_1, i_2, j_2 \in \{ 1, 2, \ldots , n \} \cup \{ \infty \}$ such that we have the equality of (unordered) sets $\{ i_1, j_1 \} = \{ i_2, j_2 \}$. 

We view the labels $\{ 1, 2, \ldots , n \} \cup \{ \infty \}$ as \emph{classes}, and we view the quadruple $(i_1, j_1; i_2, j_2)$ as an assignment of labeled arrows to a vertex $v \in \mathfrak{R}$. More specifically, $i_1$ and $j_1$ denote the classes of the incoming vertical and horizontal arrows at $v$, respectively, and $i_2$ and $j_2$ denote the classes of the outgoing vertical and horizontal arrows at $v$; see \Cref{sixvertexfigureclass}. The equality $\{ i_1, j_1 \} = \{ i_2, j_2 \}$ is again a form of arrow conservation. Thus, each edge along, entering, or exiting $\mathfrak{R}$ is assigned some class; it will be useful to view the class $\infty$ as denoting the absence of an arrow.

\begin{figure}[t]
	
	\begin{center}
		
		\begin{tikzpicture}[
		>=stealth,
		scale = .7
		]

		\draw[-, black] (-7.5, -1.6) -- (7.5, -1.6);
		\draw[-, black] (-7.5, -.8) -- (7.5, -.8);
		\draw[-, black] (-7.5, 0) -- (7.5, 0);
		\draw[-, black] (-7.5, 2) -- (7.5, 2);
		\draw[-, black] (-7.5, -1.6) -- (-7.5, 2);
		\draw[-, black] (7.5, -1.6) -- (7.5, 2);
		\draw[-, black] (-5, -1.6) -- (-5, 2);
		\draw[-, black] (5, -1.6) -- (5, 2);
		\draw[-, black] (-2.5, -1.6) -- (-2.5, 2);
		\draw[-, black] (2.5, -1.6) -- (2.5, 2);
		\draw[-, black] (0, -1.6) -- (0, 2);

		\draw[->, black, dashed] (-7.15, 1) -- (-6.35, 1); 
		\draw[->, black, dashed] (-6.15, 1) -- (-5.35, 1); 
		\draw[->, black, dashed] (-6.25, .1) -- (-6.25, .9); 
		\draw[->, black, dashed] (-6.25, 1.1) -- (-6.25, 1.9);

		\draw[->, black,  thick] (3.85, 1) -- (4.65, 1);
		\draw[->, black,  thick] (3.75, .1) -- (3.75, .9);
		\draw[->, black, dashed] (2.85, 1) -- (3.65, 1); 
		\draw[->, black, dashed] (3.75, 1.1) -- (3.75, 1.9);

		\draw[->, black,  thick] (-1.25, .1) -- (-1.25, .9);
		\draw[->, black,  thick] (-1.25, 1.1) -- (-1.25, 1.9);
		\draw[->, black,  dashed] (-2.15, 1) -- (-1.35, 1);
		\draw[->, black,  dashed] (-1.15, 1) -- (-.35, 1);

		\draw[->, black,  thick] (1.35, 1) -- (2.15, 1);
		\draw[->, black,  thick] (.35, 1) -- (1.15, 1);
		\draw[->, black,  dashed] (1.25, .1) -- (1.25, .9);
		\draw[->, black,  dashed] (1.25, 1.1) -- (1.25, 1.9);
		
		\draw[->, black,  thick] (6.25, 1.1) -- (6.25, 1.9);
		\draw[->, black,  thick] (5.35, 1) -- (6.15, 1);
		\draw[->, black, dashed] (6.35, 1) -- (7.15, 1);
		\draw[->, black, dashed] (6.25, .1) -- (6.25, .9); 
		
		\draw[->, black,  thick] (-3.75, 1.1) -- (-3.75, 1.9);
		\draw[->, black,  thick] (-3.75, .1) -- (-3.75, .9);
		\draw[->, black,  thick] (-3.65, 1) -- (-2.85, 1);
		\draw[->, black,  thick] (-4.65, 1) -- (-3.85, 1);
		
		\filldraw[fill=gray!50!white, draw=black] (-6.25, 1) circle [radius=.1] node [black,below=21] {$(j, j; j, j)$};
		\filldraw[fill=gray!50!white, draw=black] (-3.75, 1) circle [radius=.1] node [black,below=21] {$(i, i; i, i)$};
		\filldraw[fill=gray!50!white, draw=black] (-1.25, 1) circle [radius=.1] node [black,below=21] {$(i, j; i, j)$};
		\filldraw[fill=gray!50!white, draw=black] (1.25, 1) circle [radius=.1] node [black,below=21] {$(j, i; j, i)$};
		\filldraw[fill=gray!50!white, draw=black] (3.75, 1) circle [radius=.1] node [black,below=21] {$(i, j; j, i)$};
		\filldraw[fill=gray!50!white, draw=black] (6.25, 1) circle [radius=.1] node [black,below=21] {$(j, i; i, j)$};

		\filldraw[fill=gray!50!white, draw=black] (-6.25, .2) circle [radius=0] node [black,below=21] {$1$};
		\filldraw[fill=gray!50!white, draw=black] (-3.75, .2) circle [radius=0] node [black,below=21] {$1$};
		\filldraw[fill=gray!50!white, draw=black] (-1.25, .2) circle [radius=0] node [black,below=21] {$b_1$};
		\filldraw[fill=gray!50!white, draw=black] (1.25, .2) circle [radius=0] node [black,below=21] {$b_2$};
		\filldraw[fill=gray!50!white, draw=black] (3.75, .2) circle [radius=0] node [black,below=21] {$1 - b_1$};
		\filldraw[fill=gray!50!white, draw=black] (6.25, .2) circle [radius=0] node [black,below=21] {$1 - b_2$};

		\end{tikzpicture}
		
	\end{center}
	
	\caption{\label{sixvertexfigureclass} The top row in the chart depicts arrow configurations at vertices in the multi-class stochastic six-vertex model; the bottom row shows the corresponding probabilities. Here, the solid and dashed arrows correspond to classes $i, j \in \{ 1, \ldots , n \} \cup \{ \infty \}$, respectively, satisfying $i < j$. }
\end{figure}

A \emph{multi-class} (or \emph{higher rank}) \emph{six-vertex ensemble} on $\mathfrak{R}$ is an assignment of multi-class arrow configurations to each vertex of $\mathfrak{R}$ in such a way that neighboring arrow configurations are consistent; this means that there is an arrow of class $r \in \{ 1, 2, \ldots , n \} \cup \{ \infty \}$ to $(x + 1, t)$ in the configuration at $(x, t)$ if and only if there is one of class $r$ from $(x, t)$ in the configuration at $(x + 1, t)$. These arrows then form up-right directed paths (each of which has a class) on $\mathfrak{R}$ that emanate vertically from the line $y = t - 1$ and exit $\mathfrak{R}$ vertically through the line $y = t + 1$; see \Cref{classensemble} for an example.

A boundary condition is given by a sequence $\psi = \big( \psi (x) \big) \in \{ 0, 1, \ldots , n \}^{\mathbb{Z}}$; here, $\psi (x) = r \in \{ 1, 2, \ldots , n \}$ means that an arrow of class $r$ enters $\mathfrak{H}$ vertically through $(x, t - 1)$, and $\psi (x) = 0$ means that no arrow (or, equivalently, one of class $\infty$) enters through $(x, t - 1)$. We (temporarily) assume that $\sum_{x = - \infty}^{\infty} \psi (x) < \infty$, that is, the ensemble consists of only finitely many paths. 

We once again assign vertex weights to multi-class arrow configurations by setting 
\begin{flalign}
\label{wijweights}
w (i, i; i, i) = 1; \quad w (i, j; i, j) = b_1; \quad w (j, i; j, & i) = b_2; \quad w (i, j; j, i) = 1 - b_1; \quad w (j, i; i, j) = 1 - b_2,
\end{flalign}

\noindent for any $i, j \in \{ 1, 2, \ldots , n \} \cup \{ \infty \}$, where we assume that $i < j$ for the last four equalities in \eqref{wijweights}. If $(i_1, j_1; i_2, j_2)$ is not of the above form, then we set $w (i_1, j_1; i_2, j_2) = 0$; see \Cref{sixvertexfigureclass} for a depiction. 

\begin{rem}
	
\label{wclassw} 

Observe that the weights \eqref{wijweights} are similar to those \eqref{wi1j1i2j2} of the original stochastic six-vertex model. In particular, in the multi-class setting, lower class arrows view higher class ones as non-existent. 
\end{rem} 

As in \Cref{StochasticModel}, we view $w (i_1, j_1; i_2, j_2)$ as the weight of a vertex with arrow configuration $(i_1, j_1; i_2, j_2)$. Any multi-class six-vertex ensemble $\mathcal{E}$ on $\mathfrak{R}$ is assigned a weight $w (\mathcal{E})$ equal to the product of the weights of all of its vertices. For a finite boundary condition $\psi$, we define the \emph{multi-class stochastic six-vertex model} on $\mathfrak{R}$ as a probability measure on the set $\mathfrak{E} = \mathfrak{E}_{\psi; n}$ of multi-class six-vertex ensembles on $\mathfrak{R}$ with boundary condition $\psi$ such that the probability assigned to any $\mathcal{E} \in \mathfrak{E}$ is equal to $w (\mathcal{E})$. The stochasticity of the weights \eqref{wijweights} ensures that $\sum_{\mathcal{E} \in \mathfrak{E}} w (\mathcal{E}) = 1$.

\begin{rem} 
	
\label{mnclasses}

Let $m \in [1, n]$ be an integer. By \Cref{wclassw}, there is an \emph{concatenation procedure} that degenerates an $n$-class stochastic six-vertex model to an $m$-class one. Specifically, let $0 = j_0 < j_1 < j_2 < \cdots < j_m \le n$ be integers; and consider the multi-class six-vertex model described above but where each arrow of class $r \in [j_{i - 1} + 1, j_i]$ is replaced with one of class $i$ for every $i \in [1, m]$, and each arrow of class $r \in [j_m + 1, \infty]$ replaced by one of class $\infty$. This produces a measure on $m$-class six-vertex ensembles that, by \Cref{wclassw}, coincides with the $m$-class stochastic six-vertex model. 

\end{rem}

As in \Cref{Height}, we can associate a particle configuration $\eta = \big( \eta_y (x) \big)$ (where $(x, y)$ ranges over $\mathbb{Z} \times \{ t - 1, t \}$) to a multi-class six-vertex ensemble $\mathcal{E} \in \mathfrak{E}$. Specifically, if the arrow from $(x, y)$ to $(x, y + 1)$ is assigned some class $r \in \{ 1, 2, \ldots , n \}$ in $\mathcal{E}$, then we set $\eta_y (x) = r$. If this arrow is instead assigned class $\infty$, then we set $\eta_y (x) = 0$. Observe that $\psi (x) = \eta_{t - 1} (x)$ for each $x \in \mathbb{Z}$. 

Similarly to in \Cref{StochasticVertexLine}, we may view a vertical arrow of class $r \in \{ 1, 2, \ldots , n \}$ directed from $(x, s)$ to $(x, s + 1)$ in a multi-class six-vertex ensemble as a particle of class $r$ at location $x$ and time $s$. In this way, a multi-class stochastic six-vertex model on $\mathfrak{R}$ can be viewed as one time step of an interacting particle system on $\mathbb{Z}$ represented by a sequence $\textbf{p} = \big( \textbf{p}^{(r)} \big)$ of particle positions. Here, $r$ ranges over the classes $\{ 1, 2, \ldots , n \}$, and $\textbf{p}^{(r)} = \big( \textbf{p}_s^{(r)} \big) = \big( p_s^{(r)} (-M_r), p_s^{(r)} (1 - M_r), \ldots , p_s^{(r)} (N_r) \big)$ denotes the tagged positions of the class $r$ particles in the system, which are ordered so that $p_s^{(r)} (-M_r) < p_s^{(r)} (1 - M_r) < \cdots < p_s^{(r)} (N_r)$. We also set $p_t^{(r)} (i) = -\infty$ for $i < -M_r$ and $p_t^{(r)} (i) = \infty$ for $i > N_r$; it will further be convenient to set $\textbf{p}_s^{(0)}$ to be empty for $s \in \{ t - 1, t \}$. Observe that $\eta_{t - 1} (x) = \psi (x)$ is equal to $r \in \{ 1, 2, \ldots , n \}$ if $x \in \textbf{p}_0^{(r)}$ and is equal to $0$ otherwise.

\begin{figure}[t]
	
	\begin{center}
		
		\begin{tikzpicture}[
		>=stealth,
		scale = .85
		]
		
		\draw[-, black, dotted] (-3, 0) -- (12, 0); 
		\draw[-, black, dotted] (-2, -1) -- (-2, 1); 
		\draw[-, black, dotted] (-1, -1) -- (-1, 1); 
		\draw[-, black, dotted] (0, -1) -- (0, 1); 
		\draw[-, black, dotted] (1, -1) -- (1, 1); 
		\draw[-, black, dotted] (2, -1) -- (2, 1); 
		\draw[-, black, dotted] (3, -1) -- (3, 1); 
		\draw[-, black, dotted] (4, -1) -- (4, 1); 
		\draw[-, black, dotted] (5, -1) -- (5, 1); 
		\draw[-, black, dotted] (6, -1) -- (6, 1); 
		\draw[-, black, dotted] (7, -1) -- (7, 1); 
		\draw[-, black, dotted] (8, -1) -- (8, 1); 
		\draw[-, black, dotted] (9, -1) -- (9, 1); 
		\draw[-, black, dotted] (10, -1) -- (10, 1); 
		\draw[-, black, dotted] (11, -1) -- (11, 1); 
		
		\draw[->, black, dashed, thick] (-1, -1) -- (-1, -.1); 
		\draw[->, black, dashed, thick] (1, .1) -- (1, 1);

		\draw[->, black, dashed, thick] (-.9, 0) -- (-.1, 0); 
		\draw[->, black, dashed, thick] (.1, 0) -- (.9, 0); 
		
		\draw[->, black, very thick] (0, -1) -- (0, -.1); 
		\draw[->, black, very thick] (0, .1) -- (0, 1);

		\draw[->, black, very thick] (1, -1) -- (1, -.1); 
		\draw[->, black, very thick] (2, .1) -- (2, 1); 
		
		\draw[->, black, very thick] (1.1, 0) -- (1.9, 0);

		\draw[->, black, very thick] (4, -1) -- (4, -.1); 
		\draw[->, black, very thick] (4.1, 0) -- (4.9, 0); 
		\draw[->, black, very thick] (5.1, 0) -- (5.9, 0); 
		\draw[->, black, very thick] (6.1, 0) -- (6.9, 0); 
		\draw[->, black, very thick] (7, .1) -- (7, 1); 
		
		\draw[->, black, dashed, thick] (6, -1) -- (6, -.1); 
		\draw[->, black, dashed, thick] (6, .1) -- (6, 1);
		
		\draw[->, black, very thick] (7, -1) -- (7, -.1);
		\draw[->, black, very thick] (8, .1) -- (8, 1);
		
		\draw[->, black , very thick] (7.1, 0) -- (7.9, 0); 
		
		\draw[->, black, thick, dashed] (8, -1) -- (8, -.1);
		\draw[->, black, thick, dashed] (10, .1) -- (10, 1);
		
		\draw[->, black, thick, dashed] (8.1, 0) -- (8.9, 0);
		\draw[->, black, thick, dashed] (9.1, 0) -- (9.9, 0); 
		
		\filldraw[fill=gray!50!white, draw=black] (-2, 0) circle [radius=.1] node[below = 24]{$-5$};
		\filldraw[fill=gray!50!white, draw=black] (-1, 0) circle [radius=.1] node[below = 24]{$-4$};
		\filldraw[fill=gray!50!white, draw=black] (0, 0) circle [radius=.1] node[below = 24]{$-3$};
		\filldraw[fill=gray!50!white, draw=black] (1, 0) circle [radius=.1] node[below = 24]{$-2$};
		\filldraw[fill=gray!50!white, draw=black] (2, 0) circle [radius=.1] node[below = 24]{$-1$};
		\filldraw[fill=gray!50!white, draw=black] (3, 0) circle [radius=.1] node[below = 24]{$0$};
		\filldraw[fill=gray!50!white, draw=black] (4, 0) circle [radius=.1] node[below = 24]{$1$};
		\filldraw[fill=gray!50!white, draw=black] (5, 0) circle [radius=.1] node[below = 24]{$2$};
		\filldraw[fill=gray!50!white, draw=black] (6, 0) circle [radius=.1] node[below = 24]{$3$};
		\filldraw[fill=gray!50!white, draw=black] (7, 0) circle [radius=.1] node[below = 24]{$4$};
		\filldraw[fill=gray!50!white, draw=black] (8, 0) circle [radius=.1] node[below = 24]{$5$};
		\filldraw[fill=gray!50!white, draw=black] (9, 0) circle [radius=.1] node[below = 24]{$6$};
		\filldraw[fill=gray!50!white, draw=black] (10, 0) circle [radius=.1] node[below = 24]{$7$};
		\filldraw[fill=gray!50!white, draw=black] (11, 0) circle [radius=.1] node[below = 24]{$8$};

		\end{tikzpicture}
		
	\end{center}
	
	\caption{\label{classensemble} An example of a multi-class six-vertex ensemble on $\mathfrak{R}$ is depicted above in the case $n = 2$. The solid, dashed, and dotted edges correspond to classes $1$, $2$, and $\infty$, respectively. }
\end{figure}
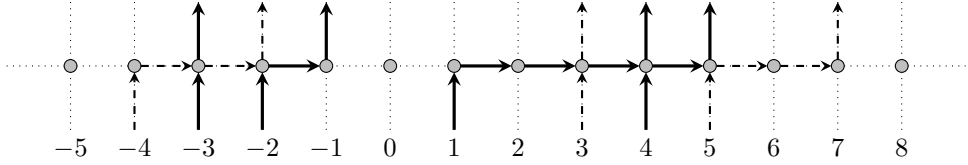

 Given the initial data $\textbf{p}_{t - 1} = \big( \textbf{p}_{t - 1}^{(r)} \big)$, we can sample $\textbf{p}_t = \big( \textbf{p}_t^{(r)} \big)$ by having the particles in $\textbf{p}_{t - 1}$ jump according to the following  stochastic dynamics. Observe under these dynamics that particles of lower class move first, ignoring all particles of higher class. 
 
 \begin{enumerate} 
 
\item For each $r \in \{ 1, 2, \ldots , n \}$ and $x \in \mathbb{Z}$, let $\big\{ \chi^{(r)} (x) \big\}$ and $\big\{ j^{(r)} (x) \big\}$ denote mutually independent random variables, with each $\chi^{(r)} (x)$ a $b_1$-Bernoulli $0-1$ random variable and each $j^{(r)} (x)$ chosen according to the $b_1$-geometric distribution (as in \Cref{sixvertexparticles1}).

\item Now let $r \in \{ 1, 2, \ldots , n \}$ denote a class; assume that $\textbf{p}_t^{(m)}$ has been defined for each integer $m \in [0, r - 1]$, and that $p_t^{(r)} (k - 1)$ has been defined for some integer $k \in [-M_r, N_r]$. Stated alternatively, the time $t$ location of each particle of class lower than $r$ has been set, as has that of the $(k - 1)$-th particle of class $r$. 

We will explain how to define the location of the $k$-th particle of class $r$, $p_t^{(r)} (k)$; let $x = p_{t - 1}^{(r)} (k)$ denote the particle's original position, and abbreviate $\chi = \chi^{(r)} (x)$ and $j = j^{(r)} (x)$. This particle will first choose to either move or stay, as follows. 

\begin{enumerate}
	\item It stays if there exists some $m \in \{ 1, 2, \ldots , r - 1 \}$ and $i \in \mathbb{Z}$ such that $p_{t - 1}^{(m)} (i) < x < p_t^{(m)} (i)$, that is, if there exists a particle of class lower than $r$ that jumped from the left to the right of $x$. See the $x = 3$ coordinate of \Cref{classensemble} for an example. 
	
	\item It moves if $x \in \bigcup_{m = 0}^{r - 1} \textbf{p}_t^{(m)} \cup \big\{ p_t (k - 1) \big\}$, that is, if a particle has already jumped to site $x$. See the $x = 4$ and $x = 5$ coordinates of \Cref{classensemble} for examples. 
	
	\item Excluding the above two events, it stays if $\chi = 1$ and moves if $\chi = 0$. Stated alternatively, it will otherwise stay or move with probabilities $b_1$ and $1 - b_1$, respectively. 
	
\end{enumerate}

\item If the particle stays, then set $p_t^{(r)} (k) = x$. 

\item If the particle moves, then define the following quantities. 

\begin{enumerate} 
	
	\item Let $U \in \mathbb{Z}$ denote the minimal integer with $U > x$ such that there exist integers $m \in [1, r - 1]$ and $i \in \mathbb{Z}$ such that $p_{t - 1}^{(m)} (i) = U < p_t^{(m)} (i)$. Stated alternatively, $U$ denotes the original location of the leftmost particle to the right of $x$ that is of class less than $r$ and that decided to move. 
	
	\item Let $V \in \mathbb{Z}$ denote the minimal integer with $V > x$ such that there exist $x_1, x_2, \ldots , x_j \in \mathbb{Z}$ satisfying $x < x_1 < x_2 < \cdots < x_j = V$ such, for each $h \in [1, j]$, there do not exist $m \in [1, r - 1]$ and $k \in \mathbb{Z}$ for which $p_{t - 1}^{(m)} (k) = x_h = p_t^{(m)} (k)$. Stated alternatively, $V$ denotes where the $k$-th particle of class $r$ would end if it jumped $j$ spaces to the right, skipping over any particles of lower class that decided to stay. 
	
\end{enumerate}

Now set $p_t^{(r)} (k) = \min \big\{ U, V, p_{t - 1} (k + 1) \big\}$. Stated alternatively, the particle jumps to the right according to a $b_2$-geometric distribution, skipping any particles of lower class that decided to stay, but neither passing the $(k + 1)$-th particle of class $r$ nor any particles of lower class that decided to move. 

\end{enumerate}

This provides a way of sampling a random multi-class particle position sequence $\textbf{p} = \big(\textbf{p}_s^{(r)} \big) = \big( \textbf{p}_s^{(r)} (k) \big)$ under some initial data $\psi = \big( \psi^{(r)} (x) \big)$. From this, one obtains a random multi-class particle configuration $\eta = \eta_s = (\eta_s) = \big( \eta_s (k) \big)$, which in turn determines a random multi-class six-vertex ensemble $\mathcal{E} \in \mathfrak{E}_{\psi; n}$. Denote the probability of selecting such an $\mathcal{E}$ under this procedure by $q (\mathcal{E})$. 

The following proposition, which is similar to \Cref{pw}, indicates that above induced measure on $\mathfrak{E}_{\psi; n}$ coincides with the multi-class six-vertex measure introduced above. Its proof is a quick consequence of the previous definitions and is therefore omitted. 

\begin{prop}
	
	\label{pclassw} 
	
	Under the above notation, we have that $q (\mathcal{E}) = w (\mathcal{E})$ for any multi-class six-vertex ensemble $\mathcal{E} \in \mathfrak{E}_{\psi; n}$. 	
	
\end{prop} 

Through a similar procedure as explained in \Cref{StochasticVertexLine}, one can also define the multi-class stochastic six-vertex model in the case of infinitely many particles; for brevity, we will not describe this in detail here. 

We conclude this section with the following lemma that bounds the speed of a tagged particle in a multi-class stochastic six-vertex model. 

\begin{lem} 

\label{xtxt1} 

Under the above notation, we have that $\mathbb{P} \big[ p_t^{(r)} (k) - p_{t - 1}^{(r)} (k) \ge v \big] \le b_2^{v - 1}$, for any integers $v, t \ge 0$; $r \in [1, n]$; and $k \in [-M_r, N_r]$.
\end{lem} 
	
\begin{proof}
	
	Let us abbreviate $X_s = p_s^{(r)} (k)$, for each $s \in \{ t - 1, t \}$. Consider the event $F = F_v$ on which $X_t - X_{t - 1} \ge v$. If $m \ge 0$ denotes the number of lower class particles (that is, with class between $1$ and $r - 1$) in the interval $[X_t + 1, X_t + v - 1]$ then, in order for $F_v$ to have occurred, each of these $m$ particles must not have moved and $X_t$ must have attempted to jump to the right by at least $v - m$ spaces. These events are independent; the former and latter occur with probabilities at most $b_1^m$ and $b_2^{v - m - 1}$, respectively. Thus, the lemma follows from the fact that $b_1 < b_2$. 
\end{proof}

\begin{rem}

\label{xtxt1cylinder} 

It can quickly by verified (through an entirely analogous proof) that \Cref{xtxt1} also holds for the multi-class stochastic six-vertex model on a cylinder, instead of on the upper half-plane. 
\end{rem}

\subsection{Coupling}

\label{CouplingHigherRank}

In this section we describe a coupling that will be used in the proofs of \Cref{sixvertexcylinderglobal} and \Cref{sixvertexlocalcylinder}; this will partly proceed through the multi-class stochastic six-vertex model of \Cref{HigherRank}. 

Fix an integer $n \ge 1$. Let us describe a coupling, which we refer to as the \emph{higher rank} (or \emph{multi-class}) \emph{coupling}, that couples $2n$ stochastic six-vertex models. More specifically, let $\eta_0^{(m)} = \big( \eta_0^{(m)} (x) \big) \in \{ 0, 1 \}^{\mathbb{Z}}$ and $\xi_0^{(m)} = \big( \xi_0^{(m)} (x) \big) \in \{ 0, 1 \}^{\mathbb{Z}}$, for each $m \in \{ 1, 2, \ldots n \}$, denote $2n$ particle configurations on $\mathbb{Z}$; assume that $\eta_0^{(i)} \le \eta_0^{(j)}$ and $\xi_0^{(i)} \le \xi_0^{(j)}$ for $1 \le i \le j \le n$. Denote the particle position sequences associated with $\eta_0^{(m)}$ and $\xi_0^{(m)}$ by $\textbf{p}_0^{(m)} = \big(  p_0^{(m)} (k) \big)$ and $\textbf{q}_0^{(m)} = \big( q_0^{(m)} (k) \big)$ for each $m \in [1, n]$, respectively. Let $\textbf{p}^{(m)} = \big(\textbf{p}_t^{(m)} \big) = \big( p_t^{(m)} (k) \big)$ and $\textbf{q}^{(m)} = \big( \textbf{q}_t^{(m)} \big) =  \big( q_t^{(m)} (k) \big)$ denote stochastic six-vertex models run with initial data $\textbf{p}_0^{(m)}$ and $\textbf{q}_0^{(m)}$, respectively. We will couple these models as follows. 

Suppose that each of the $\textbf{p}_{t - 1}^{(m)}$ and $\textbf{q}_{t - 1}^{(m)}$ have been jointly defined for some integer $t \ge 1$; we will then explain how to sample each of the $\textbf{p}_t^{(m)}$ and $\textbf{q}_t^{(m)}$. To that end, let the particle configurations associated with each $\textbf{p}_{t - 1}^{(m)}$ and $\textbf{q}_{t - 1}^{(m)}$ be $\eta_{t - 1}^{(m)}$ and $\xi_{t - 1}^{(m)}$, respectively; assume that $\eta_{t - 1}^{(i)} \le \eta_{t - 1}^{(j)}$ and $\xi_{t - 1}^{(i)} \le \xi_{t - 1}^{(j)}$ for each $1 \le i \le j \le n$. For convenience, define $\eta_{t - 1}^{(0)}$ and $\xi_{t - 1}^{(0)}$ to be the configurations with no particles, and define $\eta_{t - 1}^{(n + 1)}$ and $\xi_{t - 1}^{(n + 1)}$ to be the configurations that have a particle at every integer site. 

\begin{definition} 
	
\label{classparticles}

We will first assign a \emph{class}, which is an integer in $\{ 1, 2, \ldots , 2n \}$, to every particle in each of the $\eta_{t - 1}^{(m)}$ and $\xi_{t - 1}^{(m)}$, as follows. 

\begin{enumerate}
	
\item Any particle in $\big( \eta_{t - 1}^{(i)} \setminus \eta_{t - 1}^{(i - 1)}  \big) \cap \big( \xi_{t - 1}^{(j)} \setminus \xi_{t - 1}^{(j - 1)} \big)$ is assigned class $i + j - 1$, for each $i, j \in [1, n + 1]$ with $(i, j) \ne (n + 1, n + 1)$. 

\item Any particle in $\eta_{t - 1}^{(i)} \cap \eta_{t - 1}^{(j)}$ is assigned the same class in $\eta_{t - 1}^{(i)}$ and in $\eta_{t - 1}^{(j)}$, for each $i, j \in [1, n]$; the same statement holds if $\eta$ is replaced by $\xi$. 

\end{enumerate}

\end{definition}

\begin{example}
	
\label{etaxicoupling} 

If $n = 1$ and we abbreviate $\eta_{t - 1} = \eta_{t - 1}^{(1)}$ and $\xi_{t - 1} = \xi_{t - 1}^{(1)}$, then each particle of $\eta_{t - 1} \cap \xi_{t - 1}$ is of class $1$, while each particle in $\eta_{t - 1} \setminus \xi_{t - 1}$ or $\xi_{t - 1} \setminus \eta_{t - 1}$ is of class $2$. 

\end{example}

For any integers $m \in [1, n]$ and $r \in [1, 2n]$, let $\textbf{p}_{t - 1; r}^{(m)} = \big( p_{t - 1; r}^{(m)} (k) \big)$ and $\textbf{q}_{t - 1; r}^{(m)} = \big( p_{t - 1; r}^{(m)} (k) \big)$ denote the respective particle position sequences for the class $r$ particles in $\eta_{t - 1}^{(m)}$ and $\xi_{t - 1}^{(m)}$, respectively. Now we can define the following procedure that jointly couples the $2n$ stochastic six-vertex models $\textbf{p}^{(m)}$ and $\textbf{q}^{(m)}$ for one time step. 

\begin{definition} 
	
\label{couplinghigherrank} 

Sample each $\textbf{p}_{t; r}^{(m)}$ and $\textbf{q}_{t; r}^{(m)}$ (given $\textbf{p}_{t - 1; r}^{(m)}$ and $\textbf{q}_{t - 1; r}^{(m)}$) according to the same dynamics as explained in \Cref{HigherRank}, under which the random variables $\big\{ \chi^{(r)} (x) \big\}$ and $\big\{ j^{(r)} (k) \big\}$ are coupled across all $2n$ multi-class stochastic six-vertex models. Applying the $m = 1$ case of the concatenation procedure explained in \Cref{mnclasses} to each of these $2n$ multi-class models induces a coupling between the original $2n$ stochastic six-vertex particle configurations $\textbf{p}_t^{(m)}$ and $\textbf{q}_t^{(m)}$. We refer to this as the \emph{higher rank coupling}. 

\end{definition} 

\begin{rem}
	
	\label{etaxim} 
	
	In most cases of interest to us, we will couple $n + 1$ (and not $2n$) stochastic six-vertex models, given by $\eta_t$ and $\xi_t^{(m)}$ for $m \in \{ 1, 2, \ldots , n \}$. This can be viewed as a special case of the above framework, where we set $\eta_t^{(1)} = \eta_t$ and $\eta_t^{(m)}$ to be empty for $m \in \{ 2, 3, \ldots , n \}$. 
	
	Observe in this case that the class of a particle not in $\eta_{t - 1}$ determines which $\xi_{t - 1}^{(m)}$ it lies in. 
\end{rem}

It follows from \Cref{couplinghigherrank} that this coupling is \emph{attractive}, meaning that a pair of particles in two different stochastic six-vertex models that are coupled at time $t - 1$ remain coupled at time $t$. In particular, this implies that $\eta_t^{(i)} \le \eta_t^{(j)}$ and $\xi_t^{(i)} \le \xi_t^{(j)}$ if $1 \le i \le j \le n$, which allows us to repeat the sampling described above for larger values of $t$. This produces a simultaneous coupling of all $2n$ models $\eta^{(m)}$ and $\xi^{(m)}$ for all finite times.

\begin{rem}
	
	\label{higherrankclass}
	
	We should clarify that each particle is re-assigned its class at each time step under this procedure, that is, the class of a particle can change over time. However, we can still track the trajectory of each particle in any of the $2n$ models. It is quickly verified that the class of any particle tagged in this way is always non-increasing over time, and that it decreases only if the particle couples with another one.

\end{rem}

We can now quickly use the higher rank coupling to derive estimates on the behavior of discrepancies between two coupled stochastic six-vertex models. In particular, we establish the following proposition indicating that the speed of such discrepancies is bounded with high probability.

\begin{prop}
	
	\label{modelsequal}
	
	There exists a constant $c = c(b_2) > 0$ such that the following holds. Let $\textbf{\emph{p}} = \big( p_t (k) \big)$ and $\textbf{\emph{q}} = \big( q_t (k)\big)$ denote particle position sequences for two stochastic six-vertex models, which are coupled under the higher rank coupling of \Cref{couplinghigherrank}. 
	
	Let $M, N, R > 0$ denote integers. Assume that $p_0 (k) = q_0 (k)$ for each $k \in \mathbb{Z}$ such that either $p_0 (k) \in [-N, N]$ or $q_0 (k) \in [-N, N]$; also assume that $B = N - \frac{2M}{1 - b_2} - R > 0$. Let $E$ denote the event on which there exists some $(j, t) \in \mathbb{Z} \times [0, M]$ for which $p_t (j) \ne q_t (j)$ and either $p_t (j) \in [-R, N]$ or $q_t (j) \in [-R, N]$. Then, $\mathbb{P} [E] \le c^{-1} e^{-c (M + B)}$. 
\end{prop}

\begin{proof} 
	
Let $X_t$ and $Y_t$ denote the positions of the rightmost second-class particles in $\textbf{p}_t$ and $\textbf{q}_t$ (under the higher rank coupling) originally at time $0$ to the left of site $N + 1$, respectively. Since $\textbf{p}_0 \cap [-N, N] = \textbf{q}_0 \cap [-N, N]$, we have that $X_0, Y_0 \le -N$. Observe from \Cref{xtxt1} that 
\begin{flalign}
\label{xt1xtestimate}
\mathbb{P} [X_t - X_{t - 1} \ge k] \le b_2^{k - 1}, \quad \text{and} \quad \mathbb{P} [Y_t - Y_{t - 1} \ge k] \le b_2^{k - 1}, \quad \text{for each integer $k \ge 1$}. 
\end{flalign}

\noindent Since the jumps $\{ X_t - X_{t - 1} \}$ are mutually independent, a large deviations estimate for sums of independent geometric random variables yields a constant $c = c (b_2) > 0$ such that 
\begin{flalign}
\label{xt1xtestimate2}
\mathbb{P} \left[ X_t - X_0 \ge \displaystyle\frac{2t}{1 - b_2} + k \right]  \le e^{-c(t + k)}, \qquad \text{for any integers $t, k \ge 0$}, 
\end{flalign}

\noindent and similarly for $Y_t - Y_0$. 

Now, on the event $E$, we must have that $\max_{1 \le t \le M} \{ X_t - X_0, Y_t - Y_0 \} \ge N - R	$. In view of \eqref{xt1xtestimate2} and a union bound, this occurs with probability at most $2 M e^{-c (M + B)}$. The lemma then follows from decreasing $c$ if necessary. 
\end{proof}

\section{Classification of Translation-Invariant Stationary Measures} 

\label{StationaryTranslation}

In this section we classify the translation-invariant stationary measures for the stochastic six-vertex model; the analogous result for the ASEP was shown by Liggett in \cite{CEP}. We first state the classification result in \Cref{VertexMany}, and then we establish it in \Cref{ProofTranslation}.

\subsection{Translation-Invariant Stationary Measures}  

\label{VertexMany}

In this section we state a classification result for the translation-invariant and stationary measures of the stochastic six-vertex model, given by \Cref{translationstationary} below (which will be established in \Cref{ProofTranslation}). We begin by defining operators $\mathfrak{M}_t$ and $\mathfrak{S}_t$ that act on the space of probability measures on sets of particle configurations; the first essentially applies the stochastic six-vertex model for one time step, and the second is a shift operator. 

\begin{definition}
	
	\label{operatorm}
	
	Let $\mu$ denote a probability measure on the set of particle configurations on $\mathbb{Z}$, that is, on $\{ 0, 1 \}^{\mathbb{Z}}$. For each integer $t \ge 0$, let $\mathfrak{M}_t \mu$ denote the probability measure on $\{ 0, 1 \}^{\mathbb{Z}}$ defined as follows. Let $\eta_0 = \big( \eta_0 (x) \big) \in \{ 0, 1 \}^{\mathbb{Z}}$ denote a (random) particle configuration sampled from $\mu$, and run the stochastic six-vertex model $\eta = \big( \eta_s (x) \big)$ with initial data given by $\eta_0$. Then, $\mathfrak{M}_t \mu$ denotes the law of the particle configuration $\eta_t = \big( \eta_t (x) \big)$ of the model at time $t$.

\end{definition} 

\begin{definition} 
	
	\label{soperator} 
	
	For any $m \in \mathbb{Z}$ and particle configuration $\eta = \big( \eta (x) \big) \in \{ 0, 1 \}^{\mathbb{Z}}$, let $\mathfrak{S}_m \eta = \big( \mathfrak{S}_m \eta (x) \big) \in \{ 0, 1 \}^{\mathbb{Z}}$ denote the particle configuration obtained by shifting $\eta$ to the left by $m$, that is, by setting $\mathfrak{S}_m \eta (x) = \eta (x + m)$ for each $x \in \mathbb{Z}$. Then, $\mathfrak{S}_m$ induces an operator on the space of measures on $\{ 0, 1 \}^{\mathbb{Z}}$. 
\end{definition} 

\begin{rem} 

\label{smoperator} 

For any $t \in \mathbb{Z}_{\ge 0}$ and $m \in \mathbb{Z}$, we have that $\mathfrak{M}_t = (\mathfrak{M}_1)^t$ and $\mathfrak{S}_m = (\mathfrak{S}_1)^m$. Furthermore, $\mathfrak{M}_t$ and $\mathfrak{S}_m$ commute. 

\end{rem} 

Now we can define translation-invariant, stationary, and extremal measures for the stochastic six-vertex model.

\begin{definition} 
	
\label{textremali} 

Let $\mathscr{P} = \mathscr{P} \big( \{ 0, 1 \}^{\mathbb{Z}} \big)$ denote the space of probability measures on $\{ 0, 1 \}^{\mathbb{Z}}$. We say that an element $\mu \in \mathscr{P}$ is \emph{translation-invariant} if $\mathfrak{S}_m \mu = \mu$ for any $m \in \mathbb{Z}$, and that it is \emph{stationary} if $\mathfrak{M}_t \mu = \mu$ for any $t \in \mathbb{Z}_{\ge 0}$. Denote the sets of translation-invariant and stationary measures by $\mathscr{T} \subset \mathscr{P}$ and $\mathscr{S} \subset \mathscr{P}$, respectively. 

We furthermore call $\mu$ \emph{extremal} if, for any decomposition $\mu = p \mu_1 + (1 - p) \mu_2$ with $\mu_1, \mu_2 \in \mathscr{P}$ and $p \in (0, 1)$, we have that $\mu_1 = \mu = \mu_2$.

\end{definition} 

\begin{example} 
	
\label{zetarho} 

Fix $\rho \in [0, 1]$, and let $\Upsilon = \Upsilon^{(\rho)} \in \mathscr{P} \big( \{ 0, 1\}^{\mathbb{Z}} \big)$ denote the product $\rho$-Bernoulli measure on $\{ 0, 1 \}^{\mathbb{Z}}$; in particular, if $\eta = \big( \eta (x) \big) \in \{ 0, 1 \}^{\mathbb{Z}}$ is sampled under $\Upsilon$, then $\mathbb{E} \big[ \eta (x) \big] = \rho$ for any $x \in \mathbb{Z}$. Then, one can quickly verify that $\Upsilon$ is translation-invariant and extremal. 

Furthermore, observe that $\Upsilon$ defines the distribution of the particle configuration at time $t = 0$ associated with a six-vertex ensemble sampled from the infinite-volume translation-invariant Gibbs measure $\mu (\rho)$ from \Cref{Translation}. The translation-invariance (with respect to vertical shifts) of this measure implies that $\Upsilon$ is stationary. Thus, $\Upsilon \in \mathscr{T} \cap \mathscr{S}$ and is extremal.

\end{example} 

The following theorem indicates that the measures from \Cref{zetarho} constitute all extremal elements of $\mathscr{T} \cap \mathscr{S}$. Its proof will be given in \Cref{ProofTranslation} below.

\begin{thm}
	
	\label{translationstationary}
	
	If $\mu \in \mathscr{T} \cap \mathscr{S}$ is extremal, then $\mu = \Upsilon^{(\rho)}$ (from \Cref{zetarho}) for some $\rho \in [0, 1]$.
	 
\end{thm}

\subsection{Proof of \Cref{translationstationary}}

\label{ProofTranslation}  

In this section we establish \Cref{translationstationary}, whose proof will be partially based on the framework introduced by Liggett to establish Theorem 1.1 of \cite{CEP}. To that end, for any finite interval $I \subset \mathbb{Z}$ and two particle configurations $\eta = \big( \eta (x) \big) \in \{ 0, 1 \}^{\mathbb{Z}}$ and $\xi = \big( \xi (x) \big) \in \{ 0, 1 \}^{\mathbb{Z}}$, we first define
\begin{flalign} 
\label{rietaxi} 
\mathcal{R} ( I; \eta, \xi ) = 1 - \displaystyle\prod_{x, y \in I} \Big( 1 - \eta (x) \xi (y) \big( 1 - \eta (y) \big) \big(1 - \xi (x) \big) \Big).
\end{flalign} 

Observe in particular that either the $(x, y)$ or $(y, x)$ term of the product on the right side of \eqref{rietaxi} is equal to $0$ if and only if $\big( \eta (x), \eta (y); \xi (x), \xi (y) \big) = (1, 0; 0, 1)$ or $\big( \eta (x), \eta(y); \xi (x), \xi (y) \big) = (0, 1; 1, 0)$. Thus, $\mathcal{R} (I; \eta, \xi) = 0$ if $\eta$ and $\xi$ are \emph{ordered on $I$}, meaning that $\eta |_I \ge \xi |_I$ or $\eta |_I \le \xi |_I$, where $\eta |_I$ and $\xi |_I$ denote the restrictions of $\eta$ and $\xi$ to $I$, respectively; this is equivalent to the statement that either $\eta (i) \ge \xi (i)$ for each $i \in I$ or $\eta (i) \le \xi (i)$ for each $i \in I$. Otherwise, $\mathcal{R} (I; \eta, \xi) = 1$.

Now we can state the following proposition, which is similar to Lemma 2.2 of \cite{SLIS} (see also the proof of Lemma 3.1 in \cite{CEP}) and implies that suitably coupled stochastic six-vertex models are typically ordered on intervals that are not too large.

\begin{prop}
	
\label{guvpqsum} 

There exists a constant $C = C(b_1, b_2) > 1$ such that the following holds. Let $k, n, M, N \ge 1$ be integers, and let $\eta_0 = \big( \eta_0 (x) \big) \in \{ 0, 1 \}^{\mathbb{Z}}$ and $\xi_0^{(m)} = \big( \xi_0^{(m)} (x) \big) \in \{ 0, 1 \}^{\mathbb{Z}}$ for each $m \in \{ 1, 2, \ldots , n \}$ denote particle configurations on $\mathbb{Z}$. Assume that $\eta_0 (x) = 0 = \xi_0^{(m)} (x)$ for any $|x| > N$ and $m \in \{ 1, 2, \ldots , n \}$, and that $\xi_0^{(i)} \le \xi_0^{(j)}$ for any integers $1 \le i \le j \le n$. 

Let $\eta_t$ and $\xi_t^{(m)}$ denote stochastic six-vertex models with initial data $\eta_0$ and $\xi_0^{(m)}$, respectively, which are mutually coupled under the higher rank coupling of \Cref{couplinghigherrank} (recall \Cref{etaxim}). Also let $\mathcal{I}$ be a set of of pairwise disjoint intervals such that $\bigcup_{I \in \mathcal{I}} I = \mathbb{Z}$ and $|I| = k$ for each $I \in \mathcal{I}$. 

Then,
\begin{flalign*}
\displaystyle\sum_{t = 0}^{M - 1} \displaystyle\sum_{m = 1}^n \displaystyle\sum_{I \in \mathcal{I}} \mathbb{E} \Big[ \mathcal{R} \big( I; \eta_t, \xi_t^{(m)} \big) \Big] \le C^k n N.
\end{flalign*}
	
\end{prop} 

\begin{proof} 
	
In view of \Cref{couplinghigherrank} and the $m = 1$ case of the concatenation procedure for the multi-class stochastic six-vertex model described in \Cref{mnclasses}, it suffices to address the $n = 1$ case of this proposition. Thus, assume that $n = 1$, and abbreviate $\xi_t = \xi_t^{(1)}$; we will show that 
\begin{flalign}
\label{restimate1}
\displaystyle\sum_{t = 0}^{M - 1} \displaystyle\sum_{I \in \mathcal{I}} \mathbb{E} \big[ \mathcal{R} (I; \eta_t, \xi_t) \big] \le C^k N.
\end{flalign}

\noindent For any finite particle configurations $\eta = \big( \eta (x) \big) \in \{ 0, 1 \}^{\mathbb{Z}}$ and $\xi = \big( \xi (x) \big) \in \{ 0, 1 \}^{\mathbb{Z}}$, define 
\begin{flalign*}
f(\eta, \xi) = \displaystyle\sum_{x \in \mathbb{Z}} \big| \eta (x) - \xi (x) \big|. 
\end{flalign*}

Then, the attractivity of the higher rank coupling implies that $f (\eta_t, \xi_t)$ is non-increasing in $t$. Thus, since $\eta_t$ and $\xi_t$ each have at most $2N + 1 \le 3 N$ particles for any $t \ge 0$, it follows that
\begin{flalign}
\label{fetamximeta0xi0}
f(\eta_0, \xi_0) - f(\eta_M, \xi_M) \le f(\eta_0, \xi_0) \le 6N. 
\end{flalign}

Now, let $g (t) = f (\eta_{t - 1}, \xi_{t - 1}) - f(\eta_t, \xi_t) \ge 0$ for each integer $t \in [1, M]$. We claim that there exists a constant $c = c (b_1, b_2) > 0$ for which 
\begin{flalign}
\label{gtestimate} 
\mathbb{E} \big[ g(t) \big] \ge c^k \displaystyle\sum_{I \in \mathcal{I}} \mathbb{E}\big[ \mathcal{R} (I; \eta_{t - 1}, \xi_{t - 1}) \big].
\end{flalign}

\noindent Then, \eqref{restimate1} would follow from \eqref{fetamximeta0xi0} and \eqref{gtestimate}. 

To establish \eqref{gtestimate}, let $t \in [1, M]$ and $I = [u, v] \in \mathcal{I}$ denote an integer and interval, respectively; recall the assignment of classes from \Cref{classparticles} (and \Cref{etaxicoupling}, in our setting). Let $E (I; t)$ denote the event that there exists uncoupled (class $2$) particles in $\eta_{t - 1} \cap I$ and $\xi_{t - 1} \cap I$ that become coupled (first class) at time $t$. Then, $g(t) \ge \sum_{I \in \mathcal{I}} \textbf{1}_{E(I; t)}$, so that 
\begin{flalign}
\label{gtestimate2}
\mathbb{E} \big[ g(t) \big] \ge \displaystyle\sum_{I \in \mathcal{I}} \mathbb{P} \big[ E(I; t) \big].  
\end{flalign}

Now let us restrict to the event that $\mathcal{R} (I; \eta_{t - 1}, \xi_{t - 1}) = 1$. Then, there exist $x, y \in I$ such that $\eta_{t - 1} (x) = 1 = \xi_{t - 1} (y)$ and $\eta_{t - 1} (y) = 0 = \xi_{t - 1} (x)$; these particles at sites $x$ and $y$ are uncoupled (of class $2$) at time $t - 1$. We may assume that $x < y$ and that there is no $z \in [x + 1, y - 1]$ for which $\eta (z) \ne \xi (z)$. 

Define the event $F(I; t)$ on which the following hold. In the below, we recall the notation defining the dynamics of the multi-class stochastic six-vertex model from \Cref{HigherRank}. 

\begin{enumerate} 

\item For any $z < x$, we have that $j_t^{(1)} (z) + z \le x$.

\item We have that $\chi_t^{(1)} (z) = 1$ for each $z \in [x, y]$ and that $\chi_t^{(1)} (y + 1) = 0$. 

\item We have that $\chi_t^{(2)} (x) = 0 = \chi_t^{(2)} (y)$. 

\item Letting $h < k$ denote the number of coupled (class $1$) particles in $\eta$ and $\xi$ in the interval $[x + 1, y - 1]$, we have that $j_t^{(2)} (x) = y - x - h + 1$ and $j_t^{(2)} (y) = 1$.
	
\end{enumerate} 

Observe that, on $F(I; t)$, the uncoupled particles at locations $x$ and $y$ in $\eta_{t - 1}$ and $\xi_{t - 1}$, respectively, couple at time $t$. Indeed, the first and second conditions above imply that no first class particles in $\eta_{t - 1} \big( [x + 1, y - 1] \big) = \xi_{t - 1} \big( [x + 1, y - 1]\big)$ move; the second, third, and fourth then imply that the second class particles at $\eta_{t - 1} (x)$ and $\xi_{t - 1} (y)$ both move to site $y + 1$ at time $t$. 

Thus, $\textbf{1}_{E (I; t)} \ge \mathcal{R} (I; \eta_{t - 1}, \xi_{t - 1}) \textbf{1}_{F(I; t)}$. Furthermore, since the four events described above are mutually independent, it can quickly be deduced that $\mathbb{P} \big[ F(I; t) \big| \mathcal{R} (I; \eta_{t - 1}, \xi_{t - 1}) = 1 \big] \ge c^k$, for some constant $c = c(b_1, b_2) > 0$. Thus, 
\begin{flalign*}
\mathbb{P} \big[ E(I; t) \big] \ge \mathbb{P} \big[ F (I; t) \big| \mathcal{R} (I; \eta_{t - 1}, \xi_{t - 1}) = 1 \big] \mathbb{E} \big[ \mathcal{R} (I; \eta_{t - 1}, \xi_{t - 1}) \big] \ge c^k \mathbb{E} \big[ \mathcal{R} (I; \eta_{t - 1}, \xi_{t - 1}) \big],
\end{flalign*} 

\noindent which by \eqref{gtestimate2} implies \eqref{gtestimate}. As explained earlier, this yields the proposition. 
\end{proof}

We next require the following definition, which provides stationary, translation-invariant, and extremal measures for the dynamics of the coupled process of two stochastic six-vertex models; it is analogous to \Cref{textremali}. 

\begin{definition} 

\label{nustationarytranslation} 

Denote by $\mathscr{P} \big( \{ 0, 1 \}^{\mathbb{Z}} \times \{ 0, 1 \}^{\mathbb{Z}} \big)$ the space of probability measures on $\{ 0, 1 \}^{\mathbb{Z}} \times \{ 0, 1 \}^{\mathbb{Z}}$. Let $\nu \in \mathscr{P} \big( \{ 0, 1 \}^{\mathbb{Z}} \times \{ 0, 1 \}^{\mathbb{Z}} \big)$, and let  $(\eta_0, \xi_0)$ denote a pair of (random) particle configurations on $\mathbb{Z}$ sampled with respect to $\nu$. Let $\eta_s$ and $\xi_s$ denote stochastic six-vertex models with initial data $\eta_0$ and $\xi_0$, respectively, coupled under the higher rank coupling of \Cref{couplinghigherrank}. 

The measure $\nu$ is \emph{stationary with respect to the higher rank coupling} if the joint law of $(\eta_t, \xi_t)$ is equal to that of $(\eta_0, \xi_0)$ for any $t \in \mathbb{Z}_{\ge 0}$, and it is \emph{translation-invariant} if the joint law of $\big( \mathfrak{S}_m \eta_0, \mathfrak{S}_m \xi_0 \big)$ is equal to that of $(\eta_0, \xi_0)$ for any $m \in \mathbb{Z}$. We further call $\nu$ \emph{extremal} if, for any $\nu_1, \nu_2 \in \mathscr{P} \big( \{ 0, 1 \}^{\mathbb{Z}} \times \{ 0, 1 \}^{\mathbb{Z}} \big)$ and $p \in (0, 1)$ such that $\nu = p \nu_1 + (1 - p) \nu_2$, we have that $\nu_1 = \nu = \nu_2$. 

\end{definition} 

We next deduce the following corollary, which states that coupled translation-invariant stationary measures on particle configurations can be ordered in a certain sense. 

\begin{cor} 
	
\label{etaxiordered} 

Let $\nu$ denote a probability measure on $\{ 0, 1 \}^{\mathbb{Z}} \times \{ 0, 1 \}^{\mathbb{Z}}$ that is translation-invariant and stationary with respect to the higher rank coupling. If $(\eta, \xi)$ is a random pair of particle configurations sampled under $\nu$, then $\mathbb{P} \big[ \{ \eta \ge \xi \} \cup \{ \xi \ge \eta \} \big] = 1$.
\end{cor}

\begin{proof} 
	
	Denote $\eta_0 = \eta = \big( \eta_0 (x) \big) \in \{ 0, 1 \}^{\mathbb{Z}}$ and $\xi_0 = \xi = \big( \xi_0 (x) \big) \in \{ 0, 1 \}^{\mathbb{Z}}$. For any integer $N \ge 1$, define the particle configurations $\eta_0^{(N)} = \big( \eta_0^{(N)} (x) \big)$ and $\xi_0^{(N)} = \big( \xi_0^{(N)} (x) \big)$ by setting
	\begin{flalign}
	\label{eta0eta0m}
	\eta_0^{(N)} (x) = \eta_0 (x) \textbf{1}_{|x| \le N}; \qquad \xi_0^{(N)} (x) = \xi_0^{(N)} (x) \textbf{1}_{|x| \le N},
	\end{flalign}
	
	\noindent for any $x \in \mathbb{Z}$. Let $\mu_1^{(N)}$ and $\mu_2^{(N)}$ denote the laws of $\eta_0^{(N)}$ and $\mu_2^{(N)}$, respectively; they are coupled with $\mu_1$ and $\mu_2$, respectively, through \eqref{eta0eta0m}.
	
	Further let $\eta_t$, $\eta_t^{(N)}$, $\xi_t$, and $\xi_t^{(N)}$ denote the stochastic six-vertex models with initial data given by $\eta_0$, $\eta_0^{(N)}$, $\xi_0$, and $\xi_0^{(N)}$, respectively, coupled through \eqref{eta0eta0m} at time $0$ and through the $n = 2$ case of the higher rank coupling of \Cref{couplinghigherrank} for times $t \ge 1$. For any $t, N \in \mathbb{Z}_{\ge 1}$, let $\nu_t^{(N)}$ denote the joint law of $\big( \eta_t, \eta_t^{(N)} \xi_t, \xi_t^{(N)} \big) \in \big( \{ 0, 1 \}^{\mathbb{Z}} \big)^4$. Define $\gamma_t^{(N)} = \frac{1}{t} \sum_{s = 0}^{t - 1} \nu_s^{(N)}$, for any $t, N \in \mathbb{Z}_{\ge 1}$; observe that the marginal of $\gamma_t^{(N)}$ on the first and third components yields $\nu$, due to the stationarity of $\nu$ with respect to the higher rank coupling. 
	
	Now let $M \in \mathbb{Z}_{\ge 1}$; define $k = \big\lfloor (\log N)^{1 / 2} \big\rfloor$; and let $\mathcal{I}$ denote a partition of $\mathbb{Z}$ into (pairwise disjoint) intervals such that $|I| = k$ for each $I \in \mathcal{I}$. Then, \Cref{guvpqsum} yields the existence of a constant $C = C (b_1, b_2) > 1$ such that 
	\begin{flalign}
	\label{expectationgammamn}
	 \displaystyle\sum_{I \in \mathcal{I}} \mathbb{E}_{\gamma_M^{(N)}} \Big[ \mathcal{R} \big( I; \eta^{(N)}, \xi^{(N)} \big) \Big] \le C M^{-1} N^{3 / 2}. 
	\end{flalign}
	
	Furthermore \Cref{modelsequal} implies the existence of a constant $c = c(b_2) > 0$ such that the following holds if $M = \lfloor cN \rfloor$. With probability at least $1 - c^{-1} e^{-cM}$, we have that $\eta_t (x) = \eta_t^{(N)} (x)$ and $\xi_t (x) = \xi_t^{(N)} (x)$ for each $|x| \le M$ and $t \in [1, M]$. Combining this with \eqref{expectationgammamn}, the fact that $\mathcal{R} (I; \eta, \xi) \le 1$, and a Markov estimate yields after increasing $C$ if necessary that
	\begin{flalign*}
	\displaystyle\sum_{\substack{I \in \mathcal{I} \\ I \subseteq [-cN, cN]}} \mathbb{E}_{\nu} \big[ \mathcal{R} (I; \eta, \xi) \big] \le C N^{1 / 2} + 3 cN e^{-cM}.
	\end{flalign*}
	
	Thus, the translation-invariance of $\gamma_M$ and the facts that $k = \big\lfloor (\log N)^{1 / 2} \big\rfloor$ and that there are at most $\frac{3cN}{k}$ intervals $I \in \mathcal{I}$ such that $I \subseteq [-cN, cN]$ together imply that 
	\begin{flalign*} 
	\mathbb{P}_{\nu} \Big[ \mathcal{R} \big( [-m, m]; \eta, \xi \big) = 1 \Big] \le C (\log N)^{1 / 2} N^{-1 / 2} + C e^{- N / C},
	\end{flalign*}
	
	\noindent for any integer $m \le \frac{k}{3}$, after increasing $C = C(b_1, b_2)$ if necessary. Recalling the definition \eqref{rietaxi} of $\mathcal{R}$, it follows that 
	\begin{flalign*} 
	\mathbb{P}_{\nu} \big[ \{ \eta \ge \xi \} \cup \{ \xi \ge \eta \} \big] = \displaystyle\lim_{m \rightarrow \infty} \mathbb{P}_{\nu} \Big[ \mathcal{R} \big( [-m, m]; \eta, \xi \big) = 0 \Big] = 1,
	\end{flalign*}
	
	\noindent from which we deduce the corollary. 
\end{proof} 

Next, we require the following lemma; its proof is very similar to that of Lemma 2.3 of \cite{CEP} and is therefore omitted. 

\begin{lem}[{\cite[Lemma 2.3]{CEP}}]

\label{stationarytranslationcoupled}

Let $\mu_1, \mu_2 \in \mathscr{T} \cap \mathscr{S}$ denote two extremal, translation-invariant, stationary probability measures on $\{ 0, 1 \}^{\mathbb{Z}}$. Then, there exists a coupling $\nu \in \mathscr{P} \big( \{ 0, 1 \}^{\mathbb{Z}} \times \{ 0, 1 \}^{\mathbb{Z}} \big)$ between $\mu_1$ and $\mu_2$ that is extremal, translation-invariant, and stationary with respect to the higher rank coupling. 
\end{lem}

Now we can establish \Cref{translationstationary}.

\begin{proof}[Proof of \Cref{translationstationary}]
	
Let $\eta = \big( \eta (x) \big) \in \{ 0, 1 \}^{\mathbb{Z}}$ denote a (random) particle configuration on $\mathbb{Z}$ sampled under $\mu$, and set $\rho = \mathbb{E} \big[ \eta (0) \big]$. For any $\theta \in [0, 1]$, \Cref{stationarytranslationcoupled} implies the existence of a coupling $\nu = \nu_{\theta} \in \mathscr{P} \big( \{ 0, 1 \}^{\mathbb{Z}} \times \{ 0, 1 \}^{\mathbb{Z}} \big)$ between $\mu$ and $\Upsilon^{(\theta)}$ that is extremal, translation-invariant, and stationary with respect to the higher rank coupling. 

\Cref{etaxiordered} then yields $\mathbb{P}_{\nu} \big[ \{ \eta \ge \xi \} \cup \{ \xi \ge \eta \} \big] = 1$, if $(\eta, \xi)$ is sampled under $\nu$. The ergodicity (or extremality) of $\nu$ implies that either $\mathbb{P} [\eta \ge \xi] = 1$ or $\mathbb{P} [\xi \ge \eta] = 1$. Denoting $\eta = \big( \eta (x) \big)$ and $\xi = \big( \xi (x) \big)$, we have that $\mathbb{P} \big[ \eta (0) = 1 \big] = \rho$ and $\mathbb{P} \big[ \xi (0) =  1 \big] = \theta$. Thus, $\mathbb{P}_{\nu} \big[ \eta \ge \xi \big] = 1$ if $\rho > \theta$ and $\mathbb{P}_{\nu} \big[ \xi \ge \eta \big]$ if $\theta > \rho$. The continuity of the measures $\Upsilon^{(\theta)}$ in $\theta$ then yields that $\mu = \Upsilon^{(\rho)}$. 
\end{proof}

\section{Limit Shape for Double-Sided Bernoulli Boundary Data} 

\label{HeightLimit} 

In this section we analyze the limit shape for the stochastic six-vertex model on the discrete upper half-plane $\mathfrak{H} = \mathbb{Z} \times \mathbb{Z}_{\ge 0}$ with double-sided $(\theta, \rho)$-Bernoulli initial data. Given the couplings introduced in \Cref{CouplingHigherRank} and the classification of the translation-invariant stationary measures for the stochastic six-vertex model from \Cref{StationaryTranslation}, this will closely follow the framework introduced by Andjel-Vares in \cite{HEAS}, who established the analogous limit shape result for continuous-time, attractive interacting particle systems on $\mathbb{Z}$ (see also the work \cite{HLA} of Benassi-Fouque). 

In particular, we will first in \Cref{Limit1} state the limit shape result (given by \Cref{limitdouble} below) and establish it in the case $\theta \ge \rho$ as an application of the attractivity of the stochastic six-vertex model and the results of \cite{PTAEPSSVM}. Then, in \Cref{LimitDouble}, we reduce the $\theta < \rho$ case of \Cref{limitdouble} to a local limit statement, \Cref{localdouble}. Following \cite{HEAS}, we will then analyze properties of subsequential local limits of the stochastic six-vertex model in \Cref{LargeSmallLimitv}, which we will use to establish \Cref{localdouble} in \Cref{LimitDouble1}.

\subsection{Limit Shape When \texorpdfstring{$\theta \ge \rho$}{}}

\label{Limit1}

In this section we will state the limit shape theorem for the stochastic six-vertex model with $(\theta, \rho)$-Bernoulli initial data and outline its proof in the case $\theta \ge \rho$. Before doing so, however, let us define this boundary data more precisely. Throughout this section, we fix real numbers $\theta, \rho \in [0, 1]$. 

\begin{definition} 
	
\label{lambdarho} 

Let $\Upsilon = \Upsilon^{(\theta; \rho)} \in \mathscr{P} \big( \{ 0, 1 \}^{\mathbb{Z}} \big)$ denote the product measure on $\{ 0, 1 \}^{\mathbb{Z}}$, under which $\mathbb{P}_{\Upsilon} \big[ \eta (x) = 1 \big] = \theta \textbf{1}_{x \le 0} + \rho \textbf{1}_{x > 0}$ for each $x \in \mathbb{Z}$. 

\end{definition}

Thus, $\Upsilon$ denotes a \emph{double-sided $(\theta, \rho)$-Bernoulli} particle configuration, in which particles at or to the left of site $0$ are occupied with probability $\theta$, and particles to the right of $0$ are occupied with probability $\rho$. Observe in particular that $\Upsilon^{(\rho; \rho)} = \Upsilon^{(\rho)}$ (recall \Cref{zetarho}). 

Now we can state the limit shape theorem for the stochastic six-vertex model under $(\theta, \rho)$-Bernoulli initial data. Observe that it is analogous to \Cref{sixvertexcylinderglobal}, except that we only consider one value of $y$.

\begin{thm}
	
	\label{limitdouble}
	
	Fix a real number $\varepsilon > 0$, and let $\eta_0 = \big( \eta_0 (x) \big) \in \{ 0, 1 \}^{\mathbb{Z}}$ denote a (random) particle configuration sampled under the measure $\Upsilon^{(\theta; \rho)}$ from \Cref{lambdarho}. Further let $\eta_t$ denote the stochastic six-vertex model on $\mathbb{Z}$ with initial data $\eta_0$, and set $G_y (x) = G (x, y)$ to be the entropy solution to the equation \eqref{gxydefinition} on $\mathbb{R} \times [0, 1]$ with initial data given by $G_0 (x) = G (x, 0) = \theta \textbf{\emph{1}}_{x \le 0} + \rho \textbf{\emph{1}}_{x > 0}$ for each $x \in \mathbb{R}$. Then, 
	\begin{flalign*}
	\displaystyle\lim_{N \rightarrow \infty} \mathbb{P} \left[ \displaystyle\max_{- \frac{N}{\varepsilon} < X_1 < X_2 < \frac{N}{\varepsilon}} \bigg| \displaystyle\frac{1}{N} \displaystyle\sum_{x = X_1}^{X_2} \eta_N (x) - \displaystyle\int_{X_1 / N}^{X_2 / N} G_1 (x) dx \bigg| > \varepsilon \right] = 0. 
	\end{flalign*}  
	
\end{thm} 

\begin{rem}
	
	\label{g1sthetarho}
	
	Under the initial data $G_0 (x) =  \theta \textbf{1}_{x \le 0} + \rho \textbf{1}_{x > 0}$, the function $G_1 (x)$ is entirely explicit. Indeed, if $\theta < \rho$, then it is discontinuous (exhibiting a \emph{shock}) and is given by
	\begin{flalign}
	\label{gx1thetarhoidentity} 
	G(x, 1) = \theta, \quad \text{if $x \le \vartheta$;} \qquad G(x, 1) = \rho, \quad \text{if $x > \vartheta$}, 
	\end{flalign}
	
	\noindent where (recalling the function $\varphi$ from \eqref{kappafunction}) we have set 
	\begin{flalign}
	\label{vthetarho} 
	\vartheta = \vartheta (\theta, \rho) = \displaystyle\frac{\varphi (\rho) - \varphi (\theta)}{\rho - \theta}. 
	\end{flalign}
	
	\noindent If instead $\theta \ge \rho$, then it is continuous (exhibiting a \emph{rarefaction fan}) and given by 
	\begin{flalign}
	\label{gxidentity} 
	\begin{aligned} 
	& G_1 (x) = \theta \quad \text{if $x \le \varphi' (\theta)$}; \qquad G_1 (x) = \rho \quad \text{if $x \ge \varphi' (\rho)$}; \\
	& G_1 (x) = (\varphi')^{-1} (x) = \displaystyle\frac{\sqrt{\kappa x^{-1}} - 1}{\kappa - 1} \quad \text{if $\varphi' (\theta) < x < \varphi' (\rho)$}.
	\end{aligned} 
	\end{flalign}

\end{rem} 

We will address the $\theta < \rho$ and $\theta \ge \rho$ cases of \Cref{limitdouble} separately; see \Cref{thetarho1} (if $\theta \ge \rho$) and \Cref{thetarho2} (if $\theta < \rho$) below. In the case $\theta < \rho$, the proof of \Cref{limitdouble} will be given in detail in \Cref{LimitDouble}, \Cref{LargeSmallLimitv}, and \Cref{LimitDouble1}, proceeding along the lines of what was explained in Section 3 of \cite{HEAS}. In the case $\theta \ge \rho$, we only outline the proof since, given our analysis when $\theta < \rho$, that in the alternative case $\theta \ge \rho$ is very similar to that of Section 4 of \cite{HEAS}. 

There also exist at least two alternative potential routes towards \Cref{limitdouble} when $\theta \ge \rho$; the first would be through an asymptotic analysis of the Fredholm determinant identity given by Theorem 4.5 of \cite{CFSAEPSSVMCL}. The second would be by combining monotonicity results (see \Cref{mu1mu2} and \Cref{tmu1mu2m} below) with Theorem 1.6 of \cite{PTAEPSSVM}, which addresses the case $\rho = 0$. In fact, the latter route is the one that we will outline below.

To that end, we require the following definition that provides a partial ordering on $\mathscr{P} \big( \{ 0, 1 \}^{\mathbb{Z}} \big)$. 

\begin{definition} 
	
	\label{mu1mu2} 
	
	Let $\mu_1, \mu_2 \in \mathscr{P} \big( \{ 0, 1 \}^{\mathbb{Z}} \big)$. We say that $\mu_1 \ge \mu_2$ (or equivalently $\mu_2 \le \mu_1$) if there exists a coupling $\nu \in \mathscr{P} \big( \{ 0, 1 \}^{\mathbb{Z}} \times \{ 0, 1\}^{\mathbb{Z}} \big)$ with marginals $(\mu_1, \mu_2)$ such that the following holds. If $(\eta, \xi)$ is a pair of (random) particle configurations sampled with respect to $\nu$, then $\mathbb{P}_{\nu} [\eta \ge \xi] = 1$. 
\end{definition} 

\begin{rem}
	
	\label{tmu1mu2m}
	
	Observe that this ordering is transitive in that, if $\mu_1 \ge \mu_2$ and $\mu_2 \ge \mu_3$, then $\mu_1 \ge \mu_3$. Recalling the operators $\mathfrak{M}_t$ and $\mathfrak{S}_m$ from \Cref{operatorm} and \Cref{soperator}, respectively, $\mu_1 \ge \mu_2$ further implies that $\mathfrak{S}_m \mu_1 \ge \mathfrak{S}_m \mu_2$ for any $m \in \mathbb{Z}$. Additionally, the attractivity of the higher rank coupling from \Cref{couplinghigherrank} yields that $\mu_1 \ge \mu_2$ implies $\mathfrak{M}_t \mu_1 \ge \mathfrak{M}_t \mu_2$, for any $t \in \mathbb{Z}_{\ge 0}$. 
	
\end{rem}

The following lemma provides examples of this ordering. 

\begin{lem}
	
	\label{fgcouple} 
	
	Let $f: \mathbb{Z} \rightarrow [0, 1]$ and $g: \mathbb{Z} \rightarrow [0, 1]$ denote two functions. For $h \in \{ f, g \}$, let $\Upsilon^{(h)}$ denote the product measure on $\{ 0, 1 \}^{\mathbb{Z}}$, under which $\mathbb{P}_{\Upsilon^{(h)}} \big[ \eta (x) = 1 \big] = h (x)$ for each $x \in \mathbb{Z}$. If $f (x) \ge g(x)$ for each $x \in \mathbb{Z}$, then $\Upsilon^{(f)} \ge \Upsilon^{(g)}$. 
	
\end{lem} 

\begin{proof} 
	
	Define two (random) particle configurations $\eta = \big( \eta(x) \big) \in \{ 0, 1 \}^{\mathbb{Z}}$ and $\xi = \big( \xi (x) \big) \in \{ 0, 1 \}^{\mathbb{Z}}$ as follows. For each $x \in \mathbb{Z}$, let $U (x) \in [0, 1]$ denote a uniform random variable, and set $\eta (x) = \textbf{1}_{U (x) \le f(x)}$ and $\xi (x) = \textbf{1}_{U (x) \le g(x)}$. Then $\eta \ge \xi$ almost surely, and the marginal distributions of $\eta$ and $\xi$ are $\Upsilon^{(f)}$ and $\Upsilon^{(g)}$, respectively. Thus, $\Upsilon^{(f)} \ge \Upsilon^{(g)}$.
\end{proof} 

Now we can outline a proof of the following proposition.

\begin{prop}
	
	\label{thetarho1}
	
	If $\theta \ge \rho$, then \Cref{limitdouble} holds. 
	
\end{prop}

\begin{proof}[Proof (Outline)]
	
Define the (random) particle configurations $\omega_0 = \big( \omega_0 (x) \big)$, $\xi_0^{(\theta)} = \big( \xi_0^{(\theta)} (x) \big)$, $\xi_0^{(\rho)} = \big( \xi_0^{(\rho)} (x) \big)$, and $\zeta_0 = \big( \zeta_0 (x) \big)$ on $\{ 0, 1 \}^{\mathbb{Z}}$, sampled from the measures $\Upsilon^{(\theta; 0)}$, $\Upsilon^{(\theta)}$, $\Upsilon^{(\rho)}$, and $\Upsilon^{(1; \rho)}$, respectively. By \Cref{fgcouple}, we can couple these particle configurations with $\eta_0$ in such a way that $\omega_0 \le \eta_0 \le \xi_0^{(\theta)}$; $\xi_0^{(\rho)} \le \eta_0 \le \zeta_0$; and $\omega_0 \le \eta_0 \le \xi_0^{(\theta)} \le \xi_0^{(\rho)} \le \eta_0 \le \zeta_0$

Now let $\omega_t$, $\xi_t^{(\theta)}$, $\xi_t^{(\rho)}$, and $\zeta_t$ denote the stochastic six-vertex models with initial data $\omega_0$, $\xi_0^{(\theta)}$, $\xi_0^{(\rho)}$, and $\zeta_0$, respectively. Couple $\eta_t$ with $\big( \omega_t, \xi_t^{(\theta)} \big)$ and $\big( \xi_t^{(\rho)}, \zeta_t \big)$ under the $n = 2$ case of the higher rank coupling of \Cref{couplinghigherrank} (recall \Cref{etaxim}) so that $\omega_t \le \eta_t \le \xi_t^{(\theta)}$ and $\xi_t^{(\rho)} \le \eta_t \le \zeta_t$, for each $t \ge 0$. 

Next, let $A_t (x) = A (x, t)$ and $B_t (x) = B (x, t)$ denote the entropy solutions to the equation \eqref{gxydefinition} on $\mathbb{R} \times [0, 1]$ with initial data $A (x, 0) = \theta \textbf{1}_{x \le 0}$ and $B (x, 0) = \textbf{1}_{x \le 0} + \rho \textbf{1}_{x > 0}$, respectively. Then, Theorem 1.6 of \cite{PTAEPSSVM} implies that 
\begin{flalign}
\label{aomega}
\displaystyle\lim_{N \rightarrow \infty} \displaystyle\max_{-\frac{N}{\varepsilon} < X_1 < X_2 < \frac{N}{\varepsilon}} \mathbb{P} \Bigg[ \bigg| \displaystyle\frac{1}{N} \displaystyle\sum_{x = X_1}^{X_2} \omega_N (x) - \displaystyle\int_{X_1 / N}^{X_2/ N} A_1 (x) dx \bigg| > \varepsilon \Bigg] = 0.
\end{flalign} 

\noindent Furthermore, by combining Theorem 1.6 of \cite{PTAEPSSVM} with the duality of the stochastic six-vertex model that switches each particle with the absence of a particle, we obtain that   
\begin{flalign}
\label{bomega}
\displaystyle\lim_{N \rightarrow \infty} \displaystyle\max_{-\frac{N}{\varepsilon} < X_1 < X_2 < \frac{N}{\varepsilon}} \mathbb{P} \Bigg[ \bigg| \displaystyle\frac{1}{N} \displaystyle\sum_{x = X_1}^{X_2} \zeta_N (x) - \displaystyle\int_{X_1 / N}^{X_2/ N} B_1 (x) dx \bigg| > \varepsilon \Bigg] = 0.
\end{flalign} 

Combining \eqref{aomega}; \eqref{bomega}; the facts that $\omega_t \le \eta_t \le \xi_t^{(\theta)}$ and $\xi_t^{(\rho)} \le \eta_t \le \zeta_t$; the stationarity of $\xi^{(\theta)}$ and $\xi^{(\rho)}$; and a large deviations estimate for sums of independent Bernoulli random variables yields 
\begin{flalign}
\label{etaab}
\begin{aligned}
& \displaystyle\lim_{N \rightarrow \infty} \displaystyle\max_{-\frac{N}{\varepsilon} < X_1 < X_2 < \frac{N}{\varepsilon}} \mathbb{P} \Bigg[ \displaystyle\frac{1}{N} \displaystyle\sum_{x = X_1}^{X_2} \eta_N (x) - \displaystyle\int_{X_1 / N}^{X_2/ N} \Big( \max \big\{ A_1 (x), \rho \big\} \Big) dx < - \varepsilon \Bigg] = 0; \\
& \displaystyle\lim_{N \rightarrow \infty} \displaystyle\max_{-\frac{N}{\varepsilon} < X_1 < X_2 < \frac{N}{\varepsilon}} \mathbb{P} \Bigg[ \displaystyle\frac{1}{N} \displaystyle\sum_{x = X_1}^{X_2} \eta_N (x) - \displaystyle\int_{X_1 / N}^{X_2/ N} \Big( \min \big\{ B_1 (x), \theta \big\} \Big) dx  > \varepsilon \Bigg] = 0.
\end{aligned}
\end{flalign} 

Now the proposition follows from \eqref{etaab}, the fact that $\max \big\{ A_1 (x), \rho \big\} = G_1 (x) = \min \big\{ B_1 (x), \theta \big\}$ (due to the explicit descriptions for $A_1$, $B_1$, and $G_1$ given by \eqref{gxidentity}), and a union bound.
\end{proof}

\subsection{Reduction to a Local Limit Result}

\label{LimitDouble}

In view of \Cref{thetarho1}, it suffices to establish the following proposition in order to prove \Cref{limitdouble}. 

\begin{prop}
	
\label{thetarho2}

If $\theta < \rho$, then \Cref{limitdouble} holds.  
\end{prop}

To prove \Cref{thetarho2} we will establish the theorem below (in \Cref{LimitDouble1}), which is a statement about the fixed-time local statistics of the stochastic six-vertex model $\eta_t$ sampled under the $\theta < \rho$ case of \Cref{limitdouble}. In what follows, we recall the operators $\mathfrak{M}_t$ and $\mathfrak{S}_m$ from \Cref{operatorm} and \Cref{soperator}, respectively.  

\begin{thm} 
	
	\label{localdouble} 
	
	Adopt the notation of \Cref{limitdouble}, and assume that $\theta < \rho$. Fix some $u \in \mathbb{R}$ that is a continuity point of $G_1 (x)$, and set $\sigma = G (u, 1)$. Then, for any increasing sequence of integers $X_1 < X_2 < \cdots $ such that $\lim_{T \rightarrow \infty} T^{-1} X_T = u$, we have that
	\begin{flalign*}
	\displaystyle\lim_{T \rightarrow \infty} \mathfrak{S}_{X_T} \mathfrak{M}_T \Upsilon^{(\theta; \rho)} = \Upsilon^{(\sigma)}. 
	\end{flalign*}
	
\end{thm} 

It follows from \Cref{fgcouple} that $\mathfrak{S}_{-1} \Upsilon^{(\theta; \rho)} \le \Upsilon^{(\theta; \rho)} \le \mathfrak{S}_1 \Upsilon^{(\theta; \rho)}$ if $\theta < \rho$. Using this fact, we can establish \Cref{thetarho2} assuming \Cref{localdouble}.

\begin{proof}[Proof of \Cref{thetarho2} Assuming \Cref{localdouble}]
	
	By a union bound, the fact that $G_1 (x) \in [0, 1]$, and the fact that $\eta_N (x) \in \{ 0, 1 \}$, it suffices to show that 
	\begin{flalign*}
	\displaystyle\lim_{N \rightarrow \infty} \mathbb{P} \left[ \bigg| \displaystyle\frac{1}{N} \displaystyle\sum_{x = X_1}^{X_2} \eta_N (x) - \displaystyle\int_{x_1}^{x_2} G_1 (x) dx \bigg| > \varepsilon \right] = 0,
	\end{flalign*}
	
	\noindent for any $X_1, X_2 \in \big[ - \frac{N}{\varepsilon}, \frac{N}{\varepsilon} \big]$ such that $X_1 < X_2$. 
	
	Again by a union bound and the facts that $G_1 (x), \eta_N (x) \in [0, 1]$, we may assume that either $X_1 > \big( \vartheta + \frac{\varepsilon}{4} \big) N$ or that $X_2 < \big( \vartheta - \frac{\varepsilon}{4} \big) N$ (where we recall $\vartheta$ from \eqref{vthetarho}). Since these two cases are entirely analogous, let us assume the former. In this case, the explicit description \eqref{gx1thetarhoidentity} for $G (x, 1)$ implies that it suffices to show
	\begin{flalign}
	\label{etanxsumestimate2}
	\displaystyle\lim_{N \rightarrow \infty} \mathbb{P} \left[ \bigg|  \displaystyle\sum_{x = X_1}^{X_2} \eta_N (x) - \rho (X_2 - X_1) \bigg| > \varepsilon N \right] = 0,
	\end{flalign}
	
	Since $\Upsilon^{(\theta; \rho)} \le \Upsilon^{(\rho)}$ by \Cref{fgcouple}, a large deviations estimate for $0-1$ Bernoulli random variables and \Cref{tmu1mu2m} together imply that 
	\begin{flalign}
	\label{etanxsumestimate3}
	\displaystyle\lim_{N \rightarrow \infty} \mathbb{P} \left[ \displaystyle\sum_{x = X_1}^{X_2} \eta_N (x) - \rho (X_2 - X_1) > \varepsilon N \right] = 0,
	\end{flalign}
	
	\noindent so it remains to establish the lower bound on $\sum_{x = X_1}^{X_2} \eta_N (x)$. 
	
	To that end, let $m > 0$ be an arbitrary integer; for any measure $\mu \in \mathscr{P} \big( \{ 0, 1 \}^{\mathbb{Z}} \big)$, let $\mu |_{[-m, m]}$ denote the marginal of $\mu$ on $\{ 0, 1 \}^{[-m, m]}$. Then, \Cref{localdouble} and \Cref{fgcouple} together yield $\mathfrak{S}_{\lfloor xN \rfloor} \mathfrak{M}_N \Upsilon^{(\theta; \rho)} \big|_{[-m, m]} \ge \Upsilon^{(\rho - \varepsilon / 4)} |_{[-m, m]}$ for any $x \ge \vartheta + \frac{\varepsilon}{8}$, if $N$ is sufficiently large. Thus, it follows from the fact that $\mathfrak{S}_1 \Upsilon^{(\theta; \rho)} \ge \Upsilon^{(\theta; \rho)}$ and \Cref{tmu1mu2m} that $\mathfrak{S}_X \mathfrak{M}_N \Upsilon^{(\theta; \rho)} |_{[-m, m]} \ge \Upsilon^{(\rho - \varepsilon / 4)} |_{[-m, m]}$ for any $X \ge X_1$ (for sufficiently large $N$). 
	
	Now define $Y_0 < Y_1 < \cdots < Y_M$ by setting $Y_0 = X_1$, $Y_i = Y_{i - 1} + m$ for each $i \in [1, M]$, where $M$ is defined by imposing that $X_2 - m < Y_M \le X_2$. For each $i \in [1, M]$, also define the event 
	\begin{flalign*} 
	E_i = \Bigg\{ \displaystyle\sum_{x = Y_{i - 1} + 1}^{Y_i} \eta_N (x) < \bigg( \rho - \displaystyle\frac{\varepsilon}{2} \bigg)  m \Bigg\}. 
	\end{flalign*} 
	
	A large deviations bound for $0-1$ Bernoulli random variables then yields a constant $c = c (\varepsilon) > 0$ such that $\mathbb{P} \big[ E_i \big] \le c^{-1} e^{-cm}$ for each $i \in [1, M]$, if $N$ is sufficiently large. Letting $F$ denote the event on which at most $e^{-cm / 2} M$ of the $E_i$ hold, a Markov estimate then yields that $\mathbb{P} [F] < c^{-1} e^{-cm / 2}$. Hence,  
	\begin{flalign}
	\label{etansumestimate} 
	\mathbb{P} \left[ \displaystyle\sum_{x = X_1}^{X_2} \eta_N (x) < \bigg( \rho - \displaystyle\frac{\varepsilon}{2} \bigg) (1 - e^{-cm / 2}) m M \right] \le \mathbb{P} [F] < c^{-1} e^{-c m / 2}. 
	\end{flalign}
	
	Now \eqref{etanxsumestimate2} follows from combining \eqref{etanxsumestimate3} with the large $m$ limit of \eqref{etansumestimate}. 
\end{proof}

Thus, it remains to establish \Cref{localdouble}. To that end, we begin with the following lemma that establishes \Cref{localdouble} when $u$ is either sufficiently small or large. 

\begin{lem}
	
	\label{murholambdaleftright} 
	
	Fix real numbers $x \ge \frac{3}{1 - b_2}$ and $y < 0$. Let $X_1, X_2, \ldots $ and $Y_1, Y_2, \ldots $ denote increasing sequences of integers such that $\lim_{T \rightarrow \infty} T^{-1} X_T \le x$ and $\lim_{T \rightarrow \infty} T^{-1} Y_T \le y$. Then, 
	\begin{flalign} 
	\label{localxy} 
	\displaystyle\lim_{T \rightarrow \infty} \mathfrak{S}_{X_T} \mathfrak{M}_T \Upsilon^{(\theta; \rho)} = \Upsilon^{(\rho)}; \qquad \displaystyle\lim_{T \rightarrow \infty} \mathfrak{S}_{Y_T} \mathfrak{M}_T \Upsilon^{(\theta; \rho)} = \Upsilon^{(\theta)}.
	\end{flalign} 
	
\end{lem} 

\begin{proof}
	
Both statements of \eqref{localxy} follow in a similar way from \Cref{modelsequal}, so let us only establish the first; we may assume that $X_T \ge \big( \frac{2}{1 - b_2} + 1 \big) T$ for each $T > 0$. Let $\eta_0 = \big( \eta_0 (i) \big) \in \{ 0, 1 \}^{\mathbb{Z}}$ and $\xi_0 = \big( \xi_0 (i) \big) \in \{ 0, 1 \}^{\mathbb{Z}}$ denote two (random) particle configurations sampled from the measures $\Upsilon^{(\theta; \rho)}$ and $\Upsilon^{(\rho)}$, respectively, which are coupled in such a way that $\eta_0 (i) = \xi_0 (i)$ for each $i > 0$. 

Let $\eta = \big( \eta_t (i) \big)$ and $\xi = \big( \xi_t (i) \big)$ denote two stochastic six-vertex models with initial data given by $\eta_0$ and $\xi_0$, respectively, which are coupled under the higher rank coupling of \Cref{couplinghigherrank}. Then, \Cref{modelsequal} yields a constant $c = c(b_2) > 0$ such that the probability that $\eta_T$ and $\xi_T$ differ on the interval $\big[ X_T - \frac{T}{2}, X_T + \frac{T}{2} \big]$ is at most $c^{-1} e^{-cT}$, for each $T > 0$. This implies that $\lim_{T \rightarrow \infty} \mathfrak{S}_{X_T} \mathfrak{M}_T \Upsilon^{(\theta; \rho)} = \lim_{T \rightarrow \infty} \mathfrak{S}_{X_T} \mathfrak{M}_T \Upsilon^{(\rho)} = \Upsilon^{(\rho)}$, where the latter equality follows from the fact that $\Upsilon^{(\rho)} \in \mathscr{S} \cap \mathscr{T}$ is stationary and translation-invariant. 
\end{proof}

\subsection{Subsequential Local Limits} 

\label{LargeSmallLimitv}

In this section we derive several properties of subsequential local limits of the stochastic six-vertex model. We begin with the following lemma, whose statement (and proof) is similar to that of Lemma 3.1 in \cite{HEAS}; it states that certain subsequential local limits in our model are translation-invariant and stationary. 

\begin{lem}
	
	\label{limittranslationinvariant} 
	
	Let $\mu \in \mathscr{P} = \mathscr{P} \big( \{ 0, 1 \}^{\mathbb{Z}} \big)$ be such that either $\mathfrak{S}_1 \mu \ge \mu$ or $\mathfrak{S}_1 \mu \le \mu$. Then any increasing sequence of positive integers $T = (T_1, T_2, \ldots )$ admits an infinite subsequence $T_{n_1} < T_{n_2} < \cdots $ and a countable, dense subset of the real numbers $D = D (T) \subset \mathbb{R}$ such that the following holds. For any nonzero $v \in D$, there exists a stationary, translation-invariant measure $\mu = \mu_v \in \mathscr{T} \cap \mathscr{S}$ such that the three limits,
	\begin{flalign}
	\label{mulimits} 
	\begin{aligned}
	& \displaystyle\lim_{k \rightarrow \infty} \displaystyle\frac{1}{T_{n_k}} \displaystyle\sum_{j = 0}^{T_{n_k} - 1} \mathfrak{S}_{\lfloor v j \rfloor} \mathfrak{M}_j \mu; \qquad 
	\displaystyle\lim_{k \rightarrow \infty} \displaystyle\frac{1}{T_{n_k}} \displaystyle\sum_{j = 0}^{T_{n_k} - 1} \mathfrak{S}_{\lfloor v (j + 1) \rfloor} \mathfrak{M}_j \mu; \\
	& \displaystyle\lim_{k \rightarrow \infty} \displaystyle\frac{1}{\lfloor v T_{n_k} \rfloor} \displaystyle\sum_{j = 0}^{T_{n_k} - 1} \displaystyle\sum_{x = \lfloor vj \rfloor}^{\lfloor v (j + 1) \rfloor - 1} \mathfrak{S}_x \mathfrak{M}_j \mu,
	\end{aligned} 
	\end{flalign}
	
	\noindent all exist\footnote{To make sense of the last sum in \eqref{mulimits} when $v < 0$, we view $\sum_{x = i}^j f(x) = - \sum_{x = j}^i f(x)$ for any $i > j$.} and are equal to $\mu_v$. In particular, there exists a measure $\gamma = \gamma_v$ on $[0, 1]$ such that $\mu_v = \int_0^1 \Upsilon^{(a)} \gamma_v (da)$.
	
\end{lem} 	

\begin{proof}
	
	Let us assume that $\mathfrak{S}_1 \mu \ge \mu$, since the alternative case $\mu \ge \mathfrak{S}_1 \mu$ is entirely analogous. 
	
	Since $\mathscr{P}$ is compact, there exists an increasing subsequence $R = (R_1, R_2, \ldots ) \subseteq T$ and a dense, countable subset $V = V(T) \subset \mathbb{R}$ such that the limit
	\begin{flalign}
	\label{limitrmjmu}
	\displaystyle\lim_{k \rightarrow \infty} \displaystyle\frac{1}{R_k} \displaystyle\sum_{j = 0}^{R_k - 1} \mathfrak{S}_{\lfloor v j \rfloor} \mathfrak{M}_j \mu = \mu_v, 
	\end{flalign}
	
	\noindent exists for each $v \in V$. Since $\mathfrak{S}_1 \mu \ge \mu$, \Cref{smoperator} and \Cref{tmu1mu2m} together imply that $\mathfrak{S}_m \mathfrak{M}_t \mu = \mathfrak{M}_t \mathfrak{S}_m \mu \ge \mathfrak{M}_t \mu$ and $\mathfrak{S}_{-m} \mathfrak{M}_t \mu = \mathfrak{M}_t \mathfrak{S}_{-m} \mu \le \mathfrak{M}_t \mu$ for any nonnegative integers $m$ and $t$. Thus, $\mathfrak{S}_m \mu_v \ge \mu_v \ge \mathfrak{S}_{-m} \mu_v$ for any $m \in \mathbb{Z}_{\ge 0}$. 
	
	Now, for each $v \in V$, set $F(v) = \mathbb{E}_{\mu_v} \big[ \eta (0) \big]$, where $\eta = \big( \eta (x) \big) \in \{ 0, 1 \}^{\mathbb{Z}}$ is sampled under $\mu_v$. The fact that $\mathfrak{S}_m \mu \ge \mu$ for any $m > 0$ implies that $F$ is non-decreasing in $v$. Thus, there exists a countable dense subset $D \subset \mathbb{R}$ so that, for any $\varepsilon > 0$ and $v \in D$, there exist $v_1 , v_2 \in V$ such that $v_1 < v < v_2$ and $F (v_2) - F(v_1) < \varepsilon$. The domain of $F$ can then be extended to $V \cup D$ by continuity. 
	
	By again applying the compactness of $\mathscr{P}$, we may assume (after replacing $R$ with an increasing subsequence if necessary) that the limit \eqref{limitrmjmu} also exists for each $v \in D$. We claim that $\mu_v \in \mathscr{T} \cap \mathscr{S}$ for any $v \in D$. 
	
	To see that $\mu_v$ is translation-invariant, first observe that $\mathfrak{S}_1 \mu \ge \mu$ implies that $\mathbb{E}_{\mu_u} \big[ \eta (0) \big] \le \mathbb{E}_{\mu_v} \big[ \eta (0) \big]  \le \mathbb{E}_{\mu_w} \big[ \eta (0) \big]$ for any $u, v, w \in V \cup D$ such that $u \le v \le w$. Therefore, $F(v) = \mathbb{E}_{\mu_v} \big[ \eta (0) \big]$ for any $v \in D$. For any $v \in V \cup D$, also define the quantities 
	\begin{flalign*}
	A (v) = \displaystyle\lim_{m \rightarrow \infty} \mathbb{E}_{\mathfrak{S}_{-m} \mu_v} \big[ \eta (0) \big]; \qquad B (v) = \displaystyle\lim_{m \rightarrow \infty} \mathbb{E}_{\mathfrak{S}_m \mu_v} \big[ \eta (0) \big].
	\end{flalign*}
	
	These limits exist since $\mathbb{E}_{\mathfrak{S}_{-m} \mu_v} \big[ \eta (0) \big]$ and $\mathbb{E}_{\mathfrak{S}_m \mu_v} \big[ \eta (0) \big]$ are non-increasing and non-decreasing in $m > 0$, respectively. Again using the fact that $\mathfrak{S}_m \mu \ge \mu \ge \mathfrak{S}_{-m} \mu$ for any $m > 0$, it follows that $A(v) \le F(v) \le B(v)$ and $A(v), B(v) \in \big[ F(u), F(w) \big]$ for any $u, v, w \in V \cup D$ such that $u < v < w$. Since $\lim_{u \rightarrow v} F(u) = F(v)$ for any $v \in D$ (where the limit in $u$ is taken over $D$), it follows that $A(v) = F(v) = B (v)$ for any $v \in D$. 
	
	Since $\mathbb{E}_{\mu_v} \big[ \eta (x) \big]$ is non-decreasing in $x$, it follows that $\mathbb{E}_{\mu_v} \big[ \eta (x) \big]$ is constant over $x \in \mathbb{Z}$ for any $v \in D$. Combining this with the fact that $\mathfrak{S}_1 \mu_v \ge \mu_v \ge \mathfrak{S}_{-1} \mu_v$, we deduce that $\mu_v$ is translation-invariant. 

	Next, let us verify that $\mu_v$ is stationary; we only consider the case $v > 0$, since the alternative case $v < 0$ is entirely analogous. To that end, let $\mathfrak{N}_t = \mathfrak{S}_{\lfloor v t \rfloor} \mathfrak{M}_t$ and observe by \Cref{smoperator} that $\mathfrak{N}_{t + 1} \in \{ \mathfrak{N}_1 \mathfrak{N}_t, \mathfrak{N}_1 \mathfrak{N}_t \mathfrak{S}_1 \}$. Thus, since $\mathfrak{S}_1 \mu_v \ge \mu_v \ge \mathfrak{S}_{-1} \mu_v$, the definition \eqref{limitrmjmu} of $\mu_v$ and its translation-invariance together yield that  
	\begin{flalign*}
	\mathfrak{N}_1 \mu_v = \displaystyle\lim_{k \rightarrow \infty} \displaystyle\frac{1}{R_k} \displaystyle\sum_{j = 0}^{R_k - 1} \mathfrak{N}_1 \mathfrak{N}_j \mu \le \displaystyle\lim_{k \rightarrow \infty} \displaystyle\frac{1}{R_k} \displaystyle\sum_{j = 0}^{R_k - 1} \mathfrak{N}_{j + 1} \mu = \mu_v, 
	\end{flalign*}
	
	\noindent and 
	\begin{flalign*} 
	\mathfrak{N}_1 \mu_v = \mathfrak{N}_1 \mathfrak{S}_1 \mu_v = \displaystyle\lim_{k \rightarrow \infty} \displaystyle\frac{1}{R_k} \displaystyle\sum_{j = 0}^{R_k - 1} \mathfrak{N}_1 \mathfrak{N}_j \mathfrak{S}_1 \mu \ge \displaystyle\lim_{k \rightarrow \infty} \displaystyle\frac{1}{R_k} \displaystyle\sum_{j = 0}^{R_k - 1} \mathfrak{N}_{j + 1} \mu = \mu_v, 
	\end{flalign*}
	
	\noindent so $\mathfrak{N}_1 \mu_v = \mu_v$. This, with the translation-invariance of $\mu_v$ yields that $\mathfrak{M}_1 \mu_v = \mu_v$, meaning that $\mu_v$ is stationary. 
	
	Now, the compactness of $\mathscr{P}$ yields the existence of an increasing subsequence $(T_{n_1}, T_{n_2}, \ldots ) \subseteq R$ such that the three limits in \eqref{mulimits} all exist. We have verified that the first limit is equal to $\mu_v \in \mathscr{T} \cap \mathscr{S}$. Denoting the second and third limits in \eqref{mulimits} by $\mu_v^{(1)}$ and $\mu_v^{(2)}$ respectively, the fact that $\mathfrak{S}_1 \mu \ge \mu \ge \mathfrak{S}_{-1} \mu$ implies that $\mathfrak{S}_{- \lceil v \rceil - 1} \mu_v \le \mu_v^{(1)} \le \mathfrak{S}_{\lceil v \rceil + 1} \mu_v$ and similarly for $\mu_v^{(2)}$. Thus, $\mu_v^{(1)} = \mu_v = \mu_v^{(2)}$ since $\mu_v \in \mathscr{T}$. 
	
	The last statement of the lemma follows from the fact that $\mu \in \mathscr{T} \cap \mathscr{S}$ and the classification of the extremal elements of $\mathscr{T} \cap \mathscr{S}$ provided by \Cref{translationstationary}. 
\end{proof}

Next, we have the following lemma, which is the analog of Lemma 3.2 of \cite{HEAS}.

\begin{prop} 
	
\label{currentlimit} 

Adopt the notation of \Cref{limittranslationinvariant}, and relabel $T_{n_k} = T_k$ for each $k \in \mathbb{Z}_{\ge 1}$. For each $k$, let $\eta_{T_k} = \big( \eta_{T_k} (x) \big) \in \{ 0, 1 \}^{\mathbb{Z}}$ denote a (random) particle configuration sampled under the measure $\mathfrak{M}_{T_k} \Upsilon^{(\theta; \rho)}$. Then, 
\begin{flalign} 
\label{sumetalambda}
\displaystyle\lim_{k \rightarrow \infty} \displaystyle\frac{1}{T_k}  \displaystyle\sum_{x = \lfloor u T_k \rfloor}^{\lfloor w T_k \rfloor} \mathbb{E} \big[ \eta_{T_k} (x) \big] = \Lambda (w) - \Lambda (u),
\end{flalign} 

\noindent for any $u, w \in D$ satisfying $u < w$, where we have denoted 
\begin{flalign*}
\Lambda (v) = \displaystyle\int_0^1 \big( v a - \varphi (a) \big) \gamma_v (da), \qquad \text{for any $v \in D$.} 
\end{flalign*}

\end{prop} 

\begin{proof} 
	
	Let $\eta_0 = \big( \eta_0 (x) \big) \in \{ 0, 1 \}^{\mathbb{Z}}$ denote a (random) particle configuration on $\mathbb{Z}$, sampled under the measure $\Upsilon^{(\theta; \rho)}$, and let $\eta_t$ denote a stochastic six-vertex model with initial data $\eta_0$. For any integer $t \ge 0$, define 
	\begin{flalign}
	\label{atbt} 
	\mathfrak{A} (t) =  \displaystyle\sum_{x = \lfloor u t \rfloor}^{\lfloor w t \rfloor} \eta_t (x); \qquad \mathfrak{B} (t) = \mathfrak{A} (t + 1) - \mathfrak{A} (t).
	\end{flalign}
	
	Let $\mathcal{E}$ denote the (almost surely unique) six-vertex ensemble associated with the process $\eta_t$. Recall from \Cref{Translation} the vertical and horizontal indicators $\chi^{(v)} (x, y)$ and $\chi^{(h)} (x, y)$ denoting the indicators for an edge in $\mathcal{E}$ between $\big( (x, y), (x, y + 1) \big)$ and $\big( (x, y), (x + 1, y) \big)$, respectively. It is quickly verified that 
	\begin{flalign}
	\label{btctdt}
	\mathfrak{B} (t) & = \mathfrak{Y} (t) - \mathfrak{X} (t), 
	\end{flalign} 
	
	\noindent where 
	\begin{flalign*} 
	& \mathfrak{Y} (t) = \displaystyle\sum_{x = \lfloor wt \rfloor + 1}^{\lfloor w (t + 1) \rfloor} \chi^{(v)} (x, t)  - \displaystyle\sum_{x = \lfloor u t \rfloor}^{\lfloor u (t + 1) \rfloor - 1} \chi^{(v)} (x, t); \\
	& \mathfrak{X} (t) = \chi^{(h)} \big( \lfloor w (t + 1) \rfloor, t + 1 \big) -  \chi^{(h)} \big( \lfloor u (t + 1) \rfloor - 1, t + 1 \big).
	\end{flalign*}
	
	\noindent Applying the third limit in \eqref{mulimits} from \Cref{limittranslationinvariant}, we deduce that 
	\begin{flalign}
	\label{expectationtksum} 
	\displaystyle\lim_{k \rightarrow \infty} \displaystyle\frac{1}{T_k} \displaystyle\sum_{t = 0}^{T_k - 1} \mathfrak{Y} (t) = \displaystyle\int_0^1 wa \gamma_w (da) - \displaystyle\int_0^1 ua \gamma_u (da).
	\end{flalign}
	
	To estimate the sum of $\mathfrak{X} (t)$, fix $a \in [0, 1]$; let $\xi_0 = \big( \xi_0 (x) \big) \in \{ 0, 1 \}^{\mathbb{Z}}$ denote a random particle configuration sampled from the measure $\Upsilon^{(a)}$; let $\xi_t$ denote the stochastic six-vertex model with initial data $\xi_0$; and let $\mathcal{F} = \mathcal{F}^{(a)}$ denote the (almost surely unique) associated six-vertex ensemble. Then, the translation-invariance (with respect to both vertical and horizontal shifts) of the Gibbs measure $\mu (a)$ from \Cref{Translation} implies that the law of $\mathcal{F}$ is the restriction of $\mu(a)$ to the upper half-plane $\mathfrak{H}$. Thus, $\mathbb{E}_{\mathcal{F}} \big[ \chi^{(h)} (x, y) \big] = \varphi (a)$, for any  $(x, y) \in \mathbb{Z}^2$.
	
	It follows from the second limit in \eqref{mulimits} from \Cref{limittranslationinvariant} that 
	\begin{flalign}
	\label{expectationxtsum}
	\displaystyle\lim_{k \rightarrow \infty} \displaystyle\frac{1}{T_k} \displaystyle\sum_{t = 0}^{T_k - 1} \mathfrak{X} (t) = \displaystyle\int_0^1 \varphi (a) \gamma_w (da) - \displaystyle\int_0^1 \varphi (a) \gamma_u (da).
	\end{flalign}

	\noindent Thus, the proposition follows from \eqref{atbt}, \eqref{btctdt}, \eqref{expectationtksum}, and \eqref{expectationxtsum}. 
\end{proof}

\subsection{Proof of \Cref{localdouble}} 

\label{LimitDouble1}

In this section we establish \Cref{localdouble}. Before doing so, however, we require the following lemma, whose statement (and proof) is similar to that of Proposition 3.5 of \cite{HEAS}; it allows one to replace an averaged convergence of local statistics with non-averaged one, under a certain ordering assumption.

\begin{lem} 
	
	\label{localleftright} 
	
	Fix $x \in \mathbb{R}$ and $\sigma \in [0, 1]$, and let $\mu \in \mathscr{P} \big( \{ 0, 1 \}^{\mathbb{Z}} \big)$ be such that $\mathfrak{S}_1 \mu \ge \mu$ and $\mu \le \Upsilon^{(\sigma)}$. Suppose that, for any sequence of integers $X_1, X_2, \ldots $ tending to $\infty$ such that $\lim_{T \rightarrow \infty} T^{-1} X_T > x$ exists, we have  
	\begin{flalign}
	\label{limitsumxirho}
	\displaystyle\lim_{T \rightarrow \infty} \displaystyle\frac{1}{T} \displaystyle\sum_{j = 0}^{T - 1}  \mathfrak{S}_{X_j} \mathfrak{M}_j \mu = \Upsilon^{(\sigma)}.
	\end{flalign}
	
	\noindent Then, for any sequence $X_1, X_2, \ldots $ tending to $\infty$ such that $\lim_{T \rightarrow \infty} T^{-1} X_T > x$ exists, we have 
	\begin{flalign*} 
	\lim_{T \rightarrow \infty} \mathfrak{S}_{X_T} \mathfrak{M}_T \mu = \Upsilon^{(\sigma)}. 
	\end{flalign*} 
	
	\noindent The analogous statement holds if we instead assume that $\mu \ge \Upsilon^{(\sigma)}$ and $\lim_{T \rightarrow \infty} T^{-1} X_T < x$.
	
\end{lem} 

\begin{proof}
	
	We will only address the first case, when $\mu \le \Upsilon^{(\sigma)}$ and $\lim_{T \rightarrow \infty} T^{-1} X_T > x$, since the proof in the alternative case ($\mu \ge \Upsilon^{(\sigma)}$ and $\lim_{T \rightarrow \infty} T^{-1} X_T < x$) is entirely analogous. 
	
	To that end, assume to the contrary that there exist integers $Y_1 < Y_2 < \cdots $ such that the limit $\lim_{T \rightarrow \infty} T^{-1} Y_T = y > x$ exists and $\lim_{T \rightarrow \infty} \mathfrak{S}_{Y_T} \mathfrak{M}_T \mu \ne \Upsilon^{(\sigma)}$. Let $\eta_0 = \big( \eta_0 (i) \big) \in \{ 0, 1\}^{\mathbb{Z}}$ denote a (random) particle configuration on $\mathbb{Z}$ sampled according to the measure $\mu$, and let $\eta_t$ denote the stochastic six-vertex model run under initial data $\eta_0$. 
	
	Then the fact that $\mu \le \Upsilon^{(\sigma)}$ yields a positive real number $\varepsilon \in (0, y - x)$ and subsequence $Y_{n_1} < Y_{n_2} < \cdots $ of $(Y_1, Y_2, \ldots )$ such that, for each $k \in \mathbb{Z}_{\ge 1}$, there exists a uniformly bounded $j_k \in \mathbb{Z}$, for which $\mathbb{E} \big[ \eta_{T_{n_k}} (Y_{n_k} + j_k) \big] < \sigma - \varepsilon$. Since $\mathfrak{S}_1 \mu \ge \mu$, it follows that $\mathbb{E} \big[ \eta_{T_{n_k}} (Y_{n_k} + j) \big] < \sigma - \varepsilon$ for any $j \le \min_{k \ge 1} j_k$ (which is finite, in view of the uniform boundedness of the $\{ j_k \}$). For notational convenience, let us assume that $\min_{k \ge 1} j_k = 0$. 
	
	Letting $z = y - \frac{\varepsilon}{2} < x$, it follows for sufficiently large $N$ that
	\begin{flalign*}
	\displaystyle\sum_{i = \lfloor z T_{n_k} \rfloor}^{\lfloor y T_{n_k} \rfloor} \mathbb{E} \big[ \eta_{T_{n_k}} (i) \big] < (\sigma - \varepsilon) \big( \lfloor y T_{n_k} \rfloor - \lfloor z T_{n_k} \rfloor \big) + 1.  
	\end{flalign*} 
	
	Now, the fact that each horizontal edge of a six-vertex ensemble can accommodate at most one arrow implies that $\sum_{i = a}^b \eta_{t + 1} (i) \le \sum_{i = a}^b \eta_t (i) + 1$ for any integers $a, b, t$ with $a < b$ and $t \ge 0$. Thus, for sufficiently large $N$, any integer $k \ge 1$, and any real number $\delta > 0$, we have that 
	\begin{flalign*}
	\displaystyle\sum_{i = \lfloor z T \rfloor}^{\lfloor y T \rfloor} \mathbb{E} \big[ \eta_T (i) \big] & \le \displaystyle\sum_{i = \lfloor z T_{n_k} \rfloor}^{\lfloor y T_{n_k} \rfloor} \mathbb{E} \big[ \eta_T (i) \big] + \big( |y| + |z| \big) \delta T_{n_k} + 2 \\
	& < \displaystyle\sum_{i = \lfloor z T_{n_k} \rfloor}^{\lfloor y T_{n_k} \rfloor} \mathbb{E} \big[ \eta_{T_{n_k}} (i) \big] + \big( |y| + |z| + 1 \big) \delta T_{n_k} + 3 \\
	& \le \Big( (\sigma - \varepsilon) (y - z)  + \delta \big( |y| + |z| + 1 \big) \Big) T_{n_k} + 5,
	\end{flalign*} 
	
	\noindent for each $T \in \big[T_{n_k}, (1 + \delta) T_{n_k} \big]$. From this and the fact that $\mu \le \Upsilon^{(\sigma)}$, we deduce that 
	\begin{flalign}
	\label{sumetajiestimate}
	\begin{aligned}
	\displaystyle\sum_{j = 0}^{T - 1} \displaystyle\sum_{i = \lfloor zj \rfloor}^{\lfloor yj \rfloor} \mathbb{E} \big[ \eta_j (i) \big] & = \displaystyle\sum_{j = 0}^{T_{n_k} - 1} \displaystyle\sum_{i = \lfloor zj \rfloor}^{\lfloor yj \rfloor} \mathbb{E} \big[ \eta_j (i) \big] + \displaystyle\sum_{j = T_{n_k}}^{T - 1} \displaystyle\sum_{i = \lfloor zj \rfloor}^{\lfloor yj \rfloor} \mathbb{E} \big[ \eta_j (i) \big] \\
	& < \displaystyle\frac{(z - y) \sigma T_{n_k}^2}{2} + \big( |y| + |z| + 1 \big) T_{n_k} + \delta \big( (\sigma - \varepsilon) (y - z)  + \delta (y + 1) \big) T_{n_k}^2 + 5 \delta T_{n_k}, 
	\end{aligned}
	\end{flalign}
	
	\noindent for any $T \in \big[ T_{n_k}, (1 + \delta) T_{n_k} \big]$. 
	
	Next, \eqref{limitsumxirho} implies that $T^{-1} \sum_{j = 0}^{T - 1} \mathbb{E} \big[ \eta_j \big( \lfloor zj \rfloor \big) \big] > \sigma - \delta^2$, for sufficiently large $T$. Again using the fact that $\mu \le \mathfrak{S}_1 \mu \le \Upsilon^{(\rho)}$ yields that 
	\begin{flalign*} \displaystyle\sum_{j = 0}^{T - 1} \displaystyle\sum_{i = \lfloor zj \rfloor}^{\lfloor yj \rfloor} \mathbb{E} \big[ \eta_j (i) \big] > 	\displaystyle\frac{T^2 (z - y) (\sigma - \delta^2)}{2},
	\end{flalign*}	
	
	\noindent which contradicts \eqref{sumetajiestimate} with $T = \big\lfloor (1 + \delta) T_{n_k} \big\rfloor$, for sufficiently large $k$ and small $\delta = \delta (y, z, \varepsilon, \sigma) > 0$. This implies the lemma. 	
\end{proof}

Now we can establish \Cref{localdouble}. 
	
\begin{proof}[Proof of \Cref{localdouble}] 

Throughout this proof, set $\mu = \Upsilon^{(\theta; \rho)}$. In view of the explicit form for $G_y (x)$ given by \Cref{g1sthetarho}, it suffices to show that 
\begin{flalign}
\label{limitsxtmtmu}
\displaystyle\lim_{T \rightarrow \infty} \mathfrak{S}_{X_T} \mathfrak{M}_T \mu = \Upsilon^{(\theta)}, \quad \text{if $u < \vartheta$}; \qquad 
\displaystyle\lim_{T \rightarrow \infty} \mathfrak{S}_{X_T} \mathfrak{M}_T \mu = \Upsilon^{(\rho)}, \quad \text{if $u > \vartheta$}.
\end{flalign}

We begin by establishing the former identity in \eqref{limitsxtmtmu}. To that end, let $\textbf{t} = (t_1, t_2, \ldots )$ denote an arbitrary increasing sequence of positive integers. Then there exists an increasing subsequence $(T_1, T_2, \ldots ) \subseteq \textbf{t}$ and a dense subset $D = D(\textbf{t}) \subset \mathbb{R}$ such that the statements of \Cref{limittranslationinvariant} hold. For $v \in D$, let us first show that $\gamma_v = \delta_{\theta}$ if $v < \vartheta$.

To that end, fix real numbers $v, w \in D$ with $v < \vartheta$ and $w < 0$; then \Cref{murholambdaleftright} implies that $\gamma_w = \delta_{\theta}$. Furthermore, since $\Upsilon^{(\theta)} \le \mu \le \Upsilon^{(\rho)}$ (by \Cref{fgcouple}), \Cref{tmu1mu2m} implies the measure $\gamma_v$ is supported on the interval $[\theta, \rho]$. Thus, \Cref{currentlimit} yields 
\begin{flalign*} 
\Lambda (v) - \Lambda (w) = \displaystyle\lim_{k \rightarrow \infty} \displaystyle\frac{1}{T_k} \displaystyle\sum_{x = \lfloor w T_k \rfloor}^{\lfloor v T_k \rfloor} \mathbb{E} \big[ \eta_{T_k} (x) \big] \ge \theta (v - w).
\end{flalign*} 

\noindent Since $\gamma_w = \delta_{\theta}$, this means that 
\begin{flalign*} 
\displaystyle\int_{\theta}^{\rho} \big( v a - \varphi (a) \big) \gamma_v (da) - \theta w + \varphi (\theta)  \ge \theta (v - w).
\end{flalign*} 

\noindent Recalling the definition \eqref{vthetarho} of $\vartheta$ and the fact that $\varphi' (z) > 0 > \varphi'' (z)$ for $z \in [0, 1]$, we obtain 
\begin{flalign*} 
 v \displaystyle\int_{\theta}^{\rho} (a - \theta) \gamma_v (da) \ge \displaystyle\int_{\theta}^{\rho}  \big( \varphi (a) - \varphi (\theta) \big) \gamma_v (da) \ge \vartheta \displaystyle\int_{\theta}^{\rho} (a - \theta) \gamma_s (da).
\end{flalign*} 

\noindent Since $v < \vartheta$, it follows that $\int_{\theta}^{\rho} (a - \theta) \gamma_v (da) = 0$, which means that $\gamma_v = \delta_{\theta}$. So, 
\begin{flalign*}
\displaystyle\lim_{k \rightarrow \infty} \displaystyle\frac{1}{T_k} \displaystyle\sum_{j = 1}^{T_k} \mathfrak{S}_{\lfloor v j \rfloor} \mathfrak{M}_j \mu = \Upsilon^{(\theta)}, \qquad \text{if $v \in (-\infty, \vartheta) \cap D$}. 
\end{flalign*}

\noindent Since $D \subset \mathbb{R}$ is dense and $\mathfrak{S}_1 \mu \ge \mu \ge \mathfrak{S}_{-1} \mu$, it follows that 
\begin{flalign*}
\displaystyle\lim_{k \rightarrow \infty} \displaystyle\frac{1}{T_k} \displaystyle\sum_{j = 1}^{T_k} \mathfrak{S}_{\lfloor u j \rfloor} \mathfrak{M}_j \mu = \Upsilon^{(\theta)}, 
\end{flalign*}

\noindent for arbitrary $u \in \mathbb{R}$ (not necessarily in $D$) satisfying $u < \vartheta$. Since this holds along a subsequence $(T_1, T_2, \ldots )$ of an arbitrary increasing sequence $\textbf{t}$ of integers, we deduce 
\begin{flalign*}
\displaystyle\lim_{T \rightarrow \infty} \displaystyle\frac{1}{T} \displaystyle\sum_{j = 1}^T \mathfrak{S}_{\lfloor u j \rfloor} \mathfrak{M}_j \mu = \Upsilon^{(\theta)}, \qquad \text{for $u < \vartheta$}.  
\end{flalign*}

\noindent Then, since $\Upsilon^{(\theta)} \le \mu$ and $\mathfrak{S}_1 \mu \ge \mu$, we obtain from \Cref{localleftright} that 
\begin{flalign*}
\displaystyle\lim_{T \rightarrow \infty} \mathfrak{S}_{\lfloor u T \rfloor} \mathfrak{M}_T \mu = \Upsilon^{(\theta)}, \quad \text{for $u < \vartheta$},
\end{flalign*}

\noindent and so the first statement of \eqref{limitsxtmtmu} follows from the fact that $\mathfrak{S}_{-1} \mu \le \mu \le \mathfrak{S}_1 \mu$.

Next let us address the second statement of \eqref{limitsxtmtmu}; assume to the contrary that it is false. Observe that \Cref{tmu1mu2m}, the fact that $\mu \le \Upsilon^{(\rho)}$, and the translation-invariance and stationarity of $\Upsilon^{(\rho)}$ together imply that $\mathfrak{S}_{X_T} \mathfrak{M}_T \mu \le \Upsilon^{(\rho)}$. Therefore, there exists a real number $\varepsilon > 0$; a sequence $\textbf{t} = (t_1, t_2, \ldots )$ of positive integers tending to $\infty$; and a uniformly bounded sequence of integers $\textbf{j} = (j_1, j_2, \ldots )$ such that $\mathbb{E} \big[ \eta_{t_k} (X_{t_k} + j_k) \big] < \rho - \varepsilon$ for sufficiently large $k \in \mathbb{Z}_{\ge 1}$. Then there again exists an increasing subsequence $(T_1, T_2, \ldots ) \subseteq \textbf{t}$ and a dense subset $D = D(\textbf{t}) \subset \mathbb{R}$ such that the statements of \Cref{limittranslationinvariant} hold. 

Now fix real numbers $w, r \in D$ with $w < 0$ and $r > \frac{3}{1 - b_2}$, so that $\gamma_w = \delta_{\theta}$ and $\gamma_r = \delta_{\rho}$ by \Cref{murholambdaleftright}. Then, \Cref{currentlimit} yields
\begin{flalign}
\label{estimatewr} 
\displaystyle\lim_{k \rightarrow \infty} \displaystyle\frac{1}{T_k} \displaystyle\sum_{i = \lfloor w T_k \rfloor}^{\lfloor r T_k \rfloor} \mathbb{E} \big[ \eta_{T_k} (i) \big] = \Lambda (r) - \Lambda (w) = r \rho - \varphi (\rho) - w \theta + \varphi (\theta) = (r - \vartheta) \rho + (\vartheta - w) \theta. 
\end{flalign}

Since $\mathbb{E} \big[ \eta_{T_k} (X_{T_k} + j_k) \big] < \rho - \varepsilon$ for sufficiently large $k \in \mathbb{Z}_{\ge 1}$; since $\textbf{j}$ is uniformly bounded; and since $\mathfrak{S}_1 \mu \ge \mu$, we have for any $v < u$ that $\mathbb{E} \big[ \eta_{T_k} (x) \big] < \rho - \varepsilon$ whenever $x < v T_k$ and $k$ is sufficiently large. Next let $\varsigma \in (0, u - \vartheta)$ be some small positive number, and let $s \in D$ be such that $\vartheta - \varsigma < s < \vartheta$. Then the first statement of \eqref{limitsxtmtmu} implies that $\lim_{k \rightarrow \infty} \mathbb{E} \big[ \eta_{T_k} (x) \big] < \theta + \varsigma$ whenever $x \in [ w T_k, s T_k \big]$ and $k$ is sufficiently large. Fixing some $p \in (\vartheta, \vartheta + \varsigma) \cap D$ (which exists since $D \subset \mathbb{R}$ is dense), we deduce using $\mu \le \Upsilon^{(\rho)}$ that
\begin{flalign*}
\displaystyle\lim_{k \rightarrow \infty} \displaystyle\frac{1}{T_k} \displaystyle\sum_{i = \lfloor w T_k \rfloor}^{\lfloor r T_k \rfloor} \mathbb{E} \big[ \eta_{T_k} (i) \big] & = \displaystyle\lim_{k \rightarrow \infty} \displaystyle\frac{1}{T_k} \Bigg( \displaystyle\sum_{i = \lfloor w T_k \rfloor}^{\lfloor s T_k \rfloor} \mathbb{E} \big[ \eta_{T_k} (i) \big] + \displaystyle\sum_{i = \lfloor s T_k \rfloor}^{\lfloor p T_k \rfloor} \mathbb{E} \big[ \eta_{T_k} (i) \big] \\
& \qquad \qquad \qquad + \displaystyle\sum_{i = \lfloor p T_k \rfloor}^{\lfloor v T_k \rfloor} \mathbb{E} \big[ \eta_{T_k} (i) \big] + \displaystyle\sum_{i = \lfloor v T_k \rfloor}^{\lfloor r T_k \rfloor} \mathbb{E} \big[ \eta_{T_k} (i) \big] \Big) \\
& < (\theta + \varsigma) (\vartheta - w) + 2 \varsigma + (v - \vartheta) (\rho - \varepsilon) + (r - v) \rho \\
& < (r - \vartheta) \rho + (\vartheta - w) \theta, 
\end{flalign*}

\noindent for $v \in (\vartheta + \varsigma, u)$ and sufficiently small $\varsigma = \varsigma (\varepsilon, w, \vartheta, v, u) > 0$. This contradicts \eqref{estimatewr}, and so the second statement of \eqref{limitsxtmtmu} also holds. 
\end{proof}

\section{Proof of \Cref{sixvertexcylinderglobal}} 

\label{LimitGeneral} 

In this section we establish \Cref{sixvertexcylinderglobal}, which provides the limit shape for the stochastic six-vertex model on the cylinder with arbitrary boundary data. We first in \Cref{LimitPlane} analyze the limit shape of the stochastic six-vertex model on the upper half-plane $\mathfrak{H} = \mathbb{Z} \times \mathbb{Z}_{\ge 0}$; given the content of the previous sections, this will follow the framework introduced by Bahadoran-Guiol-Ravishankar-Saada \cite{HAPAE}, who provided a procedure to understand the limit profile of attractive interacting particle systems under arbitrary initial data, assuming that one has access to it under double-sided Bernoulli initial data. Then, in \Cref{LimitCylinder} we apply these results to establish \Cref{sixvertexcylinderglobal} on the cylinder.

\subsection{Limit Shape on the Upper Half-Plane} 

\label{LimitPlane}

In this section we establish the fixed-time analog of \Cref{sixvertexcylinderglobal} on the upper half-plane, given as follows. 

\begin{thm}
	
	\label{limitgeneral} 
	
	Fix real numbers $0 < b_1 < b_2 < 1$ and $\lambda, \varepsilon > 0$, and a measurable function $\Psi: \mathbb{R} \rightarrow [0, 1]$ supported on $[0, 1]$. Let $G_y (x) = G (x, y)$ denote the entropy solution to \eqref{gxydefinition} on $\mathbb{R} \times [0, \lambda]$, with initial data given by $G_0 (x) = G(x, 0) = \Psi (x)$ for each $x \in \mathbb{R}$. 
	
	For each $N \in \mathbb{Z}_{\ge 1}$, let $\psi = \psi^{(N)} = \big( \psi (x) \big) = \big( \psi^{(N)} (x) \big)_{x \in \mathbb{Z}} \in \{ 0, 1 \}^{\mathbb{Z}}$ denote a boundary condition such that $\psi (x) = 0$ if $x \notin [1, N]$, and assume that \eqref{limitpsi} holds. Furthermore, for each $N$, let $\eta = \eta^{(N)} = \big( \eta_y (x) \big) = \big( \eta_y^{(N)} (x) \big)$ denote a stochastic six-vertex model with initial data given by $\eta_0 = \eta_0^{(N)} = \psi^{(N)}$. Then, 
	\begin{flalign*}
	\displaystyle\lim_{N \rightarrow \infty} \mathbb{P} \left[ \displaystyle\max_{- \frac{N}{\varepsilon} \le X_1 \le X_2 \le \frac{N}{\varepsilon}} \bigg| \displaystyle\frac{1}{N} \displaystyle\sum_{x = X_1}^{X_2} \eta_{\lfloor \lambda N \rfloor} (x) - \displaystyle\int_{X_1 / N}^{X_2 / N} G_{\lambda} (x) dx \bigg| > \varepsilon \right] = 0. 
	\end{flalign*}  
	
\end{thm} 	

The proof of \Cref{limitgeneral} will proceed similarly to what was explained in Section 3 of \cite{HAPAE}. Before continuing, we require some properties of the entropy solution of the conservation law \eqref{gxydefinition}, to which end we require some additional notation. We begin with the following definition, which provides a semi-group operator for the solution to \eqref{gxydefinition}. 

\begin{definition} 
	
\label{ptoperator} 

Fix a compactly supported, measurable function $u: \mathbb{R} \rightarrow [0, 1]$. For any real number $t \ge 0$, define $\mathfrak{P}_t u (x) = G (x, t)$ for each $x \in \mathbb{R}$, where $G$ denotes the entropy solution of the equation \eqref{gxydefinition} on $\mathbb{R} \times \mathbb{R}_{\ge 0}$, with initial data $G (x, 0) = u (x)$.  

Since \eqref{gxydefinition} admits unique entropy solutions, we have that $\mathfrak{P}_s \mathfrak{P}_t = \mathfrak{P}_{s + t}$ for any $s, t \ge 0$.

\end{definition} 

Next, for any compactly supported, measurable functions $u, v: \mathbb{R} \rightarrow [0, 1]$, define for any $x \in \mathbb{R}$ 
\begin{flalign*}
I (x; u) = \displaystyle\int_{-\infty}^x u(y) dy; \qquad I (x; v) = \displaystyle\int_{-\infty}^x v(y) dy; \qquad \Delta (u, v) = \displaystyle\sup_{x \in \mathbb{R}} \big| I (x; u) - I (x; v) \big|. 
\end{flalign*}

Now we can state the following lemma, which appears as Theorem 3.1 of \cite{HAPAE}. The first part of the below statement is given by Proposition 2.3.6 of \cite{SCL}; the second part follows from the fact that the function $I = I (x; u)$ satisfies the Hamilton-Jacobi equation $\partial_t I = \varphi (\partial_x I)$ and a comparison principle (see Theorem 5.3 of \cite{VSE}) for such equations. 

\begin{lem}[{\cite[Theorem 3.1]{HAPAE}}]
	
\label{uvequation} 
	
For any compactly supported, measurable functions $u, v: \mathbb{R} \rightarrow [0, 1]$ and real numbers $t \ge 0$ and $c > \kappa \ge \max_{z \in [0, 1]} \big| \varphi' (z) \big|$, the following statements hold. 

\begin{enumerate}
	
	\item If $u (x) = v (x)$ for each $x \in [A, B]$ and some real numbers $A < B$, then $\mathfrak{P}_t u (x) = \mathfrak{P}_t v (x)$ for each $x \in [A + ct, B - ct]$.

	\item We have that $\Delta (\mathfrak{P}_t u, \mathfrak{P}_t v) \le \Delta (u, v)$. 
\end{enumerate} 

\end{lem} 

\begin{rem}
	
\label{uvequation2} 

The first part of \Cref{uvequation} also applies to solutions of \eqref{gxydefinition} on the torus $\mathbb{T}$. More specifically, suppose that $G(x, t)$ and $H(x, t)$ are two solutions of \eqref{gxydefinition} on $\mathbb{T} \times \mathbb{R}_{\ge 0}$ such that $G(x, 0) = H(x, 0)$ for each $x \in [A, B]$ and some $A, B \in \mathbb{T}$. If $d(A, B) \ge 2ct$, (where $d(A, B)$ denotes the distance between $A$ and $B$ on $\mathbb{T}$), then $G(x, t) = H(x, t)$ for each $x \in [A + ct, B - ct]$.

\end{rem} 

It will be useful to have discrete analogs of the two properties from \Cref{uvequation} for the stochastic six-vertex model; \Cref{modelsequal} is such an analog of the first property. The following lemma provides a discrete analog of the second. In the below, for any positive integer $N \ge 1$; compactly supported, measurable function $f: \mathbb{R} \rightarrow [0, 1]$; and two finite particle configurations $\eta = \big( \eta (x) \big) \in \{ 0, 1 \}^{\mathbb{Z}}$ and $\xi = \big( \xi (x) \big) \in \{ 0, 1 \}^{\mathbb{Z}}$, we define 
\begin{flalign*}
\Delta_N (\eta, \xi) = \displaystyle\frac{1}{N} \displaystyle\max_{x \in \mathbb{Z}} \left| \displaystyle\sum_{i = -\infty}^x \big( \eta (i) - \xi (i) \big) \right|; \qquad \Delta_N (\eta; f) = \displaystyle\sup_{x \in \mathbb{R}}  \left| \displaystyle\int_{-\infty}^x \Big( \eta \big( \lfloor yN \rfloor \big) - f (y) \Big) dy \right|. 	
\end{flalign*}

\begin{lem}
	
\label{deltaetaxi} 

Fix real numbers $t, \varepsilon, \delta, \varsigma \ge 0$; an integer $N \ge 1$; and a compactly supported, measurable function $f: \mathbb{R} \rightarrow [0, 1]$. Let $\eta_0 = \big( \eta_0 (x) \big) \in \{ 0, 1 \}^{\mathbb{Z}}$ and $\xi_0 = \big( \xi_0 (x) \big) \in \{ 0, 1 \}^{\mathbb{Z}}$ denote two finite particle configurations on $\mathbb{Z}$, such that $\Delta_N \big( \eta_0, \xi_0 \big) \le \delta$. Further let $\eta_s$ and $\xi_s$ denote stochastic six-vertex models run with initial data $\eta_0$ and $\xi_0$, respectively.

If $\mathbb{P} \big[ \Delta_N (\xi_t; f) > \varepsilon \big] \le \varsigma$, then $\mathbb{P} \big[ \Delta_N (\eta_t; f) > \varepsilon + \delta \big] \le 2 \varsigma$. 
\end{lem} 

\begin{proof} 
	
By a union bound, it suffices to show that 
\begin{flalign}
\label{xiestimatef} 
\begin{aligned}
& \mathbb{P} \Bigg[ \displaystyle\sup_{x \in \mathbb{R}}  \displaystyle\int_{-\infty}^x \Big( \eta_t \big( \lfloor yN \rfloor \big) - f (y) \Big) dy > \varepsilon + \delta \Bigg] \le \varsigma; \\
& \mathbb{P} \Bigg[ \displaystyle\inf_{x \in \mathbb{R}}  \displaystyle\int_{-\infty}^x \Big( \eta_t \big( \lfloor yN \rfloor \big) - f (y) \Big) dy < - \varepsilon - \delta \Bigg] \le \varsigma.
\end{aligned}
\end{flalign} 

Let us only establish the former bound in \eqref{xiestimatef}, since the proof of the latter is very similar. To that end, for each $s \in \mathbb{Z}_{\ge 0}$, let $\textbf{p}_s = \big( p_s (1), p_s (2), \ldots , p_s (N_1) \big)$ and $\textbf{q}_s = \big( q_s (1), q_s (2), \ldots , q_s (N_2) \big)$ denote the particle position sequences associated with $\eta_s$ and $\xi_s$, respectively; also set $p_s (i) = -\infty = q_s (i)$ for each $i \le 0$; $p_s (i) = \infty$ for each $i > N_1$; and $q_s (i) = \infty$ for each $i > N_2$. 

Define $\textbf{r}_s = \big( r_s (k) \big)$ by setting $r_s (k) = p_s \big( k + \lfloor \delta N \rfloor \big)$ for each $(k, s) \in \mathbb{Z} \times \mathbb{Z}_{\ge 0}$, that is, by shifting the labeling in the particle position sequence for $\textbf{p}_s$ by $\lfloor \delta N  \rfloor$. Then the fact that $\Delta_N (\eta_0, \xi_0) \le \delta$ implies that $\textbf{q}_0 \le \textbf{r}_0$, which together with \Cref{lambdaximonotone} yields a coupling between $\textbf{q}$ and $\textbf{r}$ such that $\textbf{q}_s \le \textbf{r}_s$ for each $s \ge 0$. 

Thus, defining the events 
\begin{flalign*} 
E_1 = \Bigg\{ \displaystyle\sup_{x \in \mathbb{R}}  \displaystyle\int_{-\infty}^x \Big( \eta_t \big( \lfloor yN \rfloor \big) - f (y) \Big) dy > \varepsilon + \delta  \Bigg\}; \quad E_2 =  \Bigg\{ \displaystyle\sup_{x \in \mathbb{R}}  \displaystyle\int_{-\infty}^x \Big( \xi_t \big( \lfloor yN \rfloor \big) - f (y) \Big) dy > \varepsilon  \Bigg\},
\end{flalign*} 	

\noindent we have that $E_1 \subseteq E_2$ under this coupling. Hence, $\mathbb{P} [E_1] \le \mathbb{P} [E_2] \le \mathbb{P} \big[ \Delta_N (\xi_t; f) > \varepsilon \big] \le \varsigma$, which implies the first bound in \eqref{xiestimatef}. As mentioned previously, the proof of the second estimate there is very similar and therefore omitted. This implies the lemma. 
\end{proof} 

Before proceeding to the proof of \Cref{limitgeneral}, we require the following lemma, which states that functions of bounded variation can be well-approximated by a sum of piecewise constant functions, whose step sizes are bounded below; it appears as Lemma 3.6 of \cite{HAPAE}. 

\begin{lem}[{\cite[Lemma 3.6]{HAPAE}}]
	
\label{uestimate}

Fix a positive number $\varepsilon > 0$ and a compactly supported function $f: \mathbb{R} \rightarrow [0, 1]$ of bounded variation. There exist an integer $r > 0$ and real numbers $\delta > 0$; $x_0 < x_1 < x_2 < \cdots < x_r$; and $\rho_1, \rho_2, \ldots , \rho_r \in [0, 1]$ (all dependent on $\varepsilon$ and $f$) such that the following two properties hold. 

\begin{enumerate} 
	\item For each $i \in [1, r]$, we have that $|x_i - x_{i - 1}| \ge \delta$. 
	
	\item Denoting $f_{\delta} (x) = \sum_{i = 1}^r \rho_i \textbf{\emph{1}}_{x_{i - 1} < x \le x_i}$ for each $x \in \mathbb{R}$, we have that $\Delta (f, f_{\delta}) < \varepsilon \delta$. 
\end{enumerate} 

\end{lem}

We can now establish \Cref{limitgeneral}.

\begin{proof}[Proof of \Cref{limitgeneral}] 
	
	The proof of this theorem will follow the framework of \cite{HAPAE}, which is based on an algorithm of Glimm \cite{SNHSE} for numerically approximating solutions of hyperbolic conservation laws. More specifically, we first use \Cref{uestimate} to approximate $\Psi$ by a piecewise constant function $\Psi_{0; \delta}$ whose step sizes are lower bounded by some $\delta > 0$. Then, we let $\xi_s \in \{ 0, 1 \}^{\mathbb{Z}}$ denote a stochastic six-vertex model whose initial data $\xi_0$ is given by a product measure approximating $\Psi_{0; \delta}$. Since $\xi_0$ ``locally coincides'' with double-sided Bernoulli boundary data, \Cref{modelsequal}, \Cref{limitdouble}, and the first part of \Cref{uvequation} will together imply that the global law of $\xi_s$ converges to that of $\mathfrak{P}_s \Psi_{0; \delta}$, if $s \le c \delta$ (for some sufficiently small constant $c = c(b_2) > 0$). Furthermore, since $\Delta_N (\eta_0, \xi_0)$ is small with high probability, \Cref{uestimate} will imply that $\Delta_N (\eta_s, \xi_s)$ is also likely small. Hence, the global law of $\eta_s$ will be close to that of $\mathfrak{P}_s \Psi_{0; \delta} \approx \mathfrak{P}_s \Psi$, assuming that $s < c \delta$ is sufficiently small. We will then inductively repeat this procedure to increase the value of $s$ to $\lambda$. 
	
	Let us now implement this procedure in detail. In view of the second part of \Cref{uvequation}, \Cref{deltaetaxi}, and the fact that any measurable function can be approximated by one that is of bounded variation, we may assume that $\Psi$ is of bounded variation. Then, the first part of \Cref{uvequation} and the sixth part of Proposition 2.3.6 of \cite{SCL} imply that $\mathfrak{P}_t \Psi$ is compactly supported and of bounded variation for each $t \ge 0$. 
	
	Then, for any $t \in [0, \lambda]$, \Cref{uestimate}, yields the existence of an integer $r = r(t) > 0$ and real numbers $\delta = \delta (t) > 0$; $0 = x_0 < x_1 < x_2 < \cdots < x_r$; and $\rho_0, \rho_1, \ldots , \rho_r \in [0, 1]$ (all dependent on $t$, $\varepsilon$, and $\Psi$) such that the following two properties hold. First, $|x_i - x_{i - 1}| \ge \delta$ for each $i \in [1, r]$. Second, denoting $\Psi_{t; \delta} (x) = \sum_{i = 0}^r \rho_r \textbf{1}_{x_{i - 1} < x \le x_i}$ for each $x \in \mathbb{R}$, we have that $\Delta \big( \mathfrak{P}_t \Psi, \Psi_{t; \delta} \big) < \frac{(1 - b_2) \varepsilon \delta}{768 \lambda}$. 
	
	Next, set $s (t) = \min \big\{ \frac{(1 - b_2) \delta (t)}{8}, \lambda - t \big\}$ and define the interval $I(t) = \big[ t, t + s(t) \big]$ for each $t \ge 0$. Since $\bigcup_{t \in [0, \lambda)} I (t) = [0, \lambda]$ is compact, there exists a finite, increasing sequence $T = \big(t_0, t_1, \ldots , t_M)$ with $t_0 = 0$ such that $[0, \lambda] \subseteq \bigcup_{t \in T} I (t)$. We may assume that $t_M + s(t_M) = \lambda$; that $s(t_M) = \frac{(1 - b_2) \delta (t_M)}{8}$; and that $\sum_{t \in T} s(t) \le 3 \lambda$. For each $k \in [0, M]$, define $r_k = r(t_k)$; $\delta_k = \delta (t_k)$; $s_k = s(t_k)$; and $I_k = I(t_k)$. 
	
	We claim that
	\begin{flalign}
	\label{etapspsik}
	\displaystyle\max_{0 \le s \le s_k} \displaystyle\lim_{N \rightarrow \infty} \mathbb{P} \Bigg[ \Delta_N \big( \eta_{(t_k + s) N}; \mathfrak{P}_{t_k + s} \Psi \big) > \displaystyle\frac{(1 - b_2) \varepsilon}{96 \lambda} \displaystyle\sum_{i = 0}^k \delta_i + \displaystyle\frac{\varepsilon}{4} \Bigg] = 0,
	\end{flalign} 
	
	\noindent for each $k \in [-1, M]$, where we define $t_{-1} = 0 = s_{-1}$. To that end, we induct on $k$, the statement being true when $k = -1$ by \eqref{limitpsi}. Thus, assume \eqref{etapspsik} holds for $k = m - 1$ and some $m \in [0, M]$, so that $\lim_{N \rightarrow \infty} \mathbb{P} [E_N] = 0$, where we have defined the event
	\begin{flalign}
	\label{etatmn}
	E = E_N = \left\{  \Delta_N \big( \eta_{t_m N}; \mathfrak{P}_{t_m} \Psi \big) > \displaystyle\frac{(1 - b_2) \varepsilon}{96 \lambda} \displaystyle\sum_{i = 0}^{m - 1} \delta_i + \displaystyle\frac{\varepsilon}{4} \right\}.
	\end{flalign} 
	
	\noindent We will show that \eqref{etapspsik} holds for $k = m$. To that end, fix $s \in [0, s_m]$, and observe by the second statement of \Cref{uvequation} that
	\begin{flalign}
	\label{etaptmspsiestimate1}
	\begin{aligned}
	\Delta_N ( \eta_{(t_m + s) N}; \mathfrak{P}_{t_m + s} \Psi ) & \le \Delta_N ( \eta_{(t_m + s)N}; \mathfrak{P}_s \Psi_{t_m; \delta} ) + \Delta (\mathfrak{P}_s \Psi_{t_m; \delta}, \mathfrak{P}_{t_m + s} \Psi) \\
	& \le \Delta_N ( \eta_{(t_m + s)N}; \mathfrak{P}_s \Psi_{t_m; \delta} ) + \Delta (\Psi_{t_m; \delta}, \mathfrak{P}_{t_m} \Psi) \\
	&  \le \Delta_N ( \eta_{(t_m + s)N}; \mathfrak{P}_s \Psi_{t_m; \delta} ) + \displaystyle\frac{(1 - b_2) \varepsilon \delta_m}{768 \lambda}. 
	\end{aligned} 
	\end{flalign}
	
	Now, recall the numbers $r \ge 1$; $\delta > 0$; $x_0 < x_1 < \cdots < x_r$; and $\rho_1, \rho_2, \ldots , \rho_r \in [0, 1]$ associated with $t_m$, and further set $\rho_0 = 0 = \rho_{r + 1}$, $x_{-1} = -\infty$, and $x_{r + 1} = \infty$. Let $\xi_0 = \big( \xi_0 (x) \big) \in \{ 0, 1 \}^{\mathbb{Z}}$ denote a random particle configuration on $\mathbb{Z}$, with the $\big\{ \xi_0 (x) \big\}$ being mutually independent $0-1$ Bernoulli random variables, such that $\mathbb{P} \big[ \xi_0 (x) = 1 \big] = \rho_i$ if $x_{i - 1} < \frac{x}{N} \le x_i$, for each $i \in [0, r + 1]$. The profile for the initial data $\xi_0$ therefore approximates $\Psi_{t_m; \delta_m}$. Let $\xi_t$ denote the stochastic six-vertex model, run with this initial data $\xi_0$. 
	
	Furthermore, for each $i \in [1, r + 1]$, let $\zeta_0^{(i)} = \big( \zeta_0^{(i)} (x) \big) \in \{ 0, 1 \}^{\mathbb{Z}}$ denote a random (double-sided Bernoulli) particle configuration on $\mathbb{Z}$, where the $\big\{ \zeta_0^{(i)} (x) \big\}$ are mutually independent, such that $\mathbb{P} \big[ \zeta_0^{(i)} (x) = 1 \big] = \rho_{i - 1}$ whenever $\frac{x}{N} \le x_{i - 1}$ and $\mathbb{P} \big[ \zeta_0^{(i)} (x) = 1 \big] = \rho_i$ whenever $\frac{x}{N} > x_{i - 1}$. Further couple $\xi_0$ and each $\zeta_0^{(i)}$ such that $\xi_0 (x) = \zeta_0^{(i)} (x)$ whenever $\frac{x}{N} \in [x_{i - 2}, x_i]$. For each $i$, let $\zeta_t^{(i)}$ denote stochastic six-vertex model with initial data $\zeta_0^{(i)}$, and couple $\zeta_t^{(i)}$ with $\xi_t$ under the higher rank coupling of \Cref{couplinghigherrank}.
	
	Next, for each $i \in [1, r + 1]$, define the function $f_i: \mathbb{R} \rightarrow [0, 1]$ by setting $f_i (x) = \rho_{i - 1} \textbf{1}_{x \le x_i} + \rho_i \textbf{1}_{x > x_i}$ for every $x \in \mathbb{R}$. Then, \Cref{limitdouble} implies for each $i$ that 
	\begin{flalign*}
	\displaystyle\lim_{N \rightarrow \infty} \mathbb{P} \Bigg[ \displaystyle\max_{- \frac{N}{\varepsilon} \le X_1 \le X_2 \le \frac{N}{\varepsilon}} \bigg| \displaystyle\frac{1}{N} \displaystyle\sum_{x = X_1}^{X_2} \zeta_{s N}^{(i)} (x) - \displaystyle\int_{X_1 / N}^{X_2 / N} \mathfrak{P}_s f_i (x) \bigg| > \displaystyle\frac{(1 - b_2) \varepsilon \delta}{768 (r + 1) \lambda} \Bigg] = 0,
	\end{flalign*}
	
	\noindent By \Cref{modelsequal} and the fact that $\frac{4s}{1 - b_2} \le \frac{\delta}{2}$, we then obtain that 
	\begin{flalign}
	\label{xisnpsfibound}
	\displaystyle\lim_{N \rightarrow \infty} \mathbb{P} \Bigg[ \displaystyle\max_{X_1, X_2} \bigg| \displaystyle\frac{1}{N} \displaystyle\sum_{x = X_1}^{X_2} \xi_{s N} (x) - \displaystyle\int_{X_1 / N}^{X_2 / N} \mathfrak{P}_s f_i (x) \bigg| > \displaystyle\frac{(1 - b_2) \varepsilon \delta}{768 (r + 1) \lambda} \Bigg] = 0,
	\end{flalign}
	
	\noindent where $(X_1, X_2)$ is taken over all such pairs satisfying $\max \big\{ \big(x_{i - 1} + \frac{\delta}{2} \big) N, - \frac{N}{\varepsilon} \big\} \le X_1 \le X_2 \le \min \big\{ \big( x_{i + 1} - \frac{\delta}{2} \big) N, \frac{N}{\varepsilon} \big\}$. 
	
	Next, since $\Psi_{t_m; \delta} (x) = f_i (x)$ for each $x \in [x_{i - 1}, x_{i + 1}]$, it follows from the first part of \Cref{uvequation} and the estimate $\frac{2s}{1 - b_2} \le \frac{\delta}{4}$ that $\mathfrak{P}_s \Psi_{t_m; \delta} (x) = \mathfrak{P}_s f_i (x)$ for each $x \in \big[ x_{i - 1} + \frac{\delta}{4}, x_{i + 1} - \frac{\delta}{4} \big]$. Therefore, the fact that $\bigcup_{i = 0}^{r + 1} \big[x_{i - 1} + \frac{\delta}{4}, x_i - \frac{\delta}{4} \big] = \mathbb{R}$ implies from \eqref{xisnpsfibound} and a union bound that 
	\begin{flalign*}
	\displaystyle\lim_{N \rightarrow \infty} \mathbb{P} \Bigg[ \displaystyle\max_{- \frac{N}{\varepsilon} \le X_1 \le  X_2 \le \frac{N}{\varepsilon}} \bigg| \displaystyle\frac{1}{N} \displaystyle\sum_{x = X_1}^{X_2} \xi_{s N} (x) - \displaystyle\int_{X_1 / N}^{X_2 / N} \mathfrak{P}_s \Psi_{t_m; \delta} \bigg| > \displaystyle\frac{(1 -b_2) \varepsilon \delta}{768 \lambda} \Bigg] = 0,
	\end{flalign*}
	
	\noindent which yields  
	\begin{flalign}
	\label{xisnpsfibound3}
	\displaystyle\lim_{N \rightarrow \infty} \mathbb{P} \Bigg[ \Delta_N ( \xi_{sN}; \mathfrak{P}_{s} \Psi_{t_m; \delta}) > \displaystyle\frac{(1 - b_2) \varepsilon \delta}{384 \lambda} \Bigg] = 0, 
	\end{flalign}
	
	\noindent since by \Cref{modelsequal} the probability that $\xi_{sN}$ contains a particle outside of the interval $\big[ -\frac{N}{\varepsilon}, \frac{N}{\varepsilon} \big]$ tends to $0$, as $N$ tends to $\infty$ (assuming that $\varepsilon > 0$ is sufficiently small). 
	
	Now, for each integer $N \ge 1$, define the event 
	\begin{flalign*}
	F = F_N = \left\{  \Delta_N (\xi_0; \Psi_{t_m; \delta}) > \displaystyle\frac{(1 - b_2) \varepsilon \delta}{768 \lambda} \right\}.
	\end{flalign*} 
	
	\noindent Large deviations bounds for sums of $0-1$ Bernoulli random variables then yield $\lim_{N \rightarrow \infty} \mathbb{P} [F_N] = 0$. Next observe that
	\begin{flalign*} 
	\textbf{1}_{E^c} \Delta_N (\eta_{t_m N}; \Psi_{t_m; \delta}) & \le \textbf{1}_{E^c} \Delta_N (\eta_{t_m N}; \mathfrak{P}_{t_m} \Psi) + \Delta (\mathfrak{P}_{t_m} \Psi; \Psi_{t_m; \delta}) \\
	& \le \displaystyle\frac{(1 - b_2) \varepsilon}{96 \lambda} \displaystyle\sum_{i = 0}^{m - 1} \delta_i + \displaystyle\frac{\varepsilon}{4} + \displaystyle\frac{(1 - b_2) \varepsilon \delta}{768 \lambda},
	\end{flalign*} 
	
	\noindent which, together with \eqref{xisnpsfibound3} and \Cref{deltaetaxi}, implies that 
	\begin{flalign*}
	\displaystyle\lim_{N \rightarrow \infty} \mathbb{P}\left[ \textbf{1}_{E^c} \textbf{1}_{F^c} \Delta_N (\eta_{(t_m + s) N}; \mathfrak{P}_s \Psi_{t_m; \delta}) > \displaystyle\frac{(1 - b_2) \varepsilon}{96 \lambda} \displaystyle\sum_{i = 0}^{m - 1} \delta_i + \displaystyle\frac{(1 - b_2)\varepsilon \delta_m}{192 \lambda} + \displaystyle\frac{\varepsilon}{4} \right] = 0.
	\end{flalign*}
	
	\noindent Combined with \eqref{etaptmspsiestimate1} and the facts that $\lim_{N \rightarrow \infty} \mathbb{P} [E_N] = 0 = \lim_{N \rightarrow \infty} \mathbb{P} [F_N]$, this yields  
	\begin{flalign*}
	\displaystyle\lim_{N \rightarrow \infty} \mathbb{P} \left[ \Delta_N (\eta_{(t_m + s) N}, \mathfrak{P}_{t_m + s} \Psi) >  \displaystyle\frac{(1 - b_2) \varepsilon}{96 \lambda} \displaystyle\sum_{i = 0}^m \delta_i + \displaystyle\frac{\varepsilon}{4} \right] = 0,
	\end{flalign*}
	
	\noindent which verifies \eqref{etapspsik}. Taking $k = M$ and $s = \lambda - t_M = s_M$ in \eqref{etapspsik} and using the fact that $\sum_{i = 1}^M \delta_i = \frac{8}{1 - b_2} \sum_{i = 1}^M s_i \le \frac{24 \lambda}{1 - b_2}$ then gives $\lim_{N \rightarrow \infty} \mathbb{P} \big[ \Delta_N (\eta_{\lambda N}, \mathfrak{P}_{\lambda} \Psi) > \frac{\varepsilon}{2} \big] = 0$, from which we deduce the theorem. 	
\end{proof}

\subsection{Comparing the Cylinder to the Upper Half-Plane} 

\label{LimitCylinder} 

In this section we establish \Cref{sixvertexcylinderglobal} by comparing the stochastic six-vertex model on the cylinder $\mathfrak{C} = \mathfrak{C}_{N; L} = \mathfrak{T}_N \times \{ 1, 2, \ldots , L \}$ to that on the upper-half plane $\mathfrak{H} = \mathbb{Z} \times \mathbb{Z}_{\ge 0}$. 

To that end, we begin with the following proposition, which bounds the speed of discrepancies between suitably coupled stochastic six-vertex models on the cylinder and on the upper half-plane. In the below, we recall from \Cref{Height} that any six-vertex ensemble $\mathcal{E}$ on $\mathfrak{C}$ can be associated with a particle configuration $\big( \eta_y (x) \big)$. By letting $\eta_y (x)$ denote the indicator for the existence of a particle at site $x \in \mathfrak{T} = \mathfrak{T}_N$ and time $y \in \{ 1, 2, \ldots , L \}$, we can view the stochastic six-vertex model on $\mathfrak{C}$ as an interacting particle system on $\mathfrak{T}$. Under this correspondence, the (non-crossing) directed up-right paths in $\mathcal{E}$ are the space-time trajectories for particles. 

As in \Cref{StochasticVertexLine}, we can tag these particles. Set $\textbf{p}_0 = \big( p_0 (-M), p_0 (1 - M), \ldots , p_0 (R) \big)$, where $p_0 (k)$ denotes the initial position of particle $k \in [-M, R]$; we assume that no particle exists between $p_0 (k)$ and $p_0 (k + 1)$, for each $k \in [-M, R]$. Then, let $p_t (k)$ denote the position of particle $k \in [-M, R]$ at time $t \in [1, L]$, and set $\textbf{p}_t = \big( p_t (-M), p_t (1 - M), \ldots , p_t (R) \big)$. This associates $\mathcal{E}$ with a particle position sequence $\textbf{p}_t$.

\begin{prop} 
	
\label{cylinderhalfplanevertex} 

There exists a constant $c = c(b_2) > 0$ such that the following holds. Let $N, L \ge 1$ and $A \ge 0$ denote integers such that $A < N$, and let $\textbf{\emph{p}}_0 = \big( p_0 (k) \big)$ and $\textbf{\emph{q}}_0 = \big( q_0 (k) \big)$ denote (finite) particle position sequences on $\mathfrak{T} = \mathfrak{T}_N$ and $\mathbb{Z}$, respectively. Assume that $p_0 (k) = q_0 (k)$ for each $k \in \mathbb{Z}$ such that either $p_0 (k) \in [0, A]$ or $q_0 (k) \in [0, A]$. Assume further that $0 < \frac{4L}{1 - b_2} < A - \frac{4L}{1 - b_2} < A + \frac{4L}{1 - b_2} < N$, and let $\textbf{\emph{p}}_s$ and $\textbf{\emph{q}}_s$ denote the stochastic six-vertex models on $\mathfrak{C} = \mathfrak{C}_{N; L}$ and $\mathfrak{H}$ with initial data $\textbf{\emph{p}}_0$ and $\textbf{\emph{q}}_0$, respectively. 

Then, it is possible to couple $\textbf{\emph{p}}_s$ and $\textbf{\emph{q}}_s$ in such a way that $\mathbb{P} [E] \le c^{-1} e^{-cL}$, where $E$ denotes the event that there exists a pair $(k, t) \in \mathbb{Z} \times [0, L]$ such that $p_t (k) \ne q_t (k)$ and either $p_t (k)$ or $q_t (k)$ is in the interval $\big[ \frac{4L}{1 - b_2}, A - \frac{4L}{1 - b_2} \big]$. 

\end{prop} 

\begin{proof} 
	
Let $U \le V$ denote the integers (if they exist) such that $q_0 (U - 1) < 0 \le q_0 (U) \le q_0 (V) \le A < q_0 (V + 1)$, and let $\textbf{m}_0 = \big( m_0 (U), m_0 (U + 1), \ldots , m_0 (V) \big)$ and $\textbf{r}_0 = \big( r_0 (U), r_0 (U + 1), \ldots , r_0 (V)\big)$ denote the particle position sequences on $\mathfrak{T}$ and $\mathbb{Z}$, respectively, obtained by setting $m_0 (k) = r_0 (k) = p_0 (k) = q_0 (k)$ for each $k \in [U, V]$. Stated alternatively, $\textbf{m}_0$ and $\textbf{r}_0$ are obtained from $\textbf{p}_0$ and $\textbf{q}_0$, respectively, by removing all particles outside of the interval $[0, A]$. 

Now let $\textbf{m}_s$ and $\textbf{r}_s$ denote the stochastic six-vertex models on $\mathfrak{T}$ and $\mathbb{Z}$ with initial data $\textbf{m}_0$ and $\textbf{r}_0$, respectively. We will exhibit couplings between $(\textbf{p}_s, \textbf{m}_s)$, $(\textbf{m}_s, \textbf{r}_s)$, and $(\textbf{r}_s, \textbf{q}_s)$ such that these pairs of stochastic six-vertex models agree with high probability on the interval $I = \big[ \frac{4L}{1 - b_2}, A - \frac{4L}{1 - b_2} \big]$ for each $s \in [1, L]$. This will induce a coupling between $(\textbf{p}_s, \textbf{q}_s)$ with the same property. 

The models $(\textbf{r}_s, \textbf{q}_s)$, will be coupled through the higher rank coupling of \Cref{couplinghigherrank}. To couple $(\textbf{m}_s, \textbf{r}_s)$, let $F (\textbf{r})$ denote the event on which there exists some $s \in [1, L]$ such that $r_s (V) \ge N$. Also let $F (\textbf{m})$ denote the event on which there exists some $s \in [1, L]$ such that $m_s (V) < m_{s - 1} (V)$ (meaning that there exists a particle in $\textbf{m}$ that passes site $N$ at some point during the time interval $[1, L]$). Set $p (\textbf{m}) = \mathbb{P} \big[ F (\textbf{m})^c \big]$ and $p (\textbf{r}) = \mathbb{P} \big[ F (\textbf{r})^c \big]$. 

Letting $W$ denote a uniformly random variable on $[0, 1]$, we then sample $(\textbf{m}_s)$ and $(\textbf{r}_s)$ as follows. 

\begin{enumerate} 

\item If $W \le \min \big\{ p (\textbf{m}), p (\textbf{r}) \big\}$, then sample $(\textbf{m}_s)$ conditional on $F (\textbf{m})^c$, and set $(\textbf{r}_s) = (\textbf{m}_s)$. 

\item If $\min \big\{ p (\textbf{m}), p (\textbf{r}) \big\} < W \le \max \big\{ p (\textbf{m}), p (\textbf{r}) \big\}$, then let $\textbf{x}, \textbf{y} \in \{ \textbf{m}, \textbf{r} \}$ be distinct so that $p (\textbf{x}) = \min \big\{ p (\textbf{m}), p (\textbf{r}) \big\}$ and $p (\textbf{y}) = \max \big\{ p (\textbf{m}), p (\textbf{r}) \big\}$. Sample $\textbf{x}_s$ conditional on $F (\textbf{x})$ and $\textbf{y}_s$ conditional on $F(\textbf{y})^c$ independently. 

\item If $W > \max \big\{ p (\textbf{m}), p (\textbf{r}) \big\}$, then sample $\textbf{m}_s$ conditional on $F (\textbf{m})$ and $\textbf{r}_s$ conditional on $F(\textbf{r})$ independently. 

\end{enumerate} 

\noindent That this yields a couping between $\textbf{m}_s$ and $\textbf{r}_s$ follows from the fact that the weight (recall \Cref{StochasticModel}) of any finite six-vertex ensemble $\mathcal{E}$, each of whose paths do not intersect the vertical line $x = N$, does not depend on whether the domain of $\mathcal{E}$ is $\mathfrak{C}$ or $\mathfrak{H}$. 

To couple $(\textbf{m}_s, \textbf{p}_s)$ consider a multi-class stochastic six-vertex model $\textbf{n}$ (recall \Cref{HigherRank}) on $\mathfrak{C}$, with two classes. Its boundary data $\psi = \big( \psi (x) \big) \in \{ 0, 1, 2 \}^{\mathbb{Z}}$ is defined by setting $\psi (x) = 0$ if $x \notin \textbf{p}_0$; $\psi (x) = 1$ if $x \in \textbf{m}_0$; and $\psi (x) = 2$ if $x \in \textbf{p}_0 \setminus \textbf{m}_0$. Denote the positions of the first and second class particles at time $s$ in this model by $\textbf{n}_s^{(1)} = \big( n_s^{(1)} (U), n_s^{(1)} (U + 1), \ldots , n_s^{(1)} (V) \big)$ and $\textbf{n}_s^{(2)} = \big( n_s^{(2)} (B), n_s^{(2)} (B + 1), \ldots , n_s^{(2)} (D) \big)$, respectively. Then, define $\textbf{p}_s = (\textbf{p}_s) = \big( \textbf{n}_s^{(1)} \big) \cup \big( \textbf{n}_s^{(2)} \big)$, that is, by ignoring the class labels in $\textbf{n}$; also define $\textbf{m}_s = (\textbf{m}_s) = \big( \textbf{n}_s^{(1)})$, that is, by only considering the first class particles in $\textbf{n}$. By \Cref{mnclasses}, this yields a coupling between $(\textbf{m}_s, \textbf{p}_s)$.

Thus, we have defined couplings between the pairs $(\textbf{p}_s, \textbf{m}_s)$, $(\textbf{m}_s, \textbf{r}_s)$, and $(\textbf{r}_s, \textbf{q}_s)$ of stochastic six-vertex models. Let $E_1$, $E_2$, and $E_3$ denote the events that $(\textbf{p}_s, \textbf{m}_s)$, $(\textbf{m}_s, \textbf{r}_s)$, and $(\textbf{r}_s, \textbf{q}_s)$ differ on $I \times [1, L]$, respectively; will bound $\mathbb{P} [E_1]$, $\mathbb{P} [E_2]$, and $\mathbb{P} [E_3]$. Observe that \Cref{modelsequal} guarantees the existence of a constant $c = c(b_2) > 0$ such that $\mathbb{P} [E_3] \le c^{-1} e^{-cL}$. 

Next, $E_1$ is contained in the event that there exists a second-class particle that enters $I$ at some point during the time interval $[1, L]$. Thus, \Cref{xtxt1}, \Cref{xtxt1cylinder}, and a large deviations estimate for sums of independent geometric random variables together imply that $\mathbb{P} [E_1] \le c^{-1} e^{-cL}$ (after decreasing $c$ if necessary). 

We bound $\mathbb{P} [E_2]$ analogously, by observing that $\mathbb{P} [E_2] \le \max \big\{ 1 - p (\textbf{m}), 1 - p (\textbf{r}) \big\} \le c^{-1} e^{-cL}$, where the last estimate again follows from \Cref{xtxt1}, \Cref{xtxt1cylinder}, and a large deviations bound. Hence, $\mathbb{P} [E] \le \mathbb{P} [E_1] + \mathbb{P} [E_2] + \mathbb{P} [E_3] \le 3c^{-1} e^{-cL}$, from which we deduce the proposition. 
\end{proof} 

Now we can establish \Cref{sixvertexcylinderglobal}. 

\begin{proof}[Proof of \Cref{sixvertexcylinderglobal}] 
	
	The proof of this theorem will be similar to that of \Cref{limitgeneral}. For each $s \ge 0$ and function $f: \mathbb{T} \rightarrow [0, 1]$, we (similarly to in \Cref{ptoperator}) define $\mathfrak{Q}_s f (x) = G(x, s)$, where $G (x, t)$ denotes the entropy solution of the equation \eqref{gxydefinition} on $\mathbb{T} \times \mathbb{R}_{\ge 0}$, with initial data given by $G (x, 0) = f(x)$ for each $x \in \mathbb{T}$. 	

	For any fixed $\varpi, \varepsilon > 0$, let us first show that
	\begin{flalign}
	\label{qtetalimit}
	\displaystyle\lim_{N \rightarrow \infty} \mathbb{P} \Bigg[  \displaystyle\max_{0 \le X_1 \le X_2 < N} \bigg| \displaystyle\frac{1}{N} \displaystyle\sum_{x = X_1}^{X_2} \eta_{\lfloor \varpi N \rfloor} (x) - \displaystyle\int_{X_1 / N}^{X_2 / N} \mathfrak{Q}_{\varpi} \Psi (y) dy \bigg| > \varepsilon \Bigg] = 0.
	\end{flalign}
	
	\noindent Setting $\delta = \frac{1 - b_2}{32}$ and $M = \big\lceil \frac{\varpi}{\delta} \big\rceil - 1$, we claim for any integer $k \in [-1, M]$ that 
	\begin{flalign}
	\label{gkdeltaetalimit} 
	\displaystyle\max_{s \in [0, \delta]} \displaystyle\lim_{N \rightarrow \infty} \mathbb{P} \Bigg[ \displaystyle\max_{0 \le X_1 \le X_2 < N} \bigg| \displaystyle\frac{1}{N} \displaystyle\sum_{x = X_1}^{X_2} \eta_{\lfloor (k \delta + s)  N \rfloor} (x) - \displaystyle\int_{X_1 / N}^{X_2 / N} \mathfrak{Q}_{k \delta + s} \Psi (y) dy \bigg| > \varsigma \Bigg] = 0,
	\end{flalign}
	
	\noindent for any $\varsigma > 0$, where we define $\eta_m (x) = \eta_0 (x)$ for any $x \in \mathfrak{T}$ if $m \le 0$, and $G_t (y) = G_0 (y)$ for any $y \in \mathbb{T}$ if $t \le 0$. To verify \eqref{gkdeltaetalimit}, we induct on $k$, the statement being true by \eqref{limitpsi} if $k = -1$. Thus let $m \in [0, M]$ denote an integer, and assume that \eqref{gkdeltaetalimit} holds for $k = m - 1$, which implies in particular for any $\varsigma > 0$ that
	\begin{flalign}
	\label{gkdeltaetalimit2} 
	\mathbb{P} \Bigg[ \displaystyle\max_{0 \le X_1 \le X_2 < N} \bigg| \displaystyle\frac{1}{N} \displaystyle\sum_{x = X_1}^{X_2} \eta_{\lfloor m \delta  N \rfloor} (x) - \displaystyle\int_{X_1 / N}^{X_2 / N} \mathfrak{Q}_{m \delta} \Psi (y) dy \bigg| > \varsigma \Bigg] = 0.
	\end{flalign}
	
	\noindent We will then show that \eqref{gkdeltaetalimit} holds for $k = m$. 
	
	To do this, let us fix $s \in [0, \delta]$ and define the intervals $I_1 = \big( 0, \frac{3}{4} \big) \subset \mathbb{T}$ and $I_2 = \big( \frac{1}{2}, \frac{5}{4} \big) \subset \mathbb{T}$; observe that $I_1 \cup I_2 = \mathbb{T}$. Depending on the context, we may also view $I_1$ and $I_2$ as intervals $\mathbb{R}$. We then define functions $f_1, f_2: \mathbb{T} \rightarrow [0, 1]$ by setting $f_j (x) = \mathfrak{Q}_{m \delta} \Psi (x) \textbf{1}_{x \in I_j}$, for each $x \in \mathbb{T}$ and $j \in \{ 1, 2 \}$. The corresponding functions on $\mathbb{R}$ are denoted by $g_1, g_2: \mathbb{R} \rightarrow [0, 1]$, defined by setting $g_j (x) = \mathfrak{Q}_{m \delta} \Psi (x) \textbf{1}_{x \in I_j}$, for each $x \in \mathbb{R}$ and $j \in \{ 1, 2 \}$, where we view $\mathfrak{Q}_{m \delta} \Psi$ as a $1$-periodic function on $\mathbb{R}$.
	
	Similarly, we define (random) particle configurations $\xi_0^{(1)} = \big( \xi_0^{(1)} (x) \big) \in \{ 0, 1 \}^{\mathbb{Z}}$ and $\xi_0^{(2)} = \big( \xi_0^{(2)} (x) \big) \in \{ 0, 1 \}^{\mathbb{Z}}$ on $\mathbb{Z}$ by setting $\xi_0^{(j)} (x) = \eta_{\lfloor m \delta N \rfloor} (x) \textbf{1}_{x / N \in I_j}$, for each $x \in \mathbb{Z}$ and $j \in \{ 1, 2 \}$. Further let $\big( \xi_t^{(j)} \big)$ denote the stochastic six-vertex model on $\mathbb{Z}$, with initial data given by $\xi_0^{(j)}$, for each $j \in \{ 1, 2 \}$. 
	
	Then, \Cref{limitgeneral} and \eqref{gkdeltaetalimit2} together imply for each $j \in \{ 1, 2 \}$ and any $\varsigma > 0$ that  
	\begin{flalign}
	\label{gkdeltaetalimit3}
	\mathbb{P} \Bigg[ \displaystyle\max_{-N \le X_1 \le X_2 \le 2N} \bigg| \displaystyle\frac{1}{N} \displaystyle\sum_{x = X_1}^{X_2} \xi_{\lfloor (m \delta + s) N \rfloor}^{(j)} (x) - \displaystyle\int_{X_1 / N}^{X_2 / N} \mathfrak{P}_s g_j (y) dy \bigg| > \displaystyle\frac{\varsigma}{2} \Bigg] = 0.
	\end{flalign}
	
	 Now, the first part of \Cref{uvequation} and the estimate $\frac{2 \delta}{1 - b_2} \le \frac{1}{16}$ together imply that the function $\mathfrak{P}_t g_1$ is supported on the interval $\big[ -\frac{1}{16}, \frac{13}{16} \big]$ whenever $t \le \delta$. Thus, $\mathfrak{P}_t g_1$ can be viewed as a function on $\mathbb{T}$ that is equal to $0$ on a neighborhood of $\frac{7}{8}$. Combining this with the fact that $\mathfrak{P}_t g_1$ solves the equation \eqref{gxydefinition} on $\mathbb{R} \times \mathbb{R}_{\ge 0}$, we deduce that it satisfies the same equation \eqref{gxydefinition} on $\mathbb{T} \times [0, \delta]$. Hence, $\mathfrak{P}_t g_1 = \mathfrak{Q}_t f_1$ for any $t \in [0, \delta]$. Similarly, $\mathfrak{P}_t g_2 = \mathfrak{Q}_t f_2$ whenever $t \in [0, \delta]$.
	 
	 Thus, \Cref{cylinderhalfplanevertex} and \eqref{gkdeltaetalimit3} together imply that
	 \begin{flalign}
	 \label{gkdeltaetalimit4}
	 \mathbb{P} \Bigg[ \displaystyle\max_{\frac{X_1}{N}, \frac{X_2}{N} \in \mathscr{I}_j} \bigg| \displaystyle\sum_{x = X_1}^{X_2} \eta_{\lfloor (m \delta + s) N \rfloor} (x) - \displaystyle\int_{X_1 / N}^{X_2 / N} \mathfrak{Q}_s f_j (y) dy \bigg| > \displaystyle\frac{\varsigma}{2} \Bigg] = 0,
	 \end{flalign}
	 
	 \noindent for each $j \in \{ 1, 2 \}$, where we have defined the intervals $\mathscr{I}_1 = \big[ \frac{1}{8}, \frac{5}{8} \big] \subset \mathbb{T}$ and $\mathscr{I}_2 = \big[ \frac{5}{8}, \frac{9}{8} \big] \subset \mathbb{T}$ (and we assume that $X_2 \le X_1$). Now \eqref{gkdeltaetalimit} follows from \eqref{gkdeltaetalimit4}; the fact that $\mathfrak{Q}_t f_j (x) = \mathfrak{Q}_{m \delta + t} \Psi (x)$ for each $x \in \mathscr{I}_j$, $t \in (0, \delta]$, and $j \in \{ 1, 2 \}$ (by \Cref{uvequation2}); and the fact that $\mathscr{I}_1 \cup \mathscr{I}_2 = \mathbb{T}$. Then \eqref{qtetalimit} follows from \eqref{gkdeltaetalimit} by setting $k = M$, $s = \varpi - M \delta$, and $\varsigma = \varepsilon$. 
	 
	 Now \eqref{etasumestimaten} is a consequence of \eqref{qtetalimit}; a union bound; the fact that $\eta_y (x), G_y (x) \in [0, 1]$ for each $x$ and $y$; the continuity of $G_y (x)$ in $y$; and the fact that
	 \begin{flalign}
	 \label{etay1yy1}
	 \displaystyle\sum_{x = X_1}^{X_2} \eta_{Y - 1} (x) - 1 \le \displaystyle\sum_{x = X_1}^{X_2} \eta_Y (x) \le \displaystyle\sum_{x = X_1}^{X_2} \eta_{Y - 1} (x) + 1,
	 \end{flalign}
	 
	 \noindent for any integers $0 \le X_1 \le X_2 < N$ and $Y > 0$. 
\end{proof}

\section{Proof of \Cref{sixvertexlocalcylinder}} 

\label{StationaryConverge} 
		
In this section we establish \Cref{sixvertexlocalcylinder}. We begin in \Cref{InitialConstant} by reducing it to \Cref{couplexiaeta} below, which essentially addresses the case of a constant density profile. Then, we introduce and state properties of several couplings in \Cref{StationaryApproximateLocal}. Following the framework of Bahadoran-Mountford \cite{CLEP}, we then use these couplings to prove \Cref{couplexiaeta} in \Cref{StationaryLocal}.

\subsection{Reduction to a Constant Density Profile}

\label{InitialConstant} 

In this section we reduce \Cref{sixvertexlocalcylinder} to the following theorem, which is similar to (A) and (B) in Section 4 of \cite{CLEP}. It essentially states that if a stochastic six-vertex model $\eta_t$ with initial profile $\eta_0$ of approximately constant density $\rho \in (0, 1)$ is run for a sufficiently large time $t$, then $\eta_t$ can likely be coupled between two stationary distributions with densities $\rho - \varepsilon$ and $\rho + \varepsilon$, for any $\varepsilon > 0$. 

\begin{thm} 
	
	\label{couplexiaeta} 
	
	Fix $\varepsilon \in \big( 0, \frac{1}{2} \big)$ and $\rho \in [\varepsilon, 1 - \varepsilon]$. Let $\xi_0^{(\rho - \varepsilon)} = \big( \xi_0^{(\rho - \varepsilon)} (x) \big) \in \{ 0, 1 \}^{\mathbb{Z}}$ and $\xi_0^{(\rho + \varepsilon)} = \big( \xi_0^{(\rho + \varepsilon)} (x) \big) \in \{ 0, 1 \}^{\mathbb{Z}}$ denote two (random) particle configurations on $\mathbb{Z}$ sampled with respect to the measures $\Upsilon^{(\rho - \varepsilon)}$ and $\Upsilon^{(\rho + \varepsilon)}$, respectively (from \Cref{zetarho}), which are coupled so that $\xi_0^{(\rho - \varepsilon)} \le \xi_0^{(\rho + \varepsilon)}$ almost surely. Furthermore, for each $N \in \mathbb{Z}_{\ge 1}$, let $\eta_0 = \eta_0^{(N)} = \big( \eta_0 (x) \big) = \big( \eta_0^{(N)} (x) \big) \in \{ 0, 1 \}^{\mathbb{Z}}$ denote a particle configuration such that $\eta (x) = 0$ whenever $|x| > \frac{8N}{1 - b_2}$. 
	
	Assume for any $c \in (0, \frac{8}{1 - b_2} \big)$ that
	\begin{flalign}
	\label{etasumlimitx1x2} 
	\displaystyle\lim_{N \rightarrow \infty} \displaystyle\frac{1}{c N} \displaystyle\sum_{x = 0}^{\lfloor c N \rfloor} \eta_0 (x) = \rho = \displaystyle\lim_{N \rightarrow \infty} \displaystyle\frac{1}{c N} \displaystyle\sum_{x = - \lfloor cN \rfloor}^0 \eta_0 (x),
	\end{flalign} 
	
	\noindent and let $\eta_t$, $\xi_t^{(\rho - \varepsilon)}$, and $\xi_t^{(\rho + \varepsilon)}$ denote the stochastic six-vertex models on $\mathbb{Z}$, with initial data $\eta_0$, $ \xi_0^{(\rho - \varepsilon)}$, and $\xi_0^{(\rho + \varepsilon)}$, respectively, which are mutually coupled under the higher rank coupling of \Cref{couplinghigherrank} (recall \Cref{etaxim}). Also let $E_N^{(1)}$ and $E_N^{(2)}$ denote the events on which there exists some $(i, t) \in [-N, N] \times \big[ \frac{N}{2}, 2N \big]$ such that $\eta_t (i) > \xi_t^{(\rho + \varepsilon)} (i)$ and $\eta_t (i) < \xi_t^{(\rho - \varepsilon)} (i)$, respectively. 
	
	Then, denoting $E_N = E_N^{(1)} \cup E_N^{(2)}$, we have that $\lim_{N \rightarrow \infty} \mathbb{P} [E_N] = 0$. 

\end{thm} 

Assuming \Cref{couplexiaeta}, we can establish \Cref{sixvertexlocalcylinder}.

\begin{proof}[Proof of \Cref{sixvertexlocalcylinder} Assuming \Cref{couplexiaeta}]
	
	It suffices to show that any subsequential limit of the laws of $\big\{ \mathcal{E} |_{[U_N - k, U_N + k] \times [V_N - k, V_N + k]} \big\}$ is equal to the law of $\mathcal{F} |_{[-k, k] \times [-k, k]}$. To that end, fix some increasing sequence $N_1 < N_2 < \cdots $ of positive integers. 

 	By \eqref{etasumestimaten}; the fact that $\eta_y (x), G_y (x) \in [0, 1]$ for each $x$ and $y$; the continuity of $G_y (x)$ in $y$; and \eqref{etay1yy1}, there exists of a sequence of real numbers $\{ \varsigma_N \}$ tending to $0$ such that
	\begin{flalign*}
	\displaystyle\lim_{N \rightarrow \infty} \mathbb{P} \left[  \displaystyle\max_{\substack{0 \le X_1 \le X_2 < N \\ 0 \le Y \le L}} \bigg| \displaystyle\frac{1}{N} \displaystyle\sum_{x = X_1}^{X_2} \eta_Y (x) - \displaystyle\int_{X_1 / N}^{X_2 / N} G_{Y / N} (x) dx \bigg| > \varsigma_N \right] = 0. 
	\end{flalign*}
	
	 Thus a union bound; the Borel-Cantelli lemma; the continuity of $G_y (x)$ at $(u, v)$; and the fact that $G (u, v) = \rho$ together yield an increasing sequence $1 \le m_1 < m_2 < \cdots $ of positive integers and a decreasing sequence $\{ \delta_N \}$ of real numbers tending to $0$ such that the following holds. Denoting $M_j = \lfloor N_{m_j} \delta_j \rfloor$ for each $j \in \mathbb{Z}_{\ge 1}$, we have $\lim_{j \rightarrow \infty} M_j = \infty$ and the almost sure limit 
	\begin{flalign}
	\label{etax1x2limit}
	\displaystyle\lim_{j \rightarrow \infty} \displaystyle\max_{X_1, X_2} \bigg| \displaystyle\frac{1}{X_2 - X_1} \displaystyle\sum_{x = X_1}^{X_2} \eta_{Y_j} (x) - \rho \bigg| = 0,
	\end{flalign}
	
	\noindent where we have set $Y_j = V_{m_j} - M_j$, and $(X_1, X_2) = \big( X_1^{(j)}, X_2^{(j)} \big)$ ranges over all pairs of integers in $[0, N)$ satisfying $U_{m_j} - \delta_j^{1 / 2} N_{m_j} \le X_1 < X_1 + \delta_j^2 N_{m_j} \le X_2 \le U_{m_j} + \delta_j^{1 / 2} N_{m_j}$. 
	
	Now define the (random) particle configuration $\zeta = \zeta^{(M_j)} = \big( \zeta_0 (x) \big) = \big( \zeta_0^{(M_j)} (x) \big) \in \{ 0, 1 \}^{\mathbb{Z}}$ on $\mathbb{Z}$ by setting $\zeta_0 (x) = \eta_{Y_j}^{(N_{m_j})} \big( x + U_{m_j} \big)$ if $|x| \le \frac{8M_j}{1 - b_2}$ and $\zeta_0 (x) = 0$ otherwise. Then \eqref{etax1x2limit} implies that \eqref{etasumlimitx1x2} holds with the $\eta_0$ and $N$ there equal to the $\zeta_0$ and $M_j$ here, respectively. By \Cref{mnclasses}, \Cref{couplexiaeta}, and the Borel-Cantelli lemma, there almost surely exist an integer $j_0 > 0$ and a decreasing sequence $\{ \varepsilon_j \}$ of positive real numbers tending to $0$ such that the following holds, after restricting the $\{ m_j \}$ to a subsequence if necessary. 
	
	For each integer $j \ge 1$, let $\xi_0^{(\rho - \varepsilon_j)}$, $\xi_0^{(\rho)}$, and $\xi_0^{(\rho + \varepsilon_j)}$ denote (random) particle configurations on $\mathbb{Z}$, sampled from the measures $\Upsilon^{(\rho - \varepsilon_j)}$, $\Upsilon^{(\rho)}$, and $\Upsilon^{(\rho + \varepsilon_j)}$, respectively, which are mutually coupled so that $\xi_0^{(\rho + \varepsilon_j)} \ge \xi_0^{(\rho + \varepsilon_i)} \ge \xi_0^{(\rho)} \ge \xi_0^{(\rho - \varepsilon_i)} \ge \xi_0^{(\rho - \varepsilon_j)}$ whenever $i > j$. Further let $\zeta_t^{(M_j)}$, $\xi_t^{(\rho - \varepsilon_j)}$, $\xi_t^{(\rho)}$, and $\xi_t^{(\rho + \varepsilon_j)}$ denote the stochastic six-vertex models with initial data $\zeta_0^{(M_j)}$, $\xi_0^{(\rho - \varepsilon_j)}$, $\xi_0^{(\rho)}$, and $\xi_t	^{(\rho + \varepsilon_j)}$, respectively, which are mutually coupled under the higher rank coupling. Then, $\xi_t^{(\rho - \varepsilon_j)} (x) \le \zeta_t^{(M_j)} (x) \le \xi_t^{(\rho + \varepsilon_j)} (x)$ holds almost surely whenever $j > j_0$, $x \in [-M_j, M_j]$, and $t \in \big[ \frac{M_j}{2}, 2 M_j \big]$.
	
	Therefore, the attractivity of the higher rank coupling together with the Borel-Cantelli lemma and a union bound yield (after restricting to a further subsequence of the $\{ m_j \}$ and increasing $j_0$ if necessary) a nondecreasing sequence of positive integers $(K_1, K_2, \ldots )$ tending to $\infty$ such that $\zeta_{M_j - k}^{(M_j)} (x) = \xi_{M_j - k}^{(\rho)} (x)$ holds almost surely whenever $x \in [-K_j, K_j]$ and $j > j_0$. Next, let $\mathcal{F} = \mathcal{F}^{(\rho)}$ and $\mathcal{G}^{(M_j)}$ denote the (almost surely unique) six-vertex ensembles on $\mathfrak{H} = \mathbb{Z} \times \mathbb{Z}_{\ge 0}$ associated with $\xi_t^{(\rho)}$ and $\eta_t^{(M_j)}$, respectively, for each $j \in \mathbb{Z}_{\ge 1}$. Then \Cref{modelsequal} and another application of the Borel-Cantelli lemma imply (again after if necessary restricting to a subsequence of the $\{ m_j \}$, increasing $j_0$, and decreasing the $\{ K_j \}$ while still having them tend to $\infty$) that $\mathcal{G}^{(M_j)} \big|_{[-K_j, K_j] \times [M_j - k, M_j + k]} = \mathcal{F} |_{[-K_j, K_j] \times [M_j - k, M_j + k]}$ holds almost surely for each $j > j_0$. 
	
	This and \Cref{cylinderhalfplanevertex} together imply that it is possible to couple the stochastic six-vertex models $\eta_t^{(N_{m_j})}$ and $\zeta_t^{(M_j)}$ on $\mathfrak{C}$ and $\mathfrak{H}$, respectively, so that 
	\begin{flalign*} 
	\mathcal{E}^{(N_{m_j})} \big|_{[U_{m_j} - k, U_{m_j} + k] \times [V_{m_j} - k, V_{m_j} + k]} = \mathcal{G}^{(M_j)} |_{[-k, k] \times [M_j - k, M_j + k]} = \mathcal{F} |_{[-k, k] \times [M_j - k, M_j + k]},
	\end{flalign*} 
	
	\noindent holds almost surely for each $j > j_0$ (after restricting to a subsequence of the $\{ m_j \}$ and increasing $j_0$ if necessary). Since the law of $\mathcal{F}^{(\rho)}$ is given by the restriction of the measure $\mu (\rho)$ from \Cref{Translation} to $\mathfrak{H}$, and since $\mu (\rho)$ is invariant with respect to vertical shifts, it follows that any limit point of the laws of $\big\{ \mathcal{E} \big|_{[U_N - k, U_N + k] \times [V_N - k, V_N + k]} \big\}$ is equal to that of $\mathcal{F} |_{[-k, k] \times [-k, k]}$. This yields the theorem.
\end{proof}

\subsection{Additional Couplings}

\label{StationaryApproximateLocal} 

In this section we describe and provide properties of several couplings that will be used in the proof of \Cref{couplexiaeta}. 

We begin by stating the following proposition, whose analog for the ASEP was established as Proposition 12 of \cite{CLEP}. It essentially states that a stochastic six-vertex model with an approximately constant initial density profile will ``almost'' couple with a stationary stochastic six-vertex model, after being run for a sufficiently long time. Its proof will be given in \Cref{ProofApproximate} below.

\begin{prop} 
	
	\label{rhotksum2} 
	
	Adopt the notation and assumptions of \Cref{couplexiaeta}, and fix $\varsigma > 0$. Let $\xi_0^{(\rho)} = \big( \xi_0^{(\rho)} (x) \big) \in \{ 0, 1 \}^{\mathbb{Z}}$ denote a (random) particle configuration on $\mathbb{Z}$, sampled under the measure $\Upsilon^{(\rho)}$ of \Cref{zetarho}. Further let $\xi_t^{(\rho)}$ denote the stochastic six-vertex model with initial data $\xi_0^{(\rho)}$ that is coupled with $\eta_t$ through the higher rank coupling of \Cref{couplinghigherrank}. 
	
	Then, letting $R = \big\lfloor \frac{N}{1 - b_2} \big\rfloor$ and $Y = \big\lfloor \frac{N}{4} \big\rfloor$, there exists a constant $C = C (b_1, b_2, \varsigma) > 0$ such that 
	\begin{flalign*}
	\mathbb{P} \left[  \displaystyle\sum_{x = - 4 R}^{4R}  \big| \eta_Y (x) - \xi_Y^{(\rho)} (x) \big| > \varsigma N \right] < \displaystyle\frac{C}{\log N}.
	\end{flalign*}
\end{prop}

Next, we require two additional stochastic six-vertex models given by the following definition. In what follows, we recall the measure $\Upsilon^{(\theta; \rho)}$ from \Cref{lambdarho}, and also the operator $\mathfrak{S}$ from \Cref{soperator}.  

\begin{definition}
	
\label{etaomegazeta} 

Fix $X \in \mathbb{Z}$, and let $\omega_0 = \omega_0^{(X)} = \omega_0^{(\rho; \varepsilon; X)} = \big( \omega_0 (x) \big) \in \{ 0, 1 \}^{\mathbb{Z}}$ and $\zeta_0 = \zeta_0^{(X)} = \zeta_0^{(\rho; \varepsilon; X)} = \big( \zeta_0 (x) \big) \in \{ 0, 1 \}^{\mathbb{Z}}$ denote two particle configurations on $\mathbb{Z}$, sampled according to the measures $\mathfrak{S}_{-X} \Upsilon^{(\rho; \rho + \varepsilon)}$ and $\mathfrak{S}_{-X} \Upsilon^{(\rho + \varepsilon; \rho)}$, respectively. In particular, $\mathbb{P} \big[ \omega_0 (x) \big] = \rho$ and $\mathbb{P} \big[ \zeta_0 (x) \big] = \rho + \varepsilon$ whenever $x \le X$, and $\mathbb{P} \big[ \omega_0 (x) \big] = \rho + \varepsilon$ and $\mathbb{P} \big[ \zeta_0 (x) \big] = \rho$ whenever $x > X$. Further let $\omega_t$ and $\zeta_t$ denote the stochastic six-vertex models with initial data $\omega_0$ and $\zeta_0$, respectively. 

\end{definition}

We now have the following lemma, whose statement (and proof) is similar to that of Lemma 10 and Lemma 11 of \cite{CLEP} (which address the ASEP). It considers the systems $\omega_t$, $\zeta_t$, and $\xi_t^{(\rho + \varepsilon)}$ when coupled under the higher rank coupling and provides a lower bound on the number of second class particles in a specific interval. In what follows, we recall the function $\varphi$ from \eqref{kappafunction}.

\begin{lem}

\label{zetaomegaetaxi}

Adopt the notation of \Cref{couplexiaeta} and \Cref{etaomegazeta}; let $\gamma > 0$ denote a real number; let $M > 0$ denote an integer; and assume that $\omega_0 \le \xi_0^{(\rho + \varepsilon)}$ and $\zeta_0 \le \xi_0^{(\rho + \varepsilon)}$ almost surely. Defining $u < v$ by 
\begin{flalign}
\label{uvidentity} 
u = \varphi' (\rho) + \displaystyle\frac{7 \varepsilon \varphi'' (\rho)}{8}; \qquad  v = \varphi' (\rho) + \displaystyle\frac{5 \varepsilon \varphi'' (\rho)}{8},
\end{flalign} 

\noindent there exist constants $c_1 = c_1 (b_1, b_2) \in (0, 1)$ and $c_2 = c_2 (b_1, b_2) \in (0, 1)$ such that the following two statements hold. 

\begin{enumerate}
	
	\item Couple $\omega_t \le \xi_t^{(\rho + \varepsilon)}$ under the higher rank coupling. If $\varepsilon < c_1$ and $\gamma \le c_1 \varepsilon$, then
	\begin{flalign*} 
	\displaystyle\lim_{M \rightarrow \infty} \mathbb{P} \left[ \displaystyle\sum_{x = \lfloor v M \rfloor}^{\lfloor (v + \gamma) M \rfloor} \big( \xi_M^{(\rho + \varepsilon)} (X + x) - \omega_M (X + x) \big) < c_1 \varepsilon \gamma M \right] = 0.
	\end{flalign*}

	\item Couple $\zeta_t \le \xi_t^{(\rho + \varepsilon)}$ under the higher rank coupling. If $\varepsilon < c_2$ and $\gamma \le c_2 \varepsilon$, then
	\begin{flalign*}
	\displaystyle\lim_{M \rightarrow \infty} \mathbb{P} \left[ \displaystyle\sum_{x = \lfloor (u - \gamma) M \rfloor}^{\lfloor u M \rfloor} \big( \xi_M^{(\rho + \varepsilon)} (X + x) - \zeta_M (X + x) \big) < c_2 \varepsilon \gamma M \right] = 0.
	\end{flalign*}

\end{enumerate} 

\end{lem}

\begin{proof}
	
	Both of these statements follow in a similar way from \Cref{limitdouble}. Hence, we only establish the latter (with respect to $\zeta$); we may assume that $X = 0$. 
	
	To that end, recall that the entropy solution $G_1 (x) = G(x, 1)$ of the equation \eqref{gxydefinition} on $\mathbb{R} \times \mathbb{R}_{\ge 0}$ with initial data $G_0 (x) = (\rho + \varepsilon) \textbf{1}_{x \le 0} + \rho \textbf{1}_{x > 0}$ is explicitly given by 
	\begin{flalign}
	\label{gxidentity2} 
	\begin{aligned} 
	& G_1 (x) = \rho + \varepsilon \quad \text{if $x \le \varphi' (\rho + \varepsilon)$}; \qquad G_1 (x) = \rho \quad \text{if $x \ge \varphi' (\rho)$}; \\
	& G_1 (x) = (\varphi')^{-1} (x) = \displaystyle\frac{\sqrt{\kappa x^{-1}} - 1}{\kappa - 1} \quad \text{if $\varphi' (\rho + \varepsilon) < x < \varphi' (\rho)$}.
	\end{aligned} 
	\end{flalign}  
	
	Now, \Cref{limitdouble}, the fact that $\xi_t^{(\rho + \varepsilon)}$ is stationary for any $t \ge 0$, and a large deviations estimate for sums of independent $0-1$ Bernoulli random variables together imply for any $\varsigma > 0$ that 
	\begin{flalign}
	\label{msumestimate}
	\displaystyle\lim_{M \rightarrow \infty} \mathbb{P} \left[ \displaystyle\frac{1}{M} \displaystyle\sum_{x = \lfloor (u - \gamma) M \rfloor}^{\lfloor u M \rfloor} \big( \xi_M^{(\rho + \varepsilon)} (x) - \zeta_M (x) \big) < \displaystyle\int_{u - \gamma}^u \big( \rho + \varepsilon - G_1 (x) \big) dx - \varsigma \right] = 0. 
	\end{flalign}

	Next, a Taylor expansion and the fact that $\varphi'' (\rho) < 0$ together imply that $\varphi' (\rho + \varepsilon) < u + \frac{\varepsilon \varphi'' (\rho)}{16}$, for sufficiently small $\varepsilon$. Thus, from \eqref{gxidentity2}, a Taylor expansion, and the fact that any derivative of $\varphi$ is bounded away from $0$, it quickly follows that there exists a constant $c = c(b_1, b_2) > 0$ such that $\int_{u - \gamma}^u \big( \rho + \varepsilon - G_1 (x) \big) > c \varepsilon \gamma$ for $\gamma < c \varepsilon$. The second statement of the lemma (with $c_2 = \frac{c}{2}$) therefore follows from \eqref{msumestimate} upon setting $\varsigma = \frac{c \varepsilon \gamma}{2}$. 
\end{proof}

\subsection{Proof of \Cref{couplexiaeta}}  

\label{StationaryLocal}

In this section we establish \Cref{couplexiaeta}. To that end, it suffices to show that $\lim_{N \rightarrow \infty} \mathbb{P} \big[ E_N^{(1)} \big] = 0 = \lim_{N \rightarrow \infty} \mathbb{P} \big[ E_N^{(2)} \big]$. We will only show the first of these two equalities, as the proof of the latter is entirely analogous. Given the content of the previous sections, this will follow Section 4 of \cite{CLEP}. 

Throughout this section, we set $R = \big\lfloor \frac{N}{1 - b_2} \big\rfloor$, $Y = \big\lfloor \frac{N}{4} \big\rfloor$, and $Z = \big\lfloor \frac{N}{2} \big\rfloor$. Additionally, let $\xi_0^{(\rho)} = \big( \xi_0^{(\rho)} (x) \big) \in \{ 0, 1 \}^{\mathbb{Z}}$ denote a (random) particle configuration on $\mathbb{Z}$ sampled from the measure $\Upsilon^{(\rho)}$ from \Cref{zetarho}. Assume that $\xi_0^{(\rho)}$ and $\xi_0^{(\rho + \varepsilon)}$ are coupled so that $\xi_0^{(\rho)} \le \xi_0^{(\rho + \varepsilon)}$. Let $\big( \xi_t^{(\rho)} \big)$ denote the stochastic six-vertex model with initial data $\xi_0^{(\rho)}$, and mutually couple $\eta_t$ and $\xi_t^{(\rho)} \le \xi_t^{(\rho + \varepsilon)}$ under the higher rank coupling of \Cref{couplinghigherrank} (recall \Cref{etaxim}). 

Now we define several events that will be useful to us. Let $A = A_N$ denote the event on which at least one of the following three possibilities occur. 

\begin{enumerate} 
	
	\item There exists a particle in $ \eta_Y \cap \big( (-\infty, - 2R] \cup [2R, \infty) \big)$ that enters the interval $[-N, N]$ at some time $t \in [Y, 2N]$. 
	
	\item There exists a particle in  $\eta_Z \cap \big( (-\infty, 2R] \cup [2R, \infty) \big)$ that enters the interval $[-N, N]$ at some time $t \in [Z, 2N]$. 
	
	\item There exists a particle in $\eta_Y \cap \big( (-\infty, -4R] \cup [4R, \infty) \big)$ that enters the interval $[-3R, 3R]$ at some time $t \in [Y, Z]$. 
	
\end{enumerate} 
	
\noindent Furthermore, for any $\varsigma > 0$, define the event 
\begin{flalign*}
B_N (\varsigma) = \left\{ \displaystyle\sum_{x = -4R}^{4R} \big| \eta_Y (x) - \xi_Y^{(\rho)} (x) \big| > \varsigma N \right\}.
\end{flalign*}

\noindent Next, for any $\delta > 0$, and define the set 
\begin{flalign*} 
\mathcal{X} = \mathcal{X}_{\delta} = \mathcal{X}_{\delta; N} = \left\{ X: \displaystyle\frac{X}{\lfloor \delta N \rfloor} \in \mathbb{Z} \right\} \cap [-4R, 4R],
\end{flalign*} 

\noindent consisting of integer multiples of $\lfloor \delta N \rfloor$ in the interval $[-4R, 4R]$. Let $\omega_0^{(X)}, \zeta_0^{(X)}$ for each $X \in \mathcal{X}$ be particle configurations as in \Cref{etaomegazeta}. Assume that $\omega_0^{(X)}$ and $\zeta_0^{(X)}$ are coupled with $\xi_Y^{(\rho)}$ and $\xi_Y^{(\rho + \varepsilon)}$ so that $\xi_Y^{(\rho)} \le \omega_0^{(X)} \le \xi_Y^{(\rho + \varepsilon)}$ and $\xi_Y^{(\rho)} \le \zeta_0^{(X)} \le \xi_Y^{(\rho + \varepsilon)}$, for each $X \in \mathcal{X}$. 

Further couple the $\omega_t^{(X)}$ and $\zeta_t^{(X)}$ with $\eta_{Y + t}$, $\xi_{Y + t}^{(\rho)}$, and $\xi_{Y + t}^{(\rho + \varepsilon)}$ under the $n = 3$ case of the higher rank coupling (recall \Cref{etaxim}), where we set the $\eta_t$ there to be the $\eta_{Y + t}$ here and the $\xi_t^{(1)} \le \xi_t^{(2)} \le \xi_t^{(3)}$ there to be the $\xi_{Y + t}^{(\rho)} \le \omega_t^{(X)} \le \xi_{Y + t}^{(\rho + \varepsilon)}$ (or $\xi_{Y + t}^{(\rho)} \le \zeta_t^{(X)} \le \xi_{Y + t}^{(\rho + \varepsilon)}$) here. 

Recalling the quantities $u$ and $v$ from \eqref{uvidentity} and the constants $c_1$ and $c_2$ from \Cref{zetaomegaetaxi}, we define for any integer $M \ge Y$ and real number $\gamma > 0$ the events 
\begin{flalign*} 
& F_M (\omega; X; \gamma) = \left\{ \displaystyle\sum_{x = \lfloor v (M - Y) \rfloor}^{\lfloor (v + \gamma) (M - Y) \rfloor} \big( \xi_M^{(\rho + \varepsilon)} (X + x) - \omega_{M - Y}^{(X)} (X + x) \big) < \displaystyle\frac{c_1 \varepsilon \gamma N}{4} \right\}; \\
& F_M (\zeta; X; \gamma) = \left\{  \displaystyle\sum_{x = \lfloor (u - \gamma) (M - Y) \rfloor}^{\lfloor u (M - Y) \rfloor} \big( \xi_M^{(\rho + \varepsilon)} (X + x) - \zeta_{M - Y}^{(X)} (X + x) \big) < \displaystyle\frac{c_2 \varepsilon \gamma N}{4} \right\}; \\
& \quad F_M (\omega; \delta; \gamma) = \bigcap_{X \in \mathcal{X}_{\delta; N}} F_M (\omega; X; \gamma); \qquad F_M (\zeta; \delta; \gamma) = \bigcap_{X \in \mathcal{X}_{\delta; N}} F_M (\zeta; X; \gamma). 
\end{flalign*}

\noindent Then, for any $\varepsilon < \min \{ c_1, c_2 \}$, $\varsigma, \delta > 0$, and $0 < \gamma \le \varepsilon \min \{ c_1, c_2 \}$, \Cref{xtxt1}, \Cref{rhotksum2}, \Cref{zetaomegaetaxi}, a large deviations estimate for sums of independent geometric random variables, and a union bound yield a constant $c = c(b_2) > 0$ such that 
\begin{flalign}
\label{probabilityab}
\mathbb{P} [A_N] \le c^{-1} e^{-cN}; \quad \displaystyle\lim_{N \rightarrow \infty} \mathbb{P} \big[ B_N (\varsigma) \big] = 0; \quad \displaystyle\lim_{N \rightarrow \infty} \mathbb{P} \big[ F_Z (\omega; \delta; \gamma) \big] = 0 = \displaystyle\lim_{N \rightarrow \infty} \mathbb{P} \big[ F_Z (\zeta; \delta; \gamma) \big]. 
\end{flalign}

\noindent Thus, defining the events 
\begin{flalign*} 
F_N (\omega; \delta; \varsigma; \gamma) = A_N \cup B_N (\varsigma) \cup F_Z (\omega; \delta; \gamma); \qquad F_N (\zeta; \delta; \varsigma; \gamma) = A_N \cup B_N (\varsigma) \cup F_Z (\zeta; \delta; \gamma),
\end{flalign*} 

\noindent it follows from \eqref{probabilityab} that 
\begin{flalign}
\label{probabilityab1}
\displaystyle\lim_{N \rightarrow \infty} \mathbb{P} \big[ F_N (\omega; \delta; \varsigma; \gamma) \big] = 0 = \displaystyle\lim_{N \rightarrow \infty} \mathbb{P} \big[ F_N (\zeta; \delta; \varsigma; \gamma) \big]. 
\end{flalign}

Now let $\big( p_t (k) \big)$ denote the particle configuration associated with the stochastic six-vertex model $\eta = \eta^{(N)}$. For any integers $U, V \in [-3R, 3R]$, let $E_N (U; V)$ denote the event on which there exists a $k \in \mathbb{Z}$ for which $p_Y (k) = U$; $p_Z (k) = V$; and $\xi_Y^{(\rho + \varepsilon)} (U) = 0 = \xi_Z^{(\rho + \varepsilon)} (V)$. Stated alternatively, $E_N (U; V)$ denotes the event on which there exists a particle in $\eta$ at site $U$ and time $Y$ that moves to site $V$ at time $Z$, which is uncoupled with a particle in $\big( \xi_t^{(\rho + \varepsilon)} \big)$ for each $t \le Z$. Then, since $A_N \subseteq F_N (\omega; \delta; \varsigma; \gamma) \cap F_N (\zeta; \delta; \varsigma; \gamma)$, we have that 
\begin{flalign}
\label{en1}
\begin{aligned}
E_N^{(1)} & \subseteq F_N (\omega; \delta; \varsigma; \gamma) \cup F_N (\zeta; \delta; \varsigma; \gamma) \cup \bigcup_{-3R \le U \le V \le 3R} E_N (U; V) \\
& = \bigcup_{-3R \le U \le V \le 3R} \big( F_N (\omega; \delta; \varsigma; \gamma)^c \cap E_N (U; V) \big) \cap \big( F_N (\zeta; \delta; \varsigma; \gamma)^c \cap E_N (U; V) \big) \\
& \qquad \qquad \cup F_N (\omega; \delta; \varsigma; \gamma) \cup F_N (\zeta; \delta; \varsigma; \gamma). 
\end{aligned}
\end{flalign}

The following two lemmas bound $\mathbb{P} \big[ F_N (\omega; \delta; \varsigma; \gamma)^c \cap E_N (U; V) \big]$ and $\mathbb{P} \big[ F_N (\zeta; \delta; \varsigma; \gamma)^c \cap E_N (U; V) \big]$ when $V - U$ is sufficiently small or large, respectively. Here, we recall $u$ and $v$ from \eqref{uvidentity}. 

\begin{lem}

\label{abfprobability} 

There exists a constant $c = c (\varepsilon, \varsigma) > 0$ such that the following holds. Assume that $c_1 \varepsilon \gamma > 16 \delta + 8 \varsigma$ and that $U, V \in [-3R, 3R]$ are integers satisfying $V - U < v (Z - Y)$. Then $\mathbb{P} \big[ F_N (\omega; \delta; \varsigma; \gamma)^c \cap E_N (U; V) \big] < c^{-1} e^{-cN}$.

\end{lem}

\begin{proof}

	Throughout this proof, let $X \in \mathcal{X}$ be such that $U \le X \le U + \delta N$. Abbreviate $\omega = \omega^{(X)}$, and mutually couple $\eta_t$ and $\big( \xi_t^{(\rho)}, \omega_{t - Y}, \xi_t^{(\rho + \varepsilon)} \big)$ under the higher rank coupling (recall \Cref{etaxim}). Let $\big( q_t (k) \big)$ denote the positions of the fourth class particles in $\xi_t^{(\rho + \varepsilon)}$ (namely, those not in $\eta_t \cup \omega_{t - Y}	$). 
	
	In order to establish this lemma, we will show that each fourth class particle that starts to the left of $U$ at time $Y$ and ends to the right of $V$ at time $Z$ provides an ``opportunity'' for $p_t (k)$ to couple; the number of such particles will be bounded below by a positive multiple of $N$, due to the event $F_N (\omega; X; \gamma)$. Next, the probability that $p (k)$ ``declines'' the opportunity presented by any such particle will be bounded away from $1$, and these decisions will essentially be independent. Therefore, the probability that $p (k)$ remains uncoupled (declines all opportunities) until time $Z$ will decay exponentially in $N$. 
	
	To implement this in more detail, we condition on $\eta_t$ for $t \le Z$, and define the event $G_N (U; V) = \big\{ \eta_Y (U) = 1 = \eta_{Z} (V) \big\}$. Restricting to $B_N (\varsigma) \cap G_N (U; V)$, let $W_t = p_t (k) \in \eta_t$ denote the tagged position of the particle (which is fourth class under the coupling described above) in $\eta$ such that $W_Y = U$ and $W_{Z} = V$. 
	
	Next, similarly to in Section 4 of \cite{CLEP}, define the sequence of integers $T_0 < T_1 < T_2 < \cdots $ by setting $T_0 = Y$ and the remaining $T_i$ as follows. For each $i > 0$, let $T_i$ denote the minimum integer larger than $T_{i - 1}$ such that the following two conditions hold. 
	
	\begin{enumerate} 
	
	\item The particle $p_{T_i - 1} (k) \in \eta_{T_i - 1}$ is uncoupled with one in $\eta_{T_i - 1}^{(\xi + \varepsilon)}$, that is, $\xi_{T_i - 1}^{(\rho + \varepsilon)} (W_{T_i - 1}) = 0$. 
		
	\item A fourth class particle in $\xi_{T_i - 1}^{(\rho + \varepsilon)}$ to the left of $W_{T_i - 1}$ at time $T_i - 1$ jumped either to or to the right of $W_{T_i}$ at time $T_i$. Stated alternatively, there exists $m \in \mathbb{Z}$ such that $q_{T_i - 1} (m) < W_{T_i - 1}$ and $r_{T_i} (m) \ge W_{T_i}$, where $r_t (m)$ denotes the tagged position (recall \Cref{higherrankclass}) of the particle $q_{T_i - 1} (m)$. 
	
	\end{enumerate}

	 If no such integer exists, then we set $T_j = \infty$ for each $j \ge i$. Observe that $T_j = \infty$ if either the particle $p_t (k)$ couples with one of $\xi_t^{(\rho + \varepsilon)}$ for some $t \le T_{j - 1}$, or if no particle in $\xi_t^{(\rho + \varepsilon)}$ not in $\omega_t \cup \eta_t$ jumps from the left of $p_t (k)$ to the right of (or to) $p_t (k)$ for any $t \ge T_{j - 1}$. These times $t = T_i$ indicate when $p_t (k)$ is presented with the ``opportunity'' to couple. 
	
	Under the above notation, define the event $\Omega_j = \big\{ r_{T_j} (m) > W_{T_j} \big\}$, for each integer $j \ge 1$. Stated alternatively, $\Omega_j$ denotes the event on which there exists a fourth particle in $\xi_{T_j - 1}$ that jumps to the right of $p_{T_j } (k)$ at time $T_j$. This happens if and only if $p_{T_j} (k)$ does not (``declines'' to) couple with a particle in $\xi_{T_j}^{(\rho + \varepsilon)}$. Therefore, letting $M \ge 0$ denote the minimal integer such that $T_{M + 1} > Z$, it follows that 
	\begin{flalign} 
	\label{egh} 
	E_N (U; V) \subseteq G_N (U; V) \cap \bigcap_{j = 1}^M \Omega_j.  
	\end{flalign}

	Now, observe that if $X$ is a $b_2$-geometric random variable then $\mathbb{P} [X = i | X \ge i] \ge b_2$, for any integer $i \ge 0$. Recalling the definition of the sampling for the multi-class stochastic six-vertex model (from \Cref{HigherRank}), let us apply this fact when the $X$ here is equal to the $j^{(4)} \big( q_{T_i - 1} (m) \big)$ there (which denotes the number of spaces that $q_{T_i - 1} (m)$ attempts to jump to the right), and the $i$ here is equal to $W_{T_i} - q_{T_i - 1} - h$ there, where $h$ denotes the number of particles of class at most $3$ in $\xi_{T_i - 1}^{(\rho + \varepsilon)} \cap \big[ q_{T_i - 1} (m), W_{T_i} \big]$ that decided to stay. It follows that $\mathbb{P} \big[ r_{T_j} (m) = W_{T_j} \big| \bigcap_{i = 1}^{j - 1} \Omega_i \big] \ge b_2$, for each $j \in [1, M]$. Hence, $\mathbb{P} \big[ \Omega_j \big| \bigcap_{i = 1}^{j - 1} \Omega_i \big] \le 1 - b_2$ for each $j \ge 1$, and so  
	\begin{flalign}
	\label{gjprobability} 
	\mathbb{P} \left[ D_n \cap \bigcap_{j = 1}^M \Omega_j \right] \le (1 - b_2)^n, \quad \text{where} \quad D_n = \{ M \ge n \},
	\end{flalign} 
	
	\noindent for any integer $n \ge 0$. 
	
	We next claim that $D_n \subseteq F_N (\omega; \delta; \varsigma; \gamma)^c \cap G_N (U; V)$ for any $n \le \varsigma N$. To that end, observe that the facts that $F_N (\omega; \delta; \varsigma; \gamma)^c \subseteq F_Z (\omega; X; \gamma)^c$; that $U \le X \le U + \delta N$; that $\xi^{(\rho + \varepsilon)} \ge \omega$; and that $c_1 \varepsilon \gamma \ge 16 \delta + 8 \varsigma$ together yield  
	\begin{flalign*}
	\displaystyle\sum_{x = v (Z - Y)}^{(v + \gamma) (Z - Y)} & \big( \xi_Z^{\rho + \varepsilon} (x + U) - \omega_{Z - Y} (x + U) \big) \\
	& \ge \left( \displaystyle\sum_{x = v (Z - Y)}^{(v + \gamma) (Z - Y)} \big( \xi_Z^{\rho + \varepsilon} (x + X) - \omega_{Z - Y} (x + X) \big) - 2 \delta N \right) \textbf{1}_{F_N (\omega; \delta; \varsigma; \gamma)^c} \\
	 & \ge \left( \displaystyle\frac{c_1 \varepsilon \gamma N}{4} - 2 \delta N \right) \textbf{1}_{F_N (\omega; \delta; \varsigma; \gamma)^c} \ge (2 \varsigma + 2 \delta) N \textbf{1}_{F_N (\omega; \delta; \varsigma; \gamma)^c}.
	\end{flalign*}
	
	In particular, since $V < U + v (Z - Y)$; since $X \le U \le X + \delta N$; and since $\omega_0$ and $\zeta_Y^{(\rho + \varepsilon)}$ coincide to the right of site $X$, this implies that upon restricting to $F_N (\omega; \delta; \varsigma; \gamma)^c$, there exists a set $\mathcal{T}$ consisting of at least $2 \varsigma N$ particles in $\xi^{(\rho + \varepsilon)} \setminus \omega$ that passed from at or to the left of $U$ at time $Y$ to the right of $V$ at time $Z$. These are third class particles in the system $\big( \xi^{(\rho)}, \omega, \xi^{(\rho + \varepsilon)} \big)$ coupled under the higher rank coupling. 
	
	Since $F_N (\omega; \delta; \varsigma; \gamma)^c \subseteq B_N (\varsigma)^c$, there exist at most $\varsigma N$ particles in $\eta_Y \cap [-4R, 4R]$ that are not coupled with a particle in $\xi^{(\rho)}$. Thus, the facts that $F_N (\omega; \delta; \varsigma; \gamma)^c \subseteq A_N^c$ and that $\xi^{(\rho + \varepsilon)} \ge \omega \ge \xi^{(\rho)}$ together imply that at most $\varsigma N$ particles in $\mathcal{T}$ will couple with some particle $p_t (n) \ge -3R$ for $n \le k$ (that is, a particle in $\eta_t$ between $-3R$ and $p_t (k)$) and $t \in [Y, Z]$. Removing these particles from $\mathcal{T}$ forms a set $\mathcal{S}$ consisting of at least $\varsigma N$ particles.
	
	Denote the trajectories of these particles in $\mathcal{S}$ (as tagged third class particles in the coupled system $\big( \xi^{(\rho)}, \omega, \xi^{(\rho + \varepsilon)} \big)$), by $S_t (1) < S_t (2) < \ldots < S_t (K)$, for some integer $K \ge \varsigma N$. On the event $G_N (U; V)$ there exists for each $i \in [1, K]$ an integer $t_i \in [Y + 1, Z]$ for which $S_{t_i - 1} (i) < W_{t_i - 1} \le W_{t_i} \le S_{t_i} (i)$. Hence, $t_i \in \{ T_1, T_2, \ldots , T_M \}$ for each $i \in [1, K]$. Since the $\{ t_i \}$ are all distinct, it follows that $D_n \subseteq F_N (\omega; \delta; \varsigma; \gamma)^c \cap G_N (U; V)$ for any $n \ge \varsigma N$. 
	
	Now the lemma follows from \eqref{egh} and \eqref{gjprobability}. 
	\end{proof}

	\begin{lem}
		
		\label{abfprobability1}

		There exists a constant $c = c (\varepsilon, \varsigma, \delta) > 0$ such that the following holds. Assume that $c_2 \varepsilon \gamma > 20 \delta + 4 \varsigma$ and that $U, V \in [-3R, 3R]$ are integers satisfying $V - U > u (Z - Y)$. Then $\mathbb{P} \big[ F_N (\zeta; \delta; \varsigma; \gamma)^c \cap E_N (U; V) \big] < c^{-1} e^{-cN}$.		
		
	\end{lem} 

	\begin{proof} 
		
	Again let $X \in \mathcal{X}$ be such that $U \le X \le U + \delta N$, abbreviate $\zeta = \zeta^{(X)}$, and mutually couple $\eta_t$ and $\big( \xi_t^{(\rho)}, \zeta_{t - Y}, \xi_t^{(\rho + \varepsilon)} \big)$ under the higher rank coupling (recall \Cref{etaxim}). The proof of this lemma will proceed similarly to that of \Cref{abfprobability}, except that $p_t (k)$ will now be presented with an opportunity to couple for every fourth class particle in this coupled system that passes from the right of $U$ at time $Y$ to the left of $V$ at time $Z$. 
	
	To that end, let us condition on $\big( \xi_t^{(\rho)}, \zeta_{t - Y}, \xi_t^{(\rho + \varepsilon)} \big)$ for $t \le Z$ and on $\eta_t$ for $t \le Y$. Restrict to the event $F_N (\zeta; \delta; \varsigma; \gamma) \cap \big\{ \eta_Y (U) = 1 \big\}$, and let $W_t = p_t (k) \in \eta_t$ denote the tagged position of the particle in $\eta$ (under the coupling described above) such that $W_Y = U$. Further let $M = \min \{ Z, T - 1 \}$, where $T$ denotes the minimal time (if it exists) at which $W_T \in \xi_T^{(\rho + \varepsilon)}$ (that is, when $p_T (k)$ couples with a particle in $\xi_T^{(\rho + \varepsilon)}$). 
	
	Since $F_N (\zeta; \delta; \varsigma; \gamma)^c \subseteq F_Z (\zeta; X; \gamma)^c$ and $c_2 \varepsilon \gamma > 20 \delta + 4 \varsigma$, we deduce that 
	\begin{flalign*}
	 \displaystyle\sum_{x = (u - \gamma) (Z - Y)}^{u (Z - Y)} \big( \xi_Z^{\rho + \varepsilon} (x + U) - \zeta_{Z - Y} (x + U) \big) & \ge \left( \displaystyle\frac{c_2 \varepsilon \gamma N}{4} - 2 \delta N \right) \textbf{1}_{F_N (\zeta; \delta; \varsigma; \gamma)^c} \\
	 & \ge (\varsigma + 3 \delta) N \textbf{1}_{F_N (\zeta; \delta; \varsigma; \gamma)^c}. 
	\end{flalign*}
	
	Since $U \le X \le U + \delta N$ and $V - U > u (Z - Y)$, this implies, upon restricting to $F_N (\zeta; \delta; \varsigma; \gamma)^c$, that there exists a set $\mathcal{T}$ consisting of at least $(\varsigma + \delta) N$ particles in $\xi^{(\rho + \varepsilon)} \setminus \zeta$ that passed from the right of $U$ at time $Y$ to the left of $V$ at time $Z$. Since $F_N (\zeta; \delta; \varsigma; \gamma)^c \subseteq B_N (\varsigma)^c$, there exist at most $\varsigma N$ particles in $\eta_Y \cap [-4R, 4R]$ that are not coupled with one in $\xi^{(\rho)}$. Therefore, since $F_N (\zeta; \delta; \varsigma; \gamma)^c \subseteq A_N^c$, at most $\varsigma N$ particles in $\mathcal{T}$ will couple with a particle in $p_t (n) \ge -3R$ for $n \le k$ and $t \in [Y, Z]$. Removing these particles from $\mathcal{T}$ forms a set $\mathcal{S}$ consisting of at least $\delta N$ particles. 
	
	Denote the tagged positions of the particles in $\mathcal{S}$ as third-class particles in the coupled system $\big( \xi^{(\rho)}, \zeta, \xi^{(\rho + \varepsilon)} \big)$ by $S_t (1) < S_t (2) < \cdots < S_t (K)$, for some $K \ge \delta N$. For each $i \ge 0$, let $r_t$ denote the maximal integer $r \ge 1$ for which there exists $m \in [1, K]$ for which $W_{t - 1} < S_{t - 1} (m) \le S_t (m + r - 1) \le W_t$; if no such integer exists, then we set $r_t = 0$. Stated alternatively, $r_t$ denotes the number of particles in $\mathcal{S}$ over which $p_t (k)$ jumps at time $t$. 
	
	Under this notation, define the events 
	\begin{flalign*} 
	& G_t = \{ r_t \ge 1 \}; \qquad \qquad \qquad \qquad \qquad \Omega_t = \big\{ W_t > S_t (m + r_t - 1) \big\}; \\
	& G (\nu) = \left\{ \displaystyle\sum_{t = Y}^Z \textbf{1}_{G_t} \ge \nu N \right\}; \qquad \qquad \quad H = \{ W_Y = U \} \cap \{ W_Z \ge V \},
	\end{flalign*}
	
	\noindent for any $\nu > 0$. In particular, $G_t$ denotes the event on which $p_t (k)$ jumps over at least one particle, and $\Omega_t$ denotes the one on which $p_t (k)$ does not couple with a particle in $\mathcal{S}$. Then,
	\begin{flalign}
	\label{eomegah}
	E_N (U; V) \subseteq H \cap \bigcap_{t = Y}^Z \Omega_t.
	\end{flalign}
	
	Now, recall that, if $X$ is a $b_2$-geometric random variable, then $\mathbb{P} [X > i | X \ge i] \le 1 - b_2$, for any integer $i \ge 0$. Thus, the definition of the sampling procedure for the multi-class stochastic six-vertex model (from \Cref{HigherRank}) implies that $\textbf{1}_{G_t} \mathbb{P} \big[ \Omega_t \big| \bigcap_{j = Y}^{t - 1} \Omega_j \big] \le 1 - b_2$ for any $t \in [Y, M]$. This sampling procedure also implies that there exists a $b_2$-geometric random variable $X_t$, given by the quantity $j_t^{(4)} \big( p_{t - 1} (k) \big)$ from \Cref{HigherRank} (which denotes how many spaces $p (k)$ attempts to jump to the right at time $t$), such that $r_t$ satisfies $r_t \le X_t$ almost surely. 
	
	Furthermore, since $W_t$ must jump over each particle in $\mathcal{S}$ on the event $H$, we must have that $\sum_{t = Y}^Z X_t \ge \sum_{t = Y}^Z r_t \ge |\mathcal{S}| \ge \delta N$, upon restricting to $H \cap F_N (\zeta; \delta; \varsigma; \gamma)$. Since the $\{ X_t \}$ are mutually independent, a large deviations estimate for sums of independent geometric random variables and a union bound quickly yield the existence of constants $c_3 = c_3 (b_2) > 0$ and $c = c (b_2, \varsigma, \delta) > 0$ such that, if $\nu < c_3 \delta$, then the probability that there exists a subset $\mathcal{D} \subseteq [Y, Z]$ such that $|\mathcal{D}| \le \nu N$ and $\sum_{t \in \mathcal{D}} X_t \ge \delta N$ is at most $c^{-1} e^{- cN}$. Thus, $\mathbb{P} \big[ G (\nu)^c \cap H \cap F_N (\zeta; \delta; \varsigma; \gamma) \big]  \le c^{-1} e^{-cN}$ if $\nu < c_3 \delta$.
	
	Hence, since $\textbf{1}_{G_t} \mathbb{P} \big[ \Omega_t \big| \bigcap_{j = Y}^{t - 1} \Omega_j \big] \le 1 - b_2$ for any $t \in [Y, M]$, we deduce if $\nu < c_3 \delta$ that 
	\begin{flalign*} 
	\mathbb{P} \left[ H \cap F_N (\zeta; \delta; \varsigma; \gamma) \cap \bigcap_{t = Y}^Z \Omega_t \right] & \le \mathbb{P} \left[ G (\nu) \cap \bigcap_{t = Y}^Z \Omega_t \right] + \mathbb{P} \big[ G (\nu)^c \cap H \cap F_N (\zeta; \delta; \varsigma; \gamma) \big] \\
	& \le (1 - b_2)^{\nu N} + c^{-1} e^{-cN}, 
	\end{flalign*}
	
	\noindent after decreasing $c$ if necessary. By \eqref{eomegah}, this implies the lemma. 
\end{proof} 

Now we can establish \Cref{couplexiaeta}.

\begin{proof}[Proof of \Cref{couplexiaeta} Assuming \Cref{rhotksum2}]

	We may assume that $\varepsilon < \min \{ c_1, c_2 \}$. Let $\gamma > 0$ satisfy $\gamma < \varepsilon \min \big\{ c_1, c_2 \big\}$, and let $\varsigma, \delta > 0$ satisfy $\max \{ \varsigma, \delta \} < \varepsilon \gamma \min \big\{ \frac{c_1}{24}, \frac{c_2}{24} \big\}$. Then, \eqref{en1}, \eqref{probabilityab1}, \Cref{abfprobability}, \Cref{abfprobability1}, the fact that $U$ and $V$ either satisfy $V - U < v (Z - Y)$ or $V - U > u (Z - Y)$ (since $u < v$), and a union bound together yield $\lim_{N \rightarrow \infty} \mathbb{P} \big[ E_N^{(1)} \big] = 0$. As mentioned previously, the proof that $\lim_{N \rightarrow \infty} \mathbb{P} \big[ E_N^{(2)} \big] = 0$ is entirely analogous and therefore omitted. This yields the theorem. 
\end{proof}

\section{Proof of \Cref{rhotksum2}}

\label{ProofApproximate}

In this section we establish \Cref{rhotksum2}. To that end, first observe from \Cref{guvpqsum} that the any two stochastic six-vertex models $\eta_t$ and $\xi_t$ coupled under the higher rank coupling approximately become locally ordered after run for a sufficiently long time. This can be used to deduce that, on almost every interval $I \subset \mathbb{Z}$ that is not too large, the law of $\eta_t$ is likely to be approximately a product Bernoulli (stationary) distribution with some density $\rho_I \in [0, 1]$. Thus, assuming that the initial data $\eta_0$ satisfies \eqref{etasumlimitx1x2}, one must show that almost all of the $\rho_I$ are all approximately equal to $\rho$. 

The proof of this statement will be partially based on the framework introduced by Kosygina used to establish a two-block estimate for the ASEP, given by Theorem 4.1 of \cite{BEHSL}. More specifically we will approximate the total current, that is the total distance that all particles moved to the right, in both a ``global'' (in terms of $\rho)$ and ``local'' (in terms of the $\rho_I$) way. The former will be done in \Cref{Current1} and the latter in \Cref{Current2}. Then, in \Cref{ProofStochastic}, we will equate these two expressions; this with the concavity of the function $\varphi$ from \eqref{kappafunction} will then imply that $\rho_I \approx \rho$ for most $I$.

\subsection{Global Approximation} 

\label{Current1} 

In this section we establish the following lemma that ``globally'' estimates the total current of the stochastic six-vertex model.

\begin{lem} 
	
	\label{heightcompare} 
	
	Adopt the notation and assumptions of \Cref{couplexiaeta}, and fix a real number $\varpi \in (0, 1)$. Let $\textbf{\emph{p}}_t = \big( p_t (1), p_t (2), \ldots , p_t (M) \big)$ denote the particle position sequence associated with $\eta_t$. Then, there exists a constant $c = c (b_2, \varpi) > 0$ such that 
	\begin{flalign*}
	\mathbb{P} \Bigg[ \bigg| \displaystyle\sum_{i = 1}^M \big( p_{\lfloor \varpi N \rfloor} (i) - p_0 (i) \big) - \varpi \rho^{-1} \varphi (\rho) M N  \bigg| > \displaystyle\frac{60 \varpi^2 N^2}{(1 - b_2)^2} \Bigg] \le c^{-1} e^{- c N}.
	\end{flalign*}

\end{lem}

\begin{proof}
	
	Set $R = \big\lfloor \frac{N}{1 - b_2} \big\rfloor$, and set $p_t (i) = -\infty$ if $i \le 0$ and $p_t (i) = \infty$ if $i > M$. Then observe by \eqref{etasumlimitx1x2} that $\big| \rho^{-1} \varphi (\rho) M - 16 \varphi (\rho) R \big| \le \varpi R$ for sufficiently large $N$. Thus, by a union bound, it suffices to show the existence of a constant $c = c (b_2, \varpi) > 0$ such that 
	\begin{flalign}
	\label{p0piestimate}
	\begin{aligned} 
	& \mathbb{P} \Bigg[ \displaystyle\sum_{i = 1}^M \big( p_{\lfloor \varpi N \rfloor} (i) - p_0 (i) \big) > 16  \varphi (\rho) \varpi R N  + 59 \varpi^2 R^2 \Bigg] \le c^{-1} e^{- c N}; \\
	& \mathbb{P} \Bigg[ \displaystyle\sum_{i = 1}^M \big( p_{\lfloor \varpi N \rfloor} (i) - p_0 (i) \big) < 16  \varphi (\rho) \varpi R N  - 59 \varpi^2 R^2 \Bigg] \le c^{-1} e^{- c N},
	\end{aligned} 
	\end{flalign}
	
	\noindent which we will do by comparing $\textbf{p}$ to a stationary stochastic six-vertex model $\textbf{q}$.
	
	We only establish the first estimate in \eqref{p0piestimate}, since the proof of the latter is entirely analogous. To that end, let $\textbf{q}_0$ denote a particle position sequence sampled according to the measure induced by $\Upsilon^{(\rho)}$, that is, $\textbf{q}_0$ denotes the particle position sequence for a $\rho$-stationary stochastic six-vertex model. We index the particles of $\textbf{q}_0$ so that $q_0 (0) < p_0 (1) + \varpi^2 R < q_0 (1)$. Further define the particle position sequence $\textbf{r}_0$ by setting $r_0 (k) = q_0 (k)$ if $k \le M$ and $r_0 (k) = \infty$ for $k > M$. 
	
	Let $G_1$ denote the event on which $p_0 (k) < r_0 (k) < p_0 (k) + 2 \varpi^2 R$, for each $k \in [1, M]$. Then \eqref{etasumlimitx1x2} and a large deviations estimate for sums of independent $0-1$ Bernoulli random variables imply that $\mathbb{P} [G_1^c] \le c^{-1} e^{-c N}$, for some constant $c = c(b_2, \varpi) > 0$. 
	
	Now let $\textbf{q}_t$ and $\textbf{r}_t$ denote stochastic six-vertex models run under initial data $\textbf{q}_0$ and $\textbf{r}_0$, respectively, which are coupled under the higher rank coupling of \Cref{couplinghigherrank}. It follows from \Cref{lambdaximonotone} that
	\begin{flalign}
	\label{probabilityf1pisum} 
	\begin{aligned}
	\mathbb{P} \Bigg[ & \displaystyle\sum_{i = 1}^M \big( p_{\lfloor \varpi N \rfloor} (i) - p_0 (i) \big) > 16  \varphi (\rho) \varpi R N  + 59 \varpi^2 R^2 \Bigg] \\
	& \le \mathbb{P} \Bigg[ \textbf{1}_{G_1} \displaystyle\sum_{i = 1}^M \big( p_{\lfloor \varpi N \rfloor} (i) - p_0 (i) \big) > 16  \varphi (\rho) \varpi R N  + 59 \varpi^2 R^2 \Bigg] + \mathbb{P} [G_1^c] \\
	& \le \mathbb{P} \Bigg[ \textbf{1}_{G_1} \displaystyle\sum_{i = 1}^M \big( r_{\lfloor \varpi N \rfloor} (i) - p_0 (i) \big) > 16  \varphi (\rho) \varpi R N  + 59 \varpi^2 R^2 \Bigg] + c^{-1} e^{-cN} \\
	& \le \mathbb{P} \Bigg[ \displaystyle\sum_{i = 1}^M \big( r_{\lfloor \varpi N \rfloor} (i) - r_0 (i) \big) > 16  \varphi (\rho) \varpi R N  + 27 \varpi^2 R^2 \Bigg] + c^{-1} e^{-cN}.
	\end{aligned} 
	\end{flalign}

	Next, let $G_2$ denote the event on which there exists an integer $t \in [0, \varpi N]$ for which $\textbf{q}_t \cap \big[ (3 \varpi - 8) R, (8 - 3 \varpi) R \big] \ne \textbf{r}_t \cap \big[ (3 \varpi - 8) R, (8 - 3 \varpi) R \big]$. Then \Cref{modelsequal} yields $\mathbb{P} [G_2] \le c^{-1} e^{-cN}$, after decreasing $c$ if necessary. Additionally, let $G_3$ denote the event that there exists an integer $i \in [1, M]$ for which $q_{\lfloor \varpi N \rfloor} (i) - q_0 (i) \ge 2 \varpi R$. Then \Cref{xtxt1} and a large deviation estimate for sums of independent geometric random variables together imply that $\mathbb{P} [G_3] \le c^{-1} e^{-cN}$, after further decreasing $c$ if necessary. It follows that 
	\begin{flalign}
	\label{probabilityf2f3rsum} 
	\begin{aligned}
	\mathbb{P} \Bigg[ & \displaystyle\sum_{i = 1}^M \big( r_{\lfloor \varpi N \rfloor} (i) - r_0 (i) \big) > 16  \varphi (\rho) \varpi R N  + 27 \varpi^2 R^2 \Bigg] \\
	& \le \mathbb{P} \Bigg[ \textbf{1}_{G_2^c} \textbf{1}_{G_3^c} \displaystyle\sum_{i = 1}^M \big( r_{\lfloor \varpi N \rfloor} (i) - r_0 (i) \big) > 16  \varphi (\rho) \varpi R N  + 27 \varpi^2 R^2 \Bigg] + \mathbb{P} [G_2] + \mathbb{P} [G_3] \\
	& \le \mathbb{P} \Bigg[ \textbf{1}_{G_3^c} \displaystyle\sum_{i = 1}^M \big( q_{\lfloor \varpi N \rfloor} (i) - q_0 (i) \big) > 16  \varphi (\rho) \varpi R N  + 3 \varpi^2 R^2 \Bigg] + 2 c^{-1} e^{-c N}.
	\end{aligned} 
	\end{flalign}
	
	Now let $\mathcal{E}$ denote the six-vertex ensemble on $\mathfrak{H}$ corresponding to $\textbf{q}_t$, and recall from \Cref{Translation} the horizontal indicators $\chi^{(h)} (x, y)$ for this ensemble denoting the number of arrows along the edge connecting $\big( (x, y), (x + 1, y) \big)$. Upon restricting to the event $G_3^c$, we have that $q_t (i) \in \big[ -8R, (8 + 2 \varpi) R \big]$ for each $i \in [1, M]$. Hence,
	\begin{flalign}
	\label{sumqf3}
	\textbf{1}_{G_3^c} \displaystyle\sum_{i = 1}^M \big( q_{\lfloor \varpi N \rfloor} (i) - q_0 (i) \big) \le  \displaystyle\sum_{x = - \lfloor 8 R \rfloor}^{\lfloor 8R + 2 \varpi R \rfloor} \displaystyle\sum_{y = 1}^{\lfloor \varpi N \rfloor} \chi^{(h)} (x, y). 
	\end{flalign}
	
	Next observe that the law of $\mathcal{E}$ is given by the restriction of the measure $\mu (\rho)$ (from \Cref{Translation}) to $\mathfrak{H}$. The translation-invariance of $\mu (\rho)$ implies for any fixed $x \in \mathbb{Z}$ that the $\big\{ \chi^{(h)} (x, y) \big\}_{y \in [0, \varpi N]}$ are mutually independent $0-1$ Bernoulli random variables with means $\varphi (\rho)$. Thus, a large deviations estimate for sums of $0-1$ Bernoulli random variables yields (after decreasing $c$ if necessary) that 
	\begin{flalign}
	\label{sumhorizontal}
	\mathbb{P} \left[  \displaystyle\sum_{x = - \lfloor 8 R \rfloor}^{\lfloor 8R + 2 \varpi R \rfloor} \displaystyle\sum_{y = 1}^{\lfloor \varpi N \rfloor} \chi^{(h)} (x, y) > \varphi (\rho) (16 R + 2 \varpi R) \varpi N + \varpi^2 RN  \right] \le c^{-1} e^{-cN}. 
	\end{flalign}
	
	The first bound in \eqref{p0piestimate} then follows from \eqref{probabilityf1pisum}, \eqref{probabilityf2f3rsum}, \eqref{sumqf3}, \eqref{sumhorizontal}, the fact that $R \ge N$, and the fact that $\varphi (\rho) \in [0, 1]$. As mentioned previously, the proof of the second bound is entirely analogous and therefore omitted. 
\end{proof}

\subsection{Local Approximation} 

\label{Current2} 

In this section we provide a local approximation of the total current. 

To that end, set $A = \lceil N^{1 / 8} \rceil > 0$, and define $\rho_j = \frac{j}{A}$ for each $j \in [0, A]$. Further let $\xi_0^{(\rho_j)} \in \{ 0, 1 \}^{\mathbb{Z}}$ denote a (random) particle configuration on $\mathbb{Z}$, sampled according to the measure $\Upsilon^{(\rho_j)}$, for each $j \in [0, A]$. Assume that these particle configurations are mutually coupled so that $\xi_0^{(\rho_i)} \le \xi_0^{(\rho_j)}$ whenever $0 \le i \le j \le A$. For each $j \in [0, A]$, let $\big( \xi_t^{(\rho_j)} \big)$ denote the stochastic six-vertex model with initial data $\xi_0^{(\rho_j)}$; mutually couple these models with $\eta_t$ under the higher rank coupling of \Cref{couplinghigherrank} (recall \Cref{etaxim}).

Now set $k = \lceil (\log N)^{1 / 4} \rceil > 0$; fix $\varpi \in (0, 1)$; and let $\mathcal{I}$ denote a set of pairwise disjoint intervals such that $\bigcup_{I \in \mathcal{I}} I = \mathbb{Z}$ and $|I| = k$ for each $I \in \mathcal{I}$. Recalling the function $\mathcal{R} (I; \eta, \xi)$ from \eqref{rietaxi}, \Cref{guvpqsum} yields a constant $C = C(b_1, b_2, \varpi) > 0$ such that 
\begin{flalign}
\label{sumretarho} 
\mathbb{E} \left[ \displaystyle\frac{1}{\lfloor \varpi N \rfloor} \displaystyle\sum_{t = 0}^{\lfloor \varpi N \rfloor} \displaystyle\sum_{j = 0}^A \displaystyle\sum_{I \in \mathcal{I}}  \mathcal{R} \big( I; \eta_t, \xi_t^{(\rho_j)} \big) \right]  < C N^{1 / 3}. 
\end{flalign}

The following definition provides certain events and ``local densities'' that will be of use to us. 

\begin{definition} 
	
	\label{eventfti}
	
	For each interval $I \in \mathcal{I}$ and integer $t \in [0, \varpi N]$, define the events 
	\begin{flalign*}
	F (I; t) = F_N (I; t) = \left\{ \displaystyle\sum_{j = 0}^A \mathcal{R} \big( I; \eta_t, \xi_t^{(\rho_j)} \big) > 0 \right\}; \quad F = F_N = \left\{ \displaystyle\frac{1}{\lfloor \varpi N \rfloor} \displaystyle\sum_{t = 1}^{\lfloor \varpi N \rfloor} \displaystyle\sum_{I \in \mathcal{I}} \textbf{1}_{F(I; t)} > N^{2 / 3} \right\}. 
	\end{flalign*}
	
	Further fix some $I$ and $t$, and assume that $F (I; t)^c$ holds. Then, there exists some $j = j (I; t) \in [0, A - 1]$ for which $\xi_t^{(\rho_j)} \big|_I \le \eta_t |_I \le \xi_t^{(\rho_{j + 1})} \big|_I$. We define $\rho (I; t) = \rho_j = \rho_{j (I; t)}$. 
	
\end{definition} 

In particular, $F (I; t)^c$ denotes the event on which $\eta_t$ and all of the $\xi_t^{(\rho_j)}$ are ordered when restricted to $I$; $F^c$ denotes the event on which most of the $F (I; t)^c$ hold; and $\rho (I; t)$ denotes the approximate density in $\eta$ of an interval $I$ at time $t$. Observe by \eqref{sumretarho} and a Markov estimate that $\mathbb{P} [F] \le C N^{-1 / 3}$. 

The following lemma now provides a ``local'' approximation for the total current.

\begin{lem}
	
	\label{horizontalestimatesum}
	
	Adopting the notation of \Cref{heightcompare}, there exists a constant $C = C (b_1, b_2, \varpi) > 0$ such that 
	\begin{flalign*}
	\mathbb{P} \left[ \Bigg| \displaystyle\sum_{i = 1}^M \big( p_{\lfloor \varpi N \rfloor} (i) - p_0 (i) \big) - k \displaystyle\sum_{t = 0}^{\lfloor \varpi N \rfloor - 1} \displaystyle\sum_{I \in \mathcal{I}} \varphi \big( \rho (I; t) \big) \textbf{\emph{1}}_{F (I; t)^c} \Bigg| > \displaystyle\frac{C N^2}{(\log N)^{1 / 256}} \right] \le \displaystyle\frac{C}{\log N}. 
	\end{flalign*}

\end{lem}

\begin{proof}
	
	Throughout this proof, we again set $R = \big\lfloor \frac{N}{1 - b_2} \big\rfloor$. Let $\mathcal{E}$ denote the six-vertex ensemble on $\mathfrak{H}$ associated with $\textbf{p}_t$, and recall from \Cref{Translation} the horizontal indicators $\chi^{(h)} (x, y)$ for this ensemble denoting the number of arrows connecting $\big( (x, y), (x + 1, y) \big)$. Then, we have that 
	\begin{flalign} 
	\label{p0pisum} 
	\displaystyle\sum_{i = 1}^M \big( p_{\lfloor \varpi N \rfloor} (i) - p_0 (i) \big) = \displaystyle\sum_{x \in \mathbb{Z}} \displaystyle\sum_{t = 1}^{\lfloor \varpi N \rfloor} \chi^{(h)} (x, t). 
	\end{flalign} 
	
	We will now essentially approximate $\sum_{x \in I} \chi^{(h)} (x, t)$ by the associated stationary quantity, which will be $\varphi \big( \rho (I; t) \big) k$, whenever $F (I; t)^c$ holds. To that end, it will first be useful to bound the location of the rightmost particle in $\eta_s$. In particular, \Cref{xtxt1} and a large deviations estimate for sums of independent geometric random variables yields a constant $c = c(b_2) > 0$ satisfying  
	\begin{flalign}
	\label{g0probability} 
	\mathbb{P} [G_0] < c^{-1} e^{-cN}, \quad \text{where} \quad G_0 = \big\{ p_{\lfloor \varpi N \rfloor} (M) \ge 16R \big\}.
	\end{flalign}

	\noindent Next, let us fix some integer $t \in [1, \varpi N]$ and interval $I = [u, v] \in \mathcal{I}$ with $v = u + k - 1$. Abbreviate $j = j (I; t - 1)$ and observe by a union bound that, for sufficiently large $N$,  
	\begin{flalign}
	\label{g1probability}
	\mathbb{P} \big[ G_1 (I; t - 1) \big] \le k (\rho_{j + 1} - \rho_j) \le \frac{1}{N^{1 / 16}}
	\end{flalign}
	
	\noindent where
	\begin{flalign*} 
	G_1 (I; t - 1) = \Big\{ \xi_{t - 1}^{(\rho_j)} \big|_I \ne \xi_{t - 1}^{(\rho_{j + 1})} \big|_I \Big\} \cap F (I; t - 1)^c. 
	\end{flalign*}
	
	Now, let $\mathcal{F} = \mathcal{F}^{(\rho_j)}$ denote the six-vertex ensemble on $\mathfrak{H}$ corresponding to $\big( \xi_s^{(\rho_j)} \big)$, and further let $m = \big\lfloor (\log N)^{1 / 8} \big\rfloor$ and $n = \big\lfloor (\log N)^{1 / 16} \big\rfloor$. On $G_1 (I; t - 1)^c \cap F (I; t - 1)^c$, we have that $\xi_{t - 1}^{(\rho_j)} \big|_I = \eta_t |_I = \xi_{t - 1}^{(\rho_{j + 1})} \big|_I$, so \Cref{modelsequal} indicates that $\mathcal{E}$ and $\mathcal{F}^{(\rho_j)}$ will likely coincide around $\big( \frac{u + v}{2}, t \big)$. More precisely, define the interval $J = J(I) = [u + m, v - m] \subset I$ and the event
	\begin{flalign*} 
	G_2 (I; t - 1) = \Big\{ \mathcal{E} |_{J \times [t - 1, t + n]} \ne \mathcal{F}^{(\rho_j)} \big|_{J \times [t - 1, t + n]} \Big\} \cap F(I; t - 1)^c.
	\end{flalign*}
	
	\noindent Then, for sufficiently large $N$, \Cref{modelsequal} yields 
	\begin{flalign}
	\label{g2probability} 
	\mathbb{P} \big[ G_2 (I; t - 1) \cap G_1 (I; t - 1)^c \big] \le \displaystyle\frac{1}{(\log N)^5}.
	\end{flalign}

	Next, recall from \Cref{Translation} that, for each fixed $x \in J$, the $\big\{ \chi^{(h)} (x, y) \big\}$ are (with respect to $\mathcal{F}^{(\rho_j)}$) mutually independent Bernoulli $0-1$ random variables with means $\varphi (\rho_j)$. Thus, we deduce from a moderate deviations event for sums of such random variables that, for sufficiently large $N$,
	\begin{flalign}
	\label{g3probability} 
	\mathbb{P} \big[ G_3 (I; t - 1) \big] \le \displaystyle\frac{1}{(\log N)^5},
	\end{flalign}
	
	\noindent where 
	\begin{flalign*}
	G_3 (I; t - 1) = \left\{ \displaystyle\max_{x \in J} \bigg| \displaystyle\sum_{y = t}^{t + n - 1} \chi_h (x, y) - n \varphi \big( \rho (I; t) \big) \bigg| > n^{3 / 4} \right\} \cap F(I; t - 1)^c.
	\end{flalign*}

	\noindent Additionally define the event 
	\begin{flalign*} 
	G_4 (I; t - 1) = \bigcup_{s = t}^{t + n - 1} \big( F(I; s - 1) \cup G_1 (I; s - 1) \cup G_2 (I; s - 1) \cup G_3 (I; s - 1) \big).
	\end{flalign*}
	
	\noindent Then \eqref{g1probability}, \eqref{g2probability}, \eqref{g3probability}, and union bound imply that
	\begin{flalign*} 
	\mathbb{P} \left[ G_4 (I; t - 1) \cap \bigcap_{s = t}^{t + n - 1} F(I; s - 1)^c \right] \le \displaystyle\frac{1}{(\log N)^4},
	\end{flalign*}
	
	\noindent for sufficiently large $N$. Therefore, defining the event 
	\begin{flalign*}
	G = F \cup G_0 \cup \left\{ \displaystyle\sum_{t = 1}^{\lfloor \varpi N \rfloor} \displaystyle\sum_{I \in \mathcal{I}} \textbf{1}_{G_4 (I; t - 1)} > \displaystyle\frac{N^2}{(\log N)^2} \right\},
	\end{flalign*}
	
	\noindent we deduce from the fact that $\mathbb{P} [F] \le C N^{-1 / 3}$ for some constant $C = C (b_1, b_2, \varpi) > 0$ (recall \Cref{eventfti}), \eqref{g0probability}, and a Markov estimate that $\mathbb{P} [G] \le (\log N)^{-1}$ for sufficiently large $N$.
	
	Now, let us approximate the right side of \eqref{p0pisum} restricting to the event $G^c$. To that end, first observe that
	\begin{flalign}
	\label{chihsum1}
	\textbf{1}_{G^c} \left| \displaystyle\sum_{t = 1}^{\lfloor \varpi N \rfloor} \displaystyle\sum_{x \in \mathbb{Z}} \chi^{(h)} (x, t) -  \displaystyle\sum_{t = 1}^{\lfloor \varpi N \rfloor} \displaystyle\sum_{I \in \mathcal{I}} \textbf{1}_{G_4 (I; t - 1)^c} \displaystyle\sum_{x \in I} \chi^{(h)} (x, t)  \right| < \displaystyle\frac{N^2 k}{(\log N)^2} \le \displaystyle\frac{N^2}{\log N}.
	\end{flalign}
	
	\noindent Next, since $\chi (x, t) \textbf{1}_{G^c} = 0$ whenever $|x| \ge 16R$ (from the event $G_0$), we have that 
	\begin{flalign}
	\label{chihsum2}
	\textbf{1}_{G^c} \left| \displaystyle\sum_{t = 1}^{\lfloor \varpi N \rfloor} \displaystyle\sum_{I \in \mathcal{I}} \textbf{1}_{G_4 (I; t - 1)^c} \displaystyle\sum_{x \in I} \chi^{(h)} (x, t) - \displaystyle\sum_{t = 1}^{\lfloor \varpi N \rfloor} \displaystyle\sum_{I \in \mathcal{I}} \textbf{1}_{G_4 (I; t - 1)^c} \displaystyle\sum_{x \in J (I)} \chi^{(h)} (x, t)  \right| \le 64 \varpi m k^{-1} RN,
	\end{flalign}
	
	\noindent Additionally, for the same reason and from the event $G$, we have that 
	\begin{flalign}
	\label{chihsum3}
	\begin{aligned}
	\textbf{1}_{G^c} \Bigg| \displaystyle\sum_{t = 1}^{\lfloor \varpi N \rfloor} & \displaystyle\sum_{I \in \mathcal{I}} \textbf{1}_{G_4 (I; t - 1)^c} \displaystyle\sum_{x \in J(I)} \chi^{(h)} (x, t)  \\
	& \qquad - \displaystyle\frac{1}{n} \displaystyle\sum_{t = 1}^{\lfloor \varpi N \rfloor} \displaystyle\sum_{I \in \mathcal{I}} \textbf{1}_{G_4 (I; t - 1)^c} \displaystyle\sum_{x \in J(I)} \displaystyle\sum_{s = t}^{t + n - 1} \chi^{(h)} (x, s)  \Bigg| < \displaystyle\frac{k N^2}{(\log N)^2} + 64 n R.
	\end{aligned}
	\end{flalign}
	
	\noindent From the event $G_3 (I; t - 1)^c$, we have for sufficiently large $N$ that 
	\begin{flalign}
	\label{chihsum4}
	\begin{aligned} 
	\textbf{1}_{G^c} \Bigg| \displaystyle\frac{1}{n} \displaystyle\sum_{t = 1}^{\lfloor \varpi N \rfloor} & \displaystyle\sum_{I \in \mathcal{I}} \textbf{1}_{G_4 (I; t - 1)^c} \displaystyle\sum_{x \in J(I)} \displaystyle\sum_{s = t}^{t + n - 1} \chi^{(h)} (x, s) \\
	& - \displaystyle\sum_{t = 1}^{\lfloor \varpi N \rfloor} \displaystyle\sum_{I \in \mathcal{I}} \varphi \big( \rho (I; t) \big) (k - 2m) \textbf{1}_{G_4 (I; t - 1)^c} \Bigg| < n^{-1 / 8} N^2.
	\end{aligned}
	\end{flalign}
	
	\noindent Moreover, since $\textbf{1}_{G_0^c} \rho (I; t) = 0$ for $|x| \ge 16R$, we have
	\begin{flalign}
	\label{chihsum5}
	\begin{aligned}
	& \textbf{1}_{G^c} \left| \displaystyle\sum_{t = 1}^{\lfloor \varpi N \rfloor} \displaystyle\sum_{I \in \mathcal{I}} \varphi \big( \rho (I; t) \big) (k - 2m) \textbf{1}_{G_4 (I; t - 1)^c} - k \displaystyle\sum_{t = 1}^{\lfloor \varpi N \rfloor} \displaystyle\sum_{I \in \mathcal{I}} \varphi \big( \rho (I; t) \big) \textbf{1}_{G_4 (I; t - 1)^c}  \right| < 64 \varpi m k^{-1} NR; \\
	& \textbf{1}_{G^c}\left| k \displaystyle\sum_{t = 1}^{\lfloor \varpi N \rfloor} \displaystyle\sum_{I \in \mathcal{I}} \varphi \big( \rho (I; t) \big) \textbf{1}_{G_4 (I; t - 1)^c} - k \displaystyle\sum_{t = 1}^{\lfloor \varpi N \rfloor} \displaystyle\sum_{I \in \mathcal{I}} \varphi \big( \rho (I; t) \big) \textbf{1}_{F (I; t - 1)^c}  \right| < \displaystyle\frac{2 k N^2}{(\log N)^2}.
	\end{aligned}  
	\end{flalign}
	
	Now \eqref{chihsum1}, \eqref{chihsum2}, \eqref{chihsum3}, \eqref{chihsum4}, \eqref{chihsum5}, and the fact that $\mathbb{P} [G] \le (\log N)^{-1}$ together imply the lemma. 	
\end{proof}

\subsection{Proof of the Approximate Coupling} 

\label{ProofStochastic} 

In this section we establish \Cref{rhotksum2}; throughout, we adopt the notation of that proposition and of \Cref{Current2}. For any $\varpi > 0$, let us define the subset of intervals $\mathcal{J} \subset \mathcal{I}$ by 
\begin{flalign*}
\mathcal{J} = \big\{ I: I \subset \big[ -(8 + 3 \varpi) R, (8 + 3 \varpi) R \big]; I \in \mathcal{I} \big\}.
\end{flalign*}

From \Cref{heightcompare}, \Cref{horizontalestimatesum}, and the fact that $\varphi'' (0) < 0$, we will deduce the following corollary, which essentially states that $\rho (I; t)$ is typically equal to $\rho$. 

\begin{cor} 
	
	\label{rhotksum}
	
	Under the notation of \Cref{horizontalestimatesum}, there exist constants $C_1 = C_1 (b_1, b_2) > 0$ and $C_2 = C_2 (b_1, b_2, \varpi) > 0$ such that 
	\begin{flalign*} 
	\mathbb{P} \left[ \displaystyle\sum_{t = 1}^{\lfloor \varpi N \rfloor - 1} \displaystyle\sum_{I \in \mathcal{J}}   \big| \rho (I; t) \textbf{\emph{1}}_{F(I; t)^c} - \rho \big|^2 > C_1 k^{-1} \varpi^2 N^2 \right] \le \displaystyle\frac{C_2}{\log N}. 
	\end{flalign*}
	
\end{cor} 

\begin{proof}

	We first claim that $k \sum_{I \in \mathcal{J}} \rho (I; t) \textbf{1}_{F (I; t)^c} \approx \sum_{x \in \mathbb{Z}} \eta_t (x) = M \approx 16 R \rho$, for most $t \in [0, \varpi N]$, with high probability. To that end, it will first be useful to bound the position of the rightmost particle in $\textbf{p}$. Thus, observe by \Cref{xtxt1} and a large deviations estimate for sums of independent geometric random variables that there exists a constant $C_3 = C_3 (b_2, \varpi) > 0$ such that 
	\begin{flalign}
	\label{g0event} 
	\mathbb{P} [G_0] \le C_3 e^{-N / C_3}, \quad \text{where} \quad G_0 = \big\{ p_{\lfloor \varpi N \rfloor} (M) > (8 + 3 \varpi) R \big\}.
	\end{flalign}
	
	Next, we will show that $\sum_{x \in I} \eta_t (x) \approx \rho (I; t) k$ with high probability whenever $F (I; t)^c$ holds. Indeed, a moderate deviations estimate for sums of independent Bernoulli $0-1$ random variables yields 
	\begin{flalign}
	\label{xisumrhoi} 
	\mathbb{P} \Bigg[ \bigg| \displaystyle\sum_{x \in I} \xi_t^{(\rho_i)} (x) - k \rho_i \bigg| > k^{3 / 4} \Bigg] < \displaystyle\frac{C_4}{(\log N)^4},
	\end{flalign}
	
	\noindent for some constant $C_4 > 0$ and any $i \in [0, A]$. 
	
	Now, fixing $I \in \mathcal{I}$ and $t \in [0, \varpi N]$ and restricting to the event $F(I; t)^c$ and abbreviating $j = j(I; t)$ yields $\xi_t^{(\rho_j)} \big|_I \le \eta_t |_I \le \xi_t^{(\rho_{j + 1})} |_I$. Therefore, by \eqref{xisumrhoi} and a union bound we obtain 
	\begin{flalign*}
	\mathbb{P} \big[ G_1 (I; t) \big] < \displaystyle\frac{2 C_4}{(\log N)^4}, \quad \text{where} \quad G_1 (I; t) = \Bigg\{ \textbf{1}_{F (I; t)^c} \bigg| \displaystyle\sum_{x \in I} \eta_t (x) - k \rho (I; t)  \bigg| > 2 k^{3 / 4} + \displaystyle\frac{2k}{A} \Bigg\}.
	\end{flalign*}
	
	\noindent Thus, for each $t \in [0, \varpi N]$, a Markov estimate yields a constant $C_5 = C_5 (b_2) > 0$ such that 
	\begin{flalign}
	\label{g1event2}
	\mathbb{P} \big[ G_1 (t) \big] < \displaystyle\frac{C_5}{(\log N)^3}, \quad \text{where} \quad G_1 (t)  = \Bigg\{ \displaystyle\sum_{I \in \mathcal{J}} \textbf{1}_{G_1 (I; t)} > \displaystyle\frac{N}{k \log N} \Bigg\}.
	\end{flalign}
	
	\noindent Next, to establish that $k \sum_{I \in \mathcal{J}} \rho (I; t) \textbf{1}_{F (I; t)^c} \approx \sum_{x \in \mathbb{Z}} \eta_t (x) = M \approx 16 R \rho$, for most $t \in [0, \varpi N]$ (as stated above), observe for sufficiently large $N$ that $G_2 (t) \subseteq G_0 \cup G_1 (t) \cup F(t)$, where 
	\begin{flalign}
	\label{g2tevent} 
	G_2 (t) = \Bigg\{  \bigg| k \displaystyle\sum_{I \in \mathcal{J}} \rho (I; t) \textbf{1}_{F (I; t)^c}  - 16 R \rho  \bigg| > 2 \varpi N \Bigg\}; \qquad F(t) = \left\{ \displaystyle\sum_{I \in \mathcal{J}} \textbf{1}_{F (I; t)} > N^{5 / 6} \right\}.
	\end{flalign}
	
	\noindent Indeed, this holds since \eqref{etasumlimitx1x2} and $\eta_t (x) \in \{ 0, 1 \}$ imply that $| 16 R \rho - M | < \varpi N$ and  
	\begin{flalign*}
	\textbf{1}_{F(t)^c} \textbf{1}_{G_0^c} \textbf{1}_{G_1 (t)^c} \left| k \displaystyle\sum_{I \in \mathcal{J}} \rho (I; t) \textbf{1}_{F (I; t)^c} - M  \right| & \le \textbf{1}_{F(t)^c} \textbf{1}_{G_1 (t)^c} \left| k \displaystyle\sum_{I \in \mathcal{J}} \rho (I; t) \textbf{1}_{F (I; t)^c} - \displaystyle\sum_{I \in \mathcal{J}} \displaystyle\sum_{x \in I} \eta_t (x)  \right| \\
	& \le  \displaystyle\sum_{I \in \mathcal{J}} \left| \displaystyle\sum_{x \in I} \eta_t (x)  - k \rho (I; t) \right| \textbf{1}_{F (I; t)^c} \textbf{1}_{G_1 (I; t)} + \displaystyle\frac{3 N}{\log N}   \\
	& \le |\mathcal{J}| \left( 2 k^{3 / 4} + \displaystyle\frac{2k}{A} \right) + \displaystyle\frac{3N}{\log N} < \varpi N,
	\end{flalign*}
	
	\noindent respectively, both for sufficiently large $N$.
	
	We next bound the number of $t \in [0, \varpi N]$ for which $G_2 (t)$ holds. To that end, define the events
	\begin{flalign*} 
	G_1 = \left\{ \displaystyle\sum_{t = 0}^{\lfloor \varpi N \rfloor} \textbf{1}_{G_1 (t)} > \displaystyle\frac{N}{\log N}\right\}; \qquad G_2 = \left\{ \displaystyle\sum_{t = 0}^{\lfloor \varpi N \rfloor} \textbf{1}_{G_2 (t)} > \displaystyle\frac{2 N}{\log N}\right\}.
	\end{flalign*}
	
	\noindent Then, a Markov estimate implies for sufficiently large $N$ that 
	\begin{flalign} 
	\label{g1probabilityg2}
	\mathbb{P} [G_1] < \displaystyle\frac{C_5}{(\log N)^2}; \qquad \mathbb{P} [G_2] \le \mathbb{P} [F] + \mathbb{P} [G_0] + \mathbb{P} [G_1] \le \displaystyle\frac{C_6}{(\log N)^2},
	\end{flalign} 
	
	\noindent for some constant $C_6 = C_6 (b_1, b_2, \varpi) > 0$, where we have applied \eqref{g0event}, \eqref{g1event2}, and the fact that $\mathbb{P} [F] \le C_7 N^{-1 / 3}$ for some constant $C_7 = C_7 (b_1, b_2, \varpi) > 0$ (recall \Cref{eventfti} and below). 
	
	Now define $\rho (t) = |\mathcal{J}|^{-1} \sum_{I \in \mathcal{J}} \rho (I; t) \textbf{1}_{F(I; t)^c}$ for each $t \in [0, \varpi N]$, so that $\textbf{1}_{G_2 (t)} \big| \rho (t) - \rho \big| \le 3 \varpi$, since $\big| \frac{16 R}{k |\mathcal{J}|} - 1 \big| < \frac{\varpi}{2}$. Since $\varphi (0) = 0$, we have that
	\begin{flalign*}
	\textbf{1}_{G_0^c} \displaystyle\sum_{I \in \mathcal{I}} \varphi \big( \rho (I; t) \big) \textbf{1}_{F (I; t)^c} & = \textbf{1}_{G_0^c} \displaystyle\sum_{I \in \mathcal{J}} \varphi \big( \rho (I; t) \textbf{1}_{F (I; t)^c} \big) \\
	& = \textbf{1}_{G_0^c} \left( \varphi \big( \rho (t) \big) |\mathcal{J}| + \displaystyle\sum_{I \in \mathcal{J}} \displaystyle\int_{\rho (t)}^{\rho (I; t) \textbf{1}_{F (I; t)^c}} \displaystyle\int_0^z \varphi'' (s) ds dz \right). 
	\end{flalign*}
	
	\noindent Thus, since $\max_{z \in [0, 1]} \varphi'' (z) < - 2 (b_2 - b_1)^2 < 0$, it follows that
	\begin{flalign}
	\label{fg0g1}
	\begin{aligned} 
	\textbf{1}_{G_2^c} (b_2 - b_1)^2 & \displaystyle\sum_{t = 0}^{\lfloor \varpi N \rfloor - 1} \displaystyle\sum_{I \in \mathcal{J}} \big| \rho (t) - \rho (I; t) \textbf{1}_{F(I; t)^c} \big|^2 \\ 
	& \le \textbf{1}_{G_2^c} \displaystyle\sum_{t = 0}^{\lfloor \varpi N \rfloor - 1} \left( \varphi \big( \rho (t) \big) |\mathcal{J}| - \displaystyle\sum_{I \in \mathcal{I}} \varphi \big( \rho (I; t) \big) \textbf{1}_{F (I; t)^c}  \right),
	\end{aligned}
	\end{flalign}
	
	\noindent where we have used that $G_0 \subseteq G_2$. Since $\textbf{1}_{G_2 (t)} \big| \rho (t) - \rho \big| \le 3 \varpi$, we have for sufficiently large $N$ that 
	\begin{flalign}
	\label{fg0g12}
	\begin{aligned}
	\textbf{1}_{G_2^c} \displaystyle\sum_{t = 0}^{\lfloor \varpi N \rfloor - 1} \displaystyle\sum_{I \in \mathcal{J}} \big| \rho (t) - \rho (I; t) \textbf{1}_{F(I; t)^c} \big|^2 & \ge \textbf{1}_{F^c} \textbf{1}_{G_0^c} \textbf{1}_{G_2^c} \displaystyle\sum_{t = 0}^{\lfloor \varpi N \rfloor - 1} \displaystyle\sum_{I \in \mathcal{J}} \big| \rho - \rho (I; t) \textbf{1}_{F(I; t)^c} \big|^2 \\
	& \qquad - C_8 \varpi^2 k^{-1} N^2,
	\end{aligned}
	\end{flalign}
	
	\noindent for some constant $C_8 = C_8 (b_2) > 0$. Again using the fact that $\textbf{1}_{G_2 (t)} \big| \rho (t) - \rho \big| \le 3 \varpi$ yields for sufficiently large $N$ that 
	\begin{flalign}
	\label{fg0g13}
	\begin{aligned}
	\textbf{1}_{G_2^c} \displaystyle\sum_{t = 0}^{\lfloor \varpi N \rfloor - 1} & \left( \varphi \big( \rho (t) \big) |\mathcal{J}| - \displaystyle\sum_{I \in \mathcal{I}} \varphi \big( \rho (I; t) \big) \textbf{1}_{F (I; t)^c}  \right) \\
	& \le \textbf{1}_{G_2^c} \displaystyle\sum_{t = 0}^{\lfloor \varpi N \rfloor - 1} \left( \varphi (\rho) |\mathcal{J}| - \displaystyle\sum_{I \in \mathcal{I}} \varphi \big( \rho (I; t) \big) \textbf{1}_{F (I; t)^c}  \right) + 3 \kappa \varpi^2 |\mathcal{J}| N \\
	& \le \textbf{1}_{G_2^c} \left( \varpi \rho^{-1} k^{-1} \varphi (\rho) MN - \displaystyle\sum_{t = 0}^{\lfloor \varpi N \rfloor - 1} \displaystyle\sum_{I \in \mathcal{I}} \varphi \big( \rho (I; t) \big) \textbf{1}_{F (I; t)^c}  \right) + C_9 \varpi^2 k^{-1} N^2,
	\end{aligned}
	\end{flalign}
	
	\noindent for some constant $C_9 = C_9 (b_2) > 0$, where we have used \eqref{etasumlimitx1x2} (to approximate $\big| \varphi (\rho) |\mathcal{J}| - \rho^{-1} k^{-1} \varphi (\rho) M \big| \le \varpi k^{-1} N$, for sufficiently large $N$) and the fact that $\max_{z \in [0, 1]} \varphi' (z) \le \kappa$. Since \Cref{heightcompare}, \Cref{horizontalestimatesum}, and a union bound together yield constants $C_{10} = C_{10} (b_1, b_2) > 0$ and $C_{11} = C_{11} (b_1, b_2, \varpi) > 0$ such that 
	\begin{flalign*} 
	\mathbb{P} \Bigg[ \bigg| \varpi \rho^{-1} k^{-1} \varphi (\rho) MN - \displaystyle\sum_{t = 0}^{\lfloor \varpi N \rfloor - 1} \displaystyle\sum_{I \in \mathcal{I}} \varphi \big( \rho (I; t) \big) \textbf{1}_{F (I; t)^c} \bigg| > C_{10} \varpi^2 k^{-1} N^2 \Bigg] < \displaystyle\frac{C_{11}}{\log N}, 
	\end{flalign*} 
	
	\noindent the corollary follows from \eqref{g1probabilityg2}, \eqref{fg0g1}, \eqref{fg0g12}, and \eqref{fg0g13}. 
\end{proof}

We can now establish \Cref{rhotksum2}. 

\begin{proof}[Proof of \Cref{rhotksum2}]
	
	Let $\varpi \in (0, 1)$; adopt the notation of \Cref{rhotksum}; assume that $\xi_0^{(\rho)}$ is coupled with the $\xi_0^{(\rho_j)}$ such that $\xi_0^{(\rho_i)} \le \xi_0^{(\rho)} \le \xi_0^{(\rho_j)}$ whenever $0 \le \rho_i \le \rho \le \rho_j \le 1$; and mutually couple $\xi_t^{(\rho)}$ with $\eta_t$ and the $\big( \xi_t^{(\rho_j)} \big)$ under the higher rank coupling (recall \Cref{etaxim}). For each interval $I \in \mathcal{J}$ and integer $t \in [0, \varpi N - 1]$, define the events
	\begin{flalign*}
	& D_0 (I; t) = \Big\{ \big| \rho (I; t) \textbf{1}_{F (I; t)^c}  - \rho \big| > \varpi^{1 / 4} \Big\}; \quad D_0 (t) = \left\{ \displaystyle\sum_{I \in \mathcal{J}} \textbf{1}_{D_0 (I; t)} > k^{-1} \varpi^{1 / 4} N \right\}; \\
	& D_0 = G_0 \cup \bigcap_{t = 0}^{\lfloor \varpi N \rfloor - 1} \big( D_0 (t) \cup F(t) \big),
	\end{flalign*}
	
	\noindent where we recall $G_0$ and $F(t)$ from \eqref{g0event} and \eqref{g2tevent}, respectively. Then \Cref{rhotksum} and the facts that $\mathbb{P} [G_0] < C e^{-N / C}$ and $\mathbb{P} [F] < C N^{-1 / 3}$ for some constant $C = C (b_1, b_2, \varpi) > 0$ (recall \Cref{eventfti} and below) together yield that $\mathbb{P} [D_0] \le C (\log N)^{-1}$, whenever $\varpi$ is sufficiently small (after increasing $C$ if necessary). 
	
	Restricting to $D_0^c$, there exists some integer $t = t_0 \in [0, \varpi N - 1]$ such that $D_0 (t)^c \cap F(t)^c \cap G_0^c$ holds. Then, define the events
	\begin{flalign*}
	& D_1 (I) = \Bigg\{ \bigg| \displaystyle\sum_{x \in I} \xi_{t_0}^{(\rho)} (x) - \rho k \bigg| > \varpi^{1 / 4} k \Bigg\}  \cup \Bigg\{ \textbf{1}_{D_0 (I; t_0)^c} \textbf{1}_{F (I; t_0)^c} \bigg| \displaystyle\sum_{x \in I} \xi_{t_0}^{(\rho (I; t_0))} (x) - \rho k \bigg| > 2 \varpi^{1 / 4} k \Bigg\} \\ 
	& \qquad \qquad \cup \Bigg\{ \textbf{1}_{D_0 (I; t_0)^c} \textbf{1}_{F (I; t_0)^c} \bigg| \displaystyle\sum_{x \in I} \xi_{t_0}^{(\rho (I; t_0) + 1	)} (x) - \rho k \bigg| > 2 \varpi^{1 / 4} k \Bigg\};  \\
	& D_1 = \left\{ \displaystyle\sum_{I \in \mathcal{J}} \textbf{1}_{D_1 (I)} > \varpi^{1 / 4} k^{-1} N \right\}. 
	\end{flalign*} 
	
	Then, a large deviations estimate for sums of independent Bernoulli $0-1$ random variables yields that $\mathbb{P} \big[ D_1 (I) \big] \le C (\log N)^{-2}$, after increasing $C$ if necessary, which implies by a Markov estimate that $\mathbb{P} [D_1] \le C (\log N)^{-1}$. Now, since $\xi_t^{(\rho)}$ and the $\xi_t^{(\rho_j)}$ are mutually ordered, and since $\xi_t^{(\rho (I; t))} |_I \le \eta_t |_I \le \xi_t^{(\rho (I; t) + 1)} |_I$ whenever $F(I; t)^c$ holds, we deduce for sufficiently small $\varpi$ that 
	\begin{flalign}
	\label{d0d1estimate}
	\begin{aligned}
	\textbf{1}_{D_0^c} \textbf{1}_{D_1^c} \displaystyle\sum_{x = -8R}^{8R} \big| \eta_{t_0} (x) - \xi_{t_0}^{(\rho)} (x) \big| & \le \textbf{1}_{D_0^c} \textbf{1}_{D_1^c} \displaystyle\sum_{I \in \mathcal{J}} \displaystyle\sum_{x \in I} \big| \eta_{t_0} (x) - \xi_{t_0}^{(\rho)} (x) \big| \\
	& \le \displaystyle\sum_{I \in \mathcal{J}} \displaystyle\sum_{x \in I} \big| \eta_{t_0} (x) - \xi_{t_0}^{(\rho)} (x) \big| \textbf{1}_{D_1 (I)^c} \textbf{1}_{D_0 (I; t_0)^c} + 2 \varpi^{1 / 4} N\\
	& \le 6 \varpi^{1 / 4} k |\mathcal{J}| + 2 \varpi^{1 / 4} N < \varsigma N.
	\end{aligned}
	\end{flalign}
	
	Now, from \Cref{xtxt1} and a large deviations estimate for sums of independent geometric random variables, we find that the probability of an uncoupled particle between $\big( \eta_{t_0}, \xi_{t_0}^{(\rho)} \big)$ outside of $[-8R, 8R]$ entering $[-4R, 4R]$ at some point during time interval $[t_0, Y]$ is bounded by $C e^{-N / C}$, after increasing $C$ if necessary. By \eqref{d0d1estimate}, it follows that 
	\begin{flalign*}
	\mathbb{P} \left[ \displaystyle\sum_{x = -4R}^{4R} \big| \eta_Y (x) - \xi_Y^{(\rho)} (x) \big| > \varsigma N \right] < \mathbb{P} [D_0] + \mathbb{P} [D_1] + C e^{-N / C}, 
	\end{flalign*}  	
	
	\noindent from which we deduce the proposition since $\mathbb{P} [D_0], \mathbb{P} [D_1] \le C (\log N)^{-1}$.  
\end{proof}

\end{document}